\documentclass[11pt, letterpaper,final]{amsart}
\usepackage{amsfonts,amsmath, amsthm, amssymb, latexsym, epsfig}
\usepackage[english]{babel}
\usepackage[utf8]{inputenc}
\usepackage[all]{xy}
\usepackage{xspace}
\usepackage{comment}
\usepackage{setspace}
\usepackage{enumerate}
\usepackage{stmaryrd}
\usepackage[mathscr]{eucal}
\usepackage[notcite,notref]{showkeys}

	\advance \topmargin by -\headheight
	\advance \topmargin by -\headsep

	\evensidemargin \oddsidemargin
	\newcommand{\Ext}{\ensuremath{\operatorname{Ext}}}
	\newcommand{\multialg}[1]{\mathcal{M}(#1)\xspace}
		
	\newcommand{\Prim}{\ensuremath{\operatorname{Prim}}}

	\renewcommand{\span}{\mathrm{span}}
	\newcommand{\Ad}{\ensuremath{\operatorname{Ad}}\xspace}
	\renewcommand{\min}{\mathrm{min}}

	\newcommand{\Hom}{\mathrm{Hom}}

	\newcommand{\cpct}{\Subset}
	\newcommand{\as}{\mathrm{as}}
	\newcommand{\wb}{{< \!\!\! <}}

	\theoremstyle{plain}
	\newtheorem{thm}{Theorem}[section]
	\newtheorem{lemma}[thm]{Lemma}
	\newtheorem{theorem}[thm]{Theorem}
	\newtheorem{proposition}[thm]{Proposition}
	\newtheorem{corollary}[thm]{Corollary}

	\newtheorem{theoremintro}{Theorem}

	\theoremstyle{definition}
	\newtheorem{definition}[thm]{Definition}
	
	\newtheorem{remark}[thm]{Remark}
	\newtheorem{notation}[thm]{Notation}

	\newtheorem{question}[thm]{Question}
	\newtheorem{observation}[thm]{Observation}
	
	\numberwithin{equation}{section}
	\numberwithin{figure}{section}

\begin{document}
	\title{A new proof of Kirchberg's $\mathcal O_2$-stable classification}
	\author{James Gabe}
	\address{Current address: School of Mathematics and Statistics, University of Glasgow, University Place, Glasgow G12 8SQ, Scotland}
        \address{Mathematical Sciences, University of Southampton, SO17 1BJ, United Kingdom}
        \email{jamiegabe123@hotmail.com}
        \subjclass[2010]{46L05, 46L35}
        \keywords{$\mathcal O_2$-stability, classification, ideal lattice of $C^\ast$-algebras}
	\thanks{This work was partially funded by the Carlsberg Foundation through an Internationalisation Fellowship.}

\begin{abstract}
I present a new proof of Kirchberg's $\mathcal O_2$-stable classification theorem: two separable, nuclear, stable/unital, $\mathcal O_2$-stable $C^\ast$-algebras are isomorphic if and only if their ideal lattices are order isomorphic, or equivalently, their primitive ideal spaces are homeomorphic.
Many intermediate results do not depend on pure infiniteness of any sort.
\end{abstract}

\maketitle

\section{Introduction}
After classifying all $A\mathbb T$-algebras of real rank zero \cite{Elliott-classrr0}, Elliott initiated a highly ambitious programme of classifying separable, nuclear $C^\ast$-algebras by $K$-theoretic and tracial invariants. During the past three decades much effort has been put into verifying such classification results. In the special case of simple $C^\ast$-algebras the classification has been verified under the very natural assumptions that the $C^\ast$-algebras are separable, unital, simple, with finite nuclear dimension and in the UCT class of Rosenberg and Schochet \cite{RosenbergSchochet-UCT}. The purely infinite case is due to Kirchberg \cite{Kirchberg-simple} and Phillips \cite{Phillips-classification}, and the stably finite case was recently solved by the work of many hands; in particular work of Elliott, Gong, Lin, and Niu \cite{GongLinNiu-classZ-stable}, \cite{ElliottGongLinNiu-classfindec}, and by Tikuisis, White, and Winter \cite{TikuisisWhiteWinter-QDnuc}. This makes this the right time to gain a deeper understanding of the classification of non-simple $C^\ast$-algebras which is the main topic of this paper.

The Cuntz algebra $\mathcal O_2$ plays a special role in the classification programme as this has the properties of being separable, nuclear, unital, simple, purely infinite, and is $KK$-equivalent to zero. Hence if $A$ is any separable, nuclear $C^\ast$-algebra then $A\otimes \mathcal O_2$ has no $K$-theoretic nor tracial data to determine potential classification, and one may ask if such $C^\ast$-algebras are classified by their primitive ideal space alone.

Predating the Kirchberg--Phillips theorem an important special case of this question was whether $A\otimes \mathcal O_2 \cong \mathcal O_2$ for any separable, nuclear, unital, simple $C^\ast$-algebra $A$. Some of the first major breakthroughs for verifying this were Elliott's (unpublished) proof that $\mathcal O_2 \otimes \mathcal O_2 \cong \mathcal O_2$, as well as Rørdam's characterisation of when $A \otimes \mathcal O_2 \cong \mathcal O_2$, see \cite{Rordam-O2tensorO2} or \cite[Theorem 7.2.2]{Rordam-book-classification}.

In Genève 1994 Kirchberg announced the $\mathcal O_2$-embedding theorem; that any separable exact $C^\ast$-algebra embeds into the Cuntz algebra $\mathcal O_2$. As an important consequence one gets that $A\otimes \mathcal O_2 \cong \mathcal O_2$ for any separable, nuclear, unital, simple $C^\ast$-algebra $A$. The $\mathcal O_2$-embedding theorem also played an important role in the proof of the Kirchberg--Phillips theorem \cite{Kirchberg-simple}, \cite{Phillips-classification}, i.e.~the classification of separable, nuclear, simple, purely infinite $C^\ast$-algebras. The first published proof of the $\mathcal O_2$-embedding theorem appeared in \cite{KirchbergPhillips-embedding}.

Kirchberg and Rørdam initiated the study of not necessarily simple (weakly/strongly) purely infinite $C^\ast$-algebras in \cite{KirchbergRordam-purelyinf} and \cite{KirchbergRordam-absorbingOinfty}. Such $C^\ast$-algebras arise from many natural constructions. For instance, there are natural characterisations of when crossed products \cite{KirchbergSierakowski-spicrossed}, groupoids \cite{BrownClarkSierakowski-purelyinfgroupoids}, and Fell bundles \cite{KwasniewskiSzymanski-pureinffell} are (strongly) purely infinite and it is an intriguing problem to classify them.

In \cite{Kirchberg-non-simple} Kirchberg outlined a far reaching generalisation of the Kirchberg--Phillips theorem by classifying all (not necessarily simple) separable, nuclear, strongly purely infinite $C^\ast$-algebras by ideal related $KK$-theory. A full proof of the result will appear in an upcoming book \cite{Kirchberg-book}. As an intermediate result Kirchberg obtains an ideal related version of the $\mathcal O_2$-embedding theorem from which the following elegant result follows: any two separable, nuclear, stable/unital, $\mathcal O_2$-stable $C^\ast$-algebras are isomorphic if and only if their primitive ideal spaces are homeomorphic. 

The goal of this paper is to present an almost self-contained proof this result which is also shorter and more elementary than the proof contained in the widely distributed (far from finished) version of Kirchberg's upcoming book \cite{Kirchberg-book}. At the moment a detailed proof of Kirchberg's $\mathcal O_2$-stable classification has never been published nor been publicly available, so this is the first published proof of the result.

The $\mathcal O_2$-embedding theorem was the cornerstone of the Kirchberg--Phillips theorem. In the same way, the ideal related $\mathcal O_2$-embedding theorem - the main intermediate result of this paper - plays a fundamental role in the proof of Kirchberg's classification of strongly purely infinite $C^\ast$-algebras.

In general it is hard to determine when a strongly purely infinite $C^\ast$-algebra is $\mathcal O_2$-stable. 
Dadarlat showed in \cite[Theorem 1.3]{Dadarlat-fiberwiseKK} that if $X$ is a compact, metrisable space with finite covering dimension, and $A$ is a separable, unital $C(X)$-algebra for which each fibre of $A$ is isomorphic to $\mathcal O_2$, then $A \cong C(X) \otimes \mathcal O_2$. However, Dadarlat also gave examples \cite[Example 1.2]{Dadarlat-fiberwiseKK} of separable, unital $C(X)$-algebras $A$ for which $X$ is the Hilbert cube such that each fibre is isomorphic to $\mathcal O_2$, but for which $A \otimes \mathcal O_2 \not \cong A$. Inspired by these results I showed in \cite{Gabe-cplifting}, using Kirchberg's non-simple classification, that a separable, nuclear, strongly purely infinite $C^\ast$-algebra $A$ is $\mathcal O_2$-stable if and only if each two-sided, closed ideal in $A$ is $KK$-equivalent to zero.

\subsection{The main results}
In order to present the main results of the paper I introduce some notation.
The ideal lattice of a $C^\ast$-algebra $A$ is denoted by $\mathcal I(A)$, and whenever $A$ is separable then $\mathcal I(A)$ is considered as an object in the category of abstract Cuntz semigroups. A $Cu$-morphism $\mathcal I(A) \to \mathcal I(B)$ is a map that preserves the Cuntz semigroup structure. It is shown in Lemma \ref{l:idealmorphisms} that any $\ast$-homomorphism $\phi \colon A \to B$ induces a $Cu$-morphism $\mathcal I(\phi) \colon \mathcal I(A) \to \mathcal I(B)$ via the assignment $I \mapsto \overline{B\phi(I)B}$. More details will be given in Section \ref{s:I}.

The first main theorem, which is a special case of Corollary \ref{c:classhom}, gives a complete classification of $\ast$-homomorphisms between $A$ and $B \otimes \mathcal O_2\otimes \mathbb K$ whenever $A$ is separable, exact and $B$ is separable, nuclear.

\begin{theoremintro}
Let $A$ and $B$ be separable $C^\ast$-algebras with $A$ exact and $B$ nuclear. Then the approximate unitary equivalence classes of $\ast$-homomorphisms $A \to B \otimes \mathcal O_2 \otimes \mathbb K$ are in a natural one-to-one correspondence with $Cu$-morphisms $\mathcal I(A) \to \mathcal I(B)$ between the ideal lattices of $A$ and $B$.
\end{theoremintro}

As an almost immediate consequence the following classification result due to Kirchberg \cite{Kirchberg-non-simple} is obtained. A slightly more general result is provided in Theorem \ref{t:O2class}. Note that the classification is strong, i.e.~that any isomorphism on the invariant is induced by an isomorphism of the $C^\ast$-algebras.

\begin{theoremintro}
Let $A$ and $B$ be separable, nuclear $C^\ast$-algebras for which $A \cong A \otimes \mathcal O_2$ and $B \cong B \otimes \mathcal O_2$. Suppose that $A$ and $B$ are both stable or both unital, and that $f \colon \Prim A \to \Prim B$ is a homeomorphism. Then there exists an isomorphism $\phi \colon A \xrightarrow \cong B$ such that $f(I) = \phi(I)$ for every $I \in \Prim A$. 
\end{theoremintro}

The proofs are based mainly on elementary or well-known results from the literature with the main exception being a deep structural result due to Kirchberg and Rørdam \cite[Theorem 6.11]{KirchbergRordam-zero}, which shows that separable, nuclear, $\mathcal O_2$-stable $C^\ast$-algebras contain suitably well-behaved, commutative $C^\ast$-subalgebras. This will be used in the proof of Proposition \ref{p:Wliftcp}. 

The overall strategy of the proof is very classical: I provide an existence and a uniqueness result for $\ast$-homomorphisms into $\mathcal O_2$-stable $C^\ast$-algebras, using the ideal lattice of $C^\ast$-algebras as the classifying invariant. The ideal lattice is considered as a \emph{covariant} functor with target category being the category $\mathbf{Cu}$\footnote{At least when $A$ is separable, cf.~Proposition \ref{p:Csg}.} of abstract Cuntz semigroups, as first introduced in \cite{CowardElliottIvanescu-Cuntzsemigroupinv}. This is in contrast to Kirchberg's approach which is to consider $C^\ast$-algebras with actions of topological spaces, an approach which is of a more contravariant nature. My approach allows for the use of compact containment of ideals, i.e.~way-below in the Cuntz semigroup sense. Compact containment will play a crucial role, see Proposition \ref{p:idealultra}.  I will at points digress slightly from the main objective in order to present a more well-rounded theory applicable in a more general setting. For instance, certain results are proved for order zero maps, while the statements are only needed for $\ast$-homomorphisms. 

\subsection*{Outline of the paper}

The basic properties of ideal lattices are studied in Section \ref{s:I}. 

In Section \ref{s:HB}, uniqueness results for nuclear maps are presented. The first uniqueness result Theorem \ref{t:HB} is the key ingredient. It shows that nuclear $\ast$-homo\-morphisms that define the same map between ideal lattices will approximately dominate each other. I then introduce an equivalence relation on $\ast$-homomorphisms weaker than approximate and asymptotic unitary equivalence, which I call approximate and asymp\-totic Murray--von Neumann equivalence. A key feature of this equivalence relation is that it is exactly the equivalence relation for which a certain $2\times 2$-matrix trick of Connes \cite{Connes-classtypeIII} is applicable, see Proposition \ref{p:MvNeq}. Similar matrix tricks have recently appeared in the classification programme, notably in the work of Matui and Sato \cite{MatuiSato-decrankUHF} and by Bosa et al.~\cite{BBSTWW-2coloured}. This matrix trick is used to prove Theorem \ref{t:O2HB} which is a uniqueness result for so-called (strongly) $\mathcal O_2$-stable $\ast$-homomorphisms, showing that the approximate (resp.~asymptotic) Murray--von Neumann equivalence class only depends on the induced maps between the ideal lattices.

In Section \ref{s:infty}, a characterisation is given of when an approximate $\ast$-homomorphism (i.e.~a $\ast$-homomorphism going into a sequence algebra) is approximately unitary equivalent to a $\ast$-homomorphism represented point-wise by constant sequences. This technique is a discrete version of \cite[Proposition 1.3.7]{Phillips-classification}. Here no assumptions are put on our $C^\ast$-algebras.

With a uniqueness result in the utility belt I move on to study existence results for maps between ideal lattices (Section \ref{s:W}). This is done by using Michael's selection theorem to produce well-behaved c.p.~maps into commutative $C^\ast$-algebras, a method inspired by work of Blanchard \cite{Blanchard-deformations}. Similar results using the same method have previously been obtained by Harnisch and Kirchberg \cite{HarnischKirchberg-primitive}. A lot of these results are very general and have (almost) no requirements on the $C^\ast$-algebras involved, e.g.~pure infiniteness type criteria are never assumed. 

Finally, in Section \ref{s:O2} I prove the main existence theorem which is an ideal related version of the $\mathcal O_2$-embedding theorem. Kirchberg's original proof uses (non-unital, ideal related versions of) $C^\ast$-systems as introduced in \cite{Kirchberg-normalizer}. The proof presented here only uses elementary $C^\ast$-algebraic techniques inspired by (though still quite different from) the proof of the $\mathcal O_2$-embedding theorem presented by Kirchberg and Phillips in \cite{KirchbergPhillips-embedding}.
As a consequence I obtain Kirchberg's classification of $\mathcal O_2$-stable $C^\ast$-algebras.

\subsection*{Acknowledgement}
Parts of this paper were completed during a research visit at the Mittag--Leffler Institute during the programme Classification of Operator Algebras: Complexity, Rigidity, and Dynamics, and during a visit at the CRM institute during the programme IRP Operator Algebras: Dynamics and Interactions. I am thankful for their hospitality during these visits.

I am very grateful to Joan Bosa, Jorge Castillejos, Aidan Sims and Stuart White for valuable, inspiring conversations on topics of the paper. I would also like to thank the referee for many helpful comments and suggestions.

\section{The basics}\label{s:I}

  Throughout the entire paper \emph{every ideal is assumed to be two-sided and closed.} 

\subsection{Compact containment}

A very basic and somewhat overlooked property of ideal lattices of $C^\ast$-algebras is the \emph{compact containment} relation which is exactly the way-below relation in complete lattices. This relation plays a crucial role in the study of ideal lattices. My motivation for considering this relation for ideals is inspired work on the Cuntz semigroup of $C^\ast$-algebras. See Proposition \ref{p:Csg} for the connection to Cuntz semigroups.

\begin{definition}
Let $ I$ and $ J$ be ideals in a $C^\ast$-algebra $ A$. Say that \emph{$ I$ is compactly contained in $ J$}, written $I \cpct J$, if whenever $(I_\lambda)_{\lambda \in \Lambda}$ is a family of ideals in $ A$ such that $ J \subseteq \overline{\sum_{\lambda}  I_\lambda}$, then there are finitely many $\lambda_1,\dots,\lambda_n\in \Lambda$ such that $ I \subseteq \sum_{k=1}^n  I_{\lambda_n}$.
\end{definition}

Note that in the definition above if the family $(I_\lambda)$ is upwards directed then it is always possible to find a single $\lambda$ such that $I \subseteq I_\lambda$.

Given a positive element $a\in A$ and an $\epsilon>0$, $(a-\epsilon)_+:= f_\epsilon(a)$ denotes the element defined by functional calculus where $f_\epsilon (t) = \max\{ 0 , t-\epsilon\}$.

\begin{lemma}\label{l:cpct}
Let $A$ be a $C^\ast$-algebra and let $I,J$ be ideals in $A$. Then $I \cpct J$ if and only if there is a positive element $a\in J$ and an $\epsilon >0$ such that $I \subseteq \overline{A(a-\epsilon)_+ A}$.

In particular, $\overline{A(a-\epsilon)_+ A} \cpct \overline{AaA}$ for any positive $a\in A$ and $\epsilon>0$.
\end{lemma}
\begin{proof}
``In particular'' clearly follows from the ``if'' statement. For the ``if'' part, suppose there are $a\in J_+$, and $\epsilon>0$ such that $I \subseteq \overline{A(a-\epsilon)_+A}$, and let $(I_\lambda)$ be a family of ideals in $A$ such that $J \subseteq \overline{\sum I_\lambda}$. There are $\lambda_1,\dots,\lambda_n$ and a positive element $b\in \sum_{k=1}^n I_{\lambda_k}$ such that $\| a - b \| < \epsilon/2$. Note that $a - \tfrac{\epsilon}{2} \cdot 1_{\tilde A} \leq b$. Let $f\colon \mathbb R_+ \to \mathbb R_+$ be any continuous function which is $0$ on $[0,\epsilon/2]$, and $1$ on $[\epsilon, \infty)$. Then
\[
(a-\epsilon)_+ \leq f(a) (a- \tfrac{\epsilon}{2}\cdot 1_{\tilde A}) f(a) \leq f(a) b f(a).
\]
Thus $I \subseteq \overline{A(a-\epsilon)_+A} \subseteq \overline{AbA} \subseteq \sum_{k=1}^n I_{\lambda_k}$.

For the ``only if'' part, suppose that $I\cpct J$. The family of ideals $(\overline{A(a-\epsilon)_+ A})_{a\in J_+,\epsilon>0}$ is upwards directed. In fact, if $a_1,a_2\in J_+$, $\epsilon_1,\epsilon_2 >0$ let $a= a_1+a_2$. By the ``in particular''-part, $\overline{A(a_i -\epsilon_i)_+ A} \cpct \overline{AaA} = \overline{\bigcup_{\epsilon >0} \overline{A(a-\epsilon)_+A} }$ so there are $\epsilon'_i >0$ such that $\overline{A(a_i - \epsilon_i)_+ A} \subseteq \overline{A(a-\epsilon_i')_+A} \subseteq \overline{A(a-\epsilon)_+ A}$ where $\epsilon = \min(\epsilon_1' , \epsilon_2')$. Hence the family of ideals is upwards directed. So as $J = \overline{\sum_{a\in J_+,\epsilon>0} \overline{A(a-\epsilon)_+A}}$ it follows from compact containment that $I \subseteq \overline{A(a-\epsilon)_+ A}$ for some $a\in J_+$ and $\epsilon>0$.
\end{proof}

The above lemma implies that whether or not an ideal is generated by a single element, i.e.~contains a full element, can be determined from the ideal lattice.

\begin{corollary}\label{c:Cuntzsglemma}
Let $A$ be a $C^\ast$-algebra, and $I$ be an ideal in $A$. Then $I$ has a full element if and only if there is a sequence $I_1 \cpct I_2 \cpct \dots$ of ideals such that $I = \overline{\bigcup I_n}$.
\end{corollary}
\begin{proof}
If $a\in I$ is full then $I_n := \overline{A(a^\ast a-1/n)_+A}$ gives the desired sequence. Conversely, suppose the sequence of ideals is given. By Lemma \ref{l:cpct} we find positive contractions $a_n \in A$ such that $I_n \subseteq \overline{Aa_nA} \subseteq I_{n+1}$. It is easy to see that $\sum_{n=1}^\infty 2^{-n} a_n$ is full in $I$.
\end{proof}


\subsection{The ideal lattice}

\begin{notation}
Let $\mathcal I(A)$ denote the \emph{ideal lattice} of a $C^\ast$-algebra $A$. 

Consider $\mathcal I(A)$ as an ordered, abelian monoid (addition is addition of ideals). It has increasing suprema (closure of unions), and compact containment.
\end{notation}

Let $\mathbf{Cu}$ denote the category of abstract Cuntz semigroups, i.e.~the category of ordered, abelian monoids with countable increasing suprema and way-below such that these are suitably well-behaved. See \cite[Page 170]{CowardElliottIvanescu-Cuntzsemigroupinv} or \cite[Definition 4.1]{AraPereraToms-K-theoryclass} for the precise definition. 

One particular part of the definition which will be used below, is that for any element $x$ in an abstract Cuntz semigroup there is a sequence $x_1 \wb x_2 \wb \dots$ such that $\sup x_n = x$.

\begin{proposition}\label{p:Csg}
Let $A$ be a $C^\ast$-algebra. Then $\mathcal I(A)$ is an object in $\mathbf{Cu}$ if and only if any ideal in $A$ contains a full element. In particular, $\mathcal I(A)$ is in $\mathbf{Cu}$ whenever $A$ is separable.
\end{proposition}
\begin{proof}
If $\mathcal I(A)$ is an object of $\mathbf{Cu}$, then (by definition) for any $I \in \mathcal I(A)$ there is a sequence $I_1 \cpct I_2 \cpct\dots$ such that $I = \overline{\bigcup I_n}$. By Corollary \ref{c:Cuntzsglemma}, $I$ contains a full element. 

Conversely, suppose any ideal in $A$ has a full element. As $\mathcal O_\infty \otimes \mathbb K$ is simple and nuclear there is a canonical order isomorphism $\mathcal I(A) \cong \mathcal I(A \otimes \mathcal O_\infty \otimes \mathbb K)$ given by $I \mapsto I \otimes \mathcal O_\infty \otimes \mathbb K$. Note that any ideal in $A \otimes \mathcal O_\infty \otimes \mathbb K$ also contains a full element.

As $A \otimes \mathcal O_\infty \otimes \mathbb K$ is purely infinite by \cite[Theorem 5.11]{KirchbergRordam-purelyinf}, for any two positive elements $a,b$ in $A \otimes \mathcal O_\infty \otimes \mathbb K$ we have $[a] \leq [b]$ in $Cu(A \otimes \mathcal O_\infty \otimes \mathbb K)$ if and only if $a$ is contained in the ideal generated by $b$. Hence the canonical map $Cu(A \otimes \mathcal O_\infty \otimes  \mathbb K) \to \mathcal I(A \otimes \mathcal O_\infty \otimes \mathbb K)$ is an order isomorphism onto its image. Moreover, as every ideal in $A\otimes \mathcal O_\infty \otimes \mathbb K$ has a full element this map is onto so $Cu(A \otimes \mathcal O_\infty \otimes \mathbb K) \xrightarrow \cong \mathcal I(A \otimes \mathcal O_\infty \otimes \mathbb K)$.

It is easy to see (by pure infiniteness) that addition in $Cu(A \otimes \mathcal O_\infty \otimes \mathbb K)$ is uniquely determined by the order. As $\sup$ and $\wb$ are also uniquely determined by the order it follows that the isomorphism $Cu(A \otimes \mathcal O_\infty \otimes \mathbb K) \xrightarrow \cong \mathcal I(A \otimes \mathcal O_\infty \otimes \mathbb K)$ preserves $(0,+,\sup, \wb)$. As $Cu(A \otimes \mathcal O_\infty \mathbb K)$ is an object of $\mathbf{Cu}$, so is $\mathcal I(A) \cong \mathcal I(A \otimes \mathcal O_\infty \otimes \mathbb K)$. 
\end{proof}

One could alternatively also prove the ``if''-statement above simply by checking that $\mathcal I(A)$ satisfies the defining criteria for objects in $\mathbf{Cu}$ instead of using that $Cu(A\otimes \mathcal O_\infty \otimes \mathbb K)$ is an object in $\mathbf{Cu}$. This is straightforward but tedious.

The above proposition motivates the following definition.

\begin{definition}\label{d:Cumorphism}
A map $\Phi \colon \mathcal I(A) \to \mathcal I(B)$ is a \emph{generalised $Cu$-morphism} if it is an ordered monoid homomorphism which preserves increasing suprema. Say that $\Phi$ is a \emph{$Cu$-morphism} if in addition it preserves compact containment.
\end{definition}

\begin{remark}
In the category $\mathbf{Cu}$ a generalised $Cu$-morphisms is an ordered monoid homomorphism which preserves \emph{countable} increasing suprema.
A $Cu$-morphism is a generalised $Cu$-morphism which preserves compact containment.

The criterion that any $x$ in a $Cu$-semigroup is the supremum of a sequence $x_1 \wb x_2 \wb \dots$ implies that whenever a (not necessarily countable) increasing net has a supremum then any (generalised) $Cu$-morphism preserves this supremum. So (generalised) $Cu$-morphisms in $\mathbf{Cu}$ preserve (not necessarily countable) increasing suprema whenever these exist. 

Hence whenever $\mathcal I(A)$ and $\mathcal I(B)$ are objects in $\mathbf{Cu}$, see Proposition \ref{p:Csg}, a map $\Phi \colon \mathcal I(A) \to \mathcal I(B)$ is a (generalised) $Cu$-morphism in the sense of Definition \ref{d:Cumorphism} if and only if it is in the sense of $\mathbf{Cu}$.
\end{remark}

\begin{remark}\label{r:sup}
A map $\Phi \colon \mathcal I(A) \to \mathcal I(B)$ is a generalised $Cu$-morphism if and only if it preserves suprema (possibly empty, and not necessarily increasing). 

In fact, as $\sup \emptyset = 0$ and as the supremum of finitely many ideals is the sum, it follows that any supremum preserving map is an ordered monoid homomorphism and thus a generalised $Cu$-morphism. Conversely, suppose $\Phi$ is a generalised $Cu$-morphism. Then $\Phi(\sup \emptyset) = \Phi(0) = 0 = \sup \emptyset$ so $\Phi$ preserves the supremum of $\emptyset$. Moreover, if $S \subseteq \mathcal I(A)$ is non-empty, let $\sum S$ denote the set $\{ \sum_{I\in S'} I : S' \subseteq S \textrm{ is finite}\}$. Clearly $\sup \sum S = \sup S$ and as $\sum S$ is upwards directed it follows that $\Phi(\sup S) = \Phi(\sup \sum S) = \sup \Phi(\sum S) = \sup \Phi(S)$ so $\Phi$ preserves the supremum of $S$.
\end{remark}

\begin{notation}
If $\phi \colon A \to B$ is a completely positive (c.p.) map then $\mathcal I(\phi) \colon \mathcal I(A) \to \mathcal I(B)$ is the map $\mathcal I(\phi)(I) = \overline{B\phi(I)B}$ for $I \in \mathcal I(A)$.
\end{notation}

\begin{remark}
Given a $\ast$-homomorphism $\phi \colon A \to B$ it is common to consider the induced map on ideal lattices $\phi^\ast \colon \mathcal I(B) \to \mathcal I(A)$ given by $\phi^\ast(I) = \phi^{-1}(I)$. In this way the ideal lattice is a \emph{contravariant} functor. I emphasise that this is \emph{not} the same approach as used in this paper where the ideal lattice is considered a \emph{covariant} functor.
\end{remark}

\begin{remark}
$\mathcal I(\phi)$ is defined for any c.p.~map $\phi$. However, $\mathcal I(-)$ is \emph{not} functorial on the category of $C^\ast$-algebras with c.p.~maps as morphisms. For instance, if $\phi \colon \mathbb C \to M_2(\mathbb C)$ is the embedding into the $(1,1)$-corner, and $\psi \colon M_2(\mathbb C) \to \mathbb C$ is the compression to the $(2,2)$-corner, then $\psi \circ \phi = 0$, but $\mathcal I(\psi) \circ \mathcal I(\phi) = id_{\mathcal I(\mathbb C)} \neq \mathcal I(0) = \mathcal I(\psi \circ \phi)$.

However, $\mathcal I(-)$ \emph{is} functorial on the category of $C^\ast$-algebras with c.p.~\emph{order zero} maps as morphisms, cf.~Proposition \ref{p:ozfunctorial} below.
\end{remark}

\begin{lemma}\label{l:idealmorphisms}
Let $\phi \colon A \to B$ be a c.p.~map. The following hold.
\begin{itemize}
\item[$(i)$] $\mathcal I(\phi)(I) = \overline{B\phi(I_+)B}$ for all $I\in \mathcal I(A)$.
\item[$(ii)$] $\mathcal I(\phi)$ is a generalised $Cu$-morphism.
\item[$(iii)$] $\mathcal I(\phi)$ is a $Cu$-morphism whenever $\phi$ is a $\ast$-homomorphism.
\end{itemize}
\end{lemma}
\begin{proof}
$(i)$: As $\phi(a)^\ast \phi(a) \leq \phi(a^\ast a)$ it easily follows that $\mathcal I(\phi)(I) = \overline{B\phi(I_+) B}$.

$(ii)$: Clearly $\mathcal I(\phi)$ preserves zero and order. As $(I+J)_+ = I_+ + J_+$ for ideals $I,J$ in $A$, it follows from $(i)$ that
\[
\mathcal I(\phi)(I+J) = \overline{B\phi(I_+ + J_+) B} = \overline{B(\phi(I_+) \cup \phi(J_+)) B} = \mathcal I(\phi)(I) + \mathcal I(\phi)(J).
\]
So $\mathcal I(\phi)$ is an ordered monoid homomorphism. Let $(I_\lambda)$ be an increasing net of ideals in $A$, and let $I = \overline{\bigcup I_\lambda}$. Clearly $\overline{\bigcup \mathcal I(\phi)(I_\lambda)} \subseteq \mathcal I(\phi)(I)$ as $\mathcal I(\phi)(I_\lambda) \subseteq \mathcal I(\phi)(I)$ for each $\lambda$. To show the other inclusion, let $x\in I$ and pick a sequence $(x_n)$ in $\bigcup I_\lambda$ so that $\| x_n - x\| \to 0$. Then $\| \phi(x_n) - \phi(x)\| \to 0$, and as $\phi(x_n) \in \bigcup \mathcal I(\phi)(I_\lambda)$ it follows that $\phi(x) \in \overline{\bigcup\mathcal I(\phi)(I_\lambda)}$. Hence $\mathcal I(\phi)$ preserves increasing suprema and is thus a generalised $Cu$-morphism.

$(iii)$: If $\phi$ is a $\ast$-homomorphism we want to see that $\mathcal I(\phi)$ preserves compact containment, so suppose that $I \cpct J$. Note that $\mathcal I(\phi)(\overline{AbA}) = \overline{B \phi(b) B}$ for any $b\in A_+$ as $\phi$ is multiplicative. By Lemma \ref{l:cpct} there are $a\in J_+$ and $\epsilon>0$ such that $I \subseteq \overline{A(a-\epsilon)_+A}$. As 
\[
\mathcal I(\phi)(I) \subseteq \mathcal I(\phi)(\overline{A(a-\epsilon)_+A}) = \overline{B (\phi(a)-\epsilon)_+ B} \cpct \overline{B\phi(a) B} = \mathcal I(\phi)(\overline{AaA}) \subseteq \mathcal I(\phi)(J)
\]
it follows that $\mathcal I(\phi)$ preserves compact containment.
\end{proof}


\subsection{On functoriality of $\mathcal I(-)$}

It is shown that $\mathcal I(-)$ is functorial on the category of $C^\ast$-algebras with c.p.~order zero maps as morphisms. Although the generality of order zero maps is not actually needed in this paper, I believe that this level of generality could be applicable in other contexts. A few additional lemmas on functoriality of $\mathcal I(-)$ is also provided which will be used in Section \ref{s:W}.

Recall that a c.p.~map $\phi \colon A\to B$ is called \emph{order zero} if $\phi(a)\phi(b) = 0$ whenever $a,b\in A_+$ are such that $ab=0$.

\begin{remark}[Basics of order zero maps]\label{r:WZ}
Suppose $\phi \colon A \to B$ is a c.p.~order zero map. It follows from \cite[Theorem 3.3]{WinterZacharias-orderzero} that if one lets $C= C^\ast(\phi(A))$, there is a positive element $h \in \multialg{C}\cap \phi(A)'$ and a $\ast$-homomorphism $\pi \colon A \to \multialg{C} \cap \{ h\}'$ such that $\phi = h \pi(-)$. It is easy to see (using \cite[Proposition 3.2]{WinterZacharias-orderzero} for the case when $A$ is non-unital) that the pair $(\pi,h)$ with this property is unique.

This is used to do \emph{functional calculus} of order zero maps c.f.~\cite[Corollary 4.2]{WinterZacharias-orderzero}. In fact, for $f\colon [0,\infty) \to [0, \infty)$ a continuous map for which $f(0) = 0$, one defines $f(\phi)(-) := f(h) \pi(-) \colon A \to B$ which is also a c.p.~order zero map. 

This clearly implies that c.p.~order zero maps have the following bi-module type property: for any $a,b,c\in A$
\begin{equation}\label{eq:bimodule}
\phi(abc) = \lim_{k\to \infty} \phi^{1/k}(a) \phi(b) \phi^{1/k}(c).
\end{equation}
\end{remark}

\begin{lemma}\label{l:ozideals}
Let $\phi \colon A \to B$ be a c.p.~order zero map. For any $a,b,c\in A$, $\phi(abc) \in \overline{B \phi(b) B}$.
\end{lemma}
\begin{proof}
This follows from equation \eqref{eq:bimodule}.
\end{proof}

\begin{proposition}\label{p:ozfunctorial}
Let $A,B$ and $C$ be $C^\ast$-algebras, and let $\phi \colon A \to C$ and $\psi \colon C \to B$ be c.p.~maps. If $\psi$ is order zero, then $\mathcal I(\psi \circ \phi) = \mathcal I(\psi) \circ \mathcal I(\phi)$.
\end{proposition}
\begin{proof}
Let $I\in \mathcal I(A)$. Then
\begin{equation}\label{eq:idealcompose}
\mathcal I(\psi \circ \phi)(I) = \overline{B \psi(\phi(I)) B} \subseteq \overline{B \psi(\mathcal I(\phi)(I)) B} = \mathcal I(\psi) \circ \mathcal I(\phi) (I).
\end{equation}
To show the other inclusion, it suffices by Lemma \ref{l:idealmorphisms}$(i)$ to show that $\psi((\mathcal I(\phi)(I))_+) \subseteq  \overline{B \psi(\phi(I)) B}$. Fix a positive element $x\in \mathcal I(\phi)(I)$. As $\mathcal I(\phi)(I)$ is generated by $\phi(I_+)$, we may for any $\epsilon >0$ find $c_1,\dots, c_{n} \in C$ and $y_1, \dots, y_n\in I_+$ such that $(x-\epsilon)_+ = \sum_{i=1}^n c_i^\ast \phi(y_i) c_i$. Then
\[
\psi((x-\epsilon)_+) = \sum_{i=1}^n \psi(c_i^\ast \phi(y_i) c_i) \stackrel{\textrm{Lem.~}\ref{l:ozideals}}{\in} \sum_{i=1}^n \overline{B \psi(\phi(y_i))B} \subseteq \overline{B \psi(\phi(I)) B}.
\]
Hence $\psi(x) \in \overline{B \psi(\phi(I)) B}$.
\end{proof}

The following two lemmas about when $\mathcal I(-)$ preserves composition will be used in Section \ref{s:W}.

\begin{lemma}\label{l:idealcom}
Let $A,B$ and $C$ be $C^\ast$-algebras with $C$ commutative, and let $\phi \colon A \to C$ and $\psi \colon C \to B$ be c.p.~maps. Then $\mathcal I(\psi) \circ \mathcal I(\phi) = \mathcal I(\psi \circ \phi)$.
\end{lemma}
\begin{proof}
Let $I\in \mathcal I(A)$. By \eqref{eq:idealcompose}, $\mathcal I(\psi \circ \phi)(I) \subseteq \mathcal I(\psi) \circ \mathcal I(\phi)(I)$ so it remains to prove the other implication. For this it is enough to prove that $\psi((\mathcal I(\phi)(I))_+) \subseteq \mathcal I(\psi \circ \phi)(I)$. To see this it suffices to show that $\psi((f-\epsilon)_+) \in \mathcal I(\psi \circ \phi)(I)$ for every positive $f\in \mathcal I(\phi)(I)$ and every $\epsilon >0$.

As $\mathcal I(\phi)(I) = \overline{C\phi(I_+)C}$, and as $C$ is commutative, we may pick $a_1,\dots, a_n \in I_+$ and $g_1,\dots, g_n \in C_+$ such that $\| f - \sum_{k=1}^n g_k \phi(a_k)\| < \epsilon$. As $C$ is commutative, this implies that 
\[
(f-\epsilon)_+ \leq \sum_{k=1}^n g_k \phi(a_k) \leq \sum_{k=1}^n \| g_k\| \phi(a_k).
\]
Thus
\[
\psi((f-\epsilon)_+) \leq \sum_{k=1}^n \| g_k\| \psi(\phi(a_k)) \in \sum_{k=1}^n \overline{B\psi(\phi(a_k)) B} \subseteq \mathcal I(\psi \circ \phi) (I). \qedhere
\]
\end{proof}

Let $\otimes$ denote the minimal tensor product. If $B$ and $D$ are $C^\ast$-algebras, and $\eta$ is a state on $D$, $\lambda_\eta \colon B \otimes D \to B$ denotes the induced (left) slice map given on elementary tensors by $\lambda_\eta(b\otimes d) = b \eta(d)$.

\begin{lemma}\label{l:slice}
Let $D$ be an exact $C^\ast$-algebra, and suppose that $\eta$ is a faithful state on $D$. For any c.p.~map $\phi \colon A \to B\otimes D$ we have $\mathcal I(\lambda_\eta \circ \phi) = \mathcal I(\lambda_\eta) \circ \mathcal I(\phi)$.
\end{lemma}
\begin{proof}
First note that $\lambda_\eta$ is faithful: if $x\in B\otimes D$ is positive and non-zero, then there are positive linear functionals $\mu_1$ on $B$ and $\mu_2$ on $D$ such that $\mu_2(\rho_{\mu_1}(x))=(\mu_1\otimes\mu_2) (x) >0$\footnote{If $\pi_1 \colon B \to \mathbb B(\mathcal H_1), \pi_2 \colon D \to \mathbb B(\mathcal H_2)$ are faithful representations, then $\pi_1 \times \pi_2 \colon B \otimes D \to \mathbb B(\mathcal H_1 \otimes \mathcal H_2)$ is faithful (and well-defined, by the definition of minimal tensor products). Hence there are $\xi_i \in \mathcal H_i$ such that $(\pi_1\otimes \pi_2)(x) (\xi_1 \otimes \xi_2) \neq 0$. Letting $\mu_i = \langle \pi_i(-)\xi_i, \xi_i\rangle$ does the trick.} where $\rho_{\mu_1} \colon B \otimes D \to D$ is the induced right slice map. Hence $(\rho_{\mu_1})(x)\in D$ is positive and non-zero so $\rho_1( \lambda_\eta(x)) = \eta(\rho_{\mu_1}(x)) > 0$, and thus $\lambda_\eta(x) \neq 0$.

Let $x\in B\otimes D$ be positive and $J\in \mathcal I(B)$. We claim that $x\in J \otimes D$ if and only if $\lambda_\eta(x) \in J$. ``Only if'' is trivial, so for ``if'' we assume that $\lambda_\eta(x) \in J$. As $D$ is exact, $(B \otimes D)/(J\otimes D) = (B/J) \otimes D$, so $\lambda_\eta$ descends to a faithful slice map $\overline \lambda_\eta \colon (B/J) \otimes D \to B/J$. Thus $\overline \lambda_\eta(x + J \otimes D) = \lambda_\eta(x) + J = 0$. So $x+J\otimes D=0$, and hence $x\in J\otimes D$.

So for any subset $S\subseteq B\otimes D$ of positive elements, $\lambda_\eta(S)$ generates the ideal $J$ if and only if $J\otimes D = \bigcap_{J_0 \in \mathcal I(B): S \subseteq J_0 \otimes D} J_0 \otimes D$. Let $I\in \mathcal I(A)$, $S:= \phi(I_+)$ and $J = \overline{B \lambda_\eta(S) B}$. Then $\mathcal I(\phi)(I) \subseteq J \otimes D$ so
\[
(\mathcal I(\lambda_\eta) \circ \mathcal I(\phi))(I) \subseteq \mathcal I(\lambda_\eta)(J \otimes D) = J = \mathcal I(\lambda_\eta \circ \phi)(I).
\]
The implication $\mathcal I(\lambda_\eta \circ \phi)(I) \subseteq (\mathcal I(\lambda_\eta) \circ \mathcal I(\phi))(I)$ is easy, see \eqref{eq:idealcompose}.
\end{proof}


\section{Uniqueness of nuclear maps via ideals}\label{s:HB}

\subsection{Approximate domination}

In this section several uniqueness results for nuclear maps are presented. The main result of this subsection, Theorem \ref{t:HB}, is the key ingredient in all of these results, and it characterises when nuclear order zero maps approximately dominate each other in terms of their behaviour on ideals. 

Although the results in this paper only need the results in the generality of $\ast$-homo\-morphisms, and not order zero maps, I believe that the generality of order zero maps could potentially be applicable in other contexts.

\begin{definition}\label{d:approxdom}
Let $\phi,\psi \colon A \to B$ be c.p.~maps and suppose that $\phi$ is order zero. Say that $\phi$ \emph{approximately dominates} $\psi$ if for any finite subset $\mathcal F \subset A$ and any $\epsilon >0$, there are an $n\in \mathbb N$ and $b_1,\dots,b_n \in B$ such that
\[
\| \psi(a) - \sum_{k=1}^n b_k^\ast \phi(a) b_k\| < \epsilon, \qquad a \in \mathcal F.
\]
\end{definition}

Recall that a map is called \emph{nuclear} if it is a point-norm limit of maps factoring by c.p.~maps through matrix algebras. Note that nuclear maps are c.p.~by definition. 

Clearly any c.p.~map which is approximately dominated by a nuclear map is itself nuclear.

\begin{notation}
For generalised $Cu$-morphisms $\Phi,\Psi \colon \mathcal I(A) \to \mathcal I(B)$, write $\Phi \leq \Psi$ if $\Phi(I) \subseteq \Psi(I)$ for all $I\in \mathcal I(A)$.
\end{notation}

The following is a slight modification of similar results appearing in work of Kirchberg and Rørdam \cite[Lemma 7.18]{KirchbergRordam-absorbingOinfty}, \cite[Proposition 4.2]{KirchbergRordam-zero}, and the proof presented here is virtually identical to theirs. A similar result also appeared in \cite[Theorem 2.5]{Gabe-cplifting}. The true strength of the result is that condition $(i)$ is a \emph{uniform on compact sets} type condition whereas condition $(iii)$ is a \emph{point-wise} condition.

\begin{theorem}\label{t:HB}
Let $A$ and $B$ be $C^\ast$-algebras with $A$ exact, let $\phi, \rho \colon A \to B$ be nuclear maps and suppose that $\phi$ is order zero. The following are equivalent
\begin{itemize}
 \item[$(i)$] $\phi$ approximately dominates $\rho$,
 \item[$(ii)$] $\mathcal I(\rho) \leq \mathcal I(\phi)$,
 \item[$(iii)$] $\rho(a) \in \overline{B\phi(a) B}$ for all $a\in A_+$.
\end{itemize}
\end{theorem}

\begin{proof}
$(i)\Rightarrow (ii)$: Obvious. 

$(ii) \Rightarrow (iii)$: Assume $\mathcal I(\rho) \leq \mathcal I(\phi)$ and let $a\in A_+$ be positive. As $\rho(a) \in \mathcal I(\rho)(\overline{AaA}) \subseteq \mathcal I(\phi)(\overline{AaA})$, it suffices to show that $\mathcal I(\phi)(\overline{AaA}) \subseteq \overline{B\phi(a) B}$. However, $\phi(xay) \in \overline{B\phi(a)B}$ for all $x,y\in A$ by Lemma \ref{l:ozideals}, so this is obvious.

 $(iii) \Rightarrow (i)$: Assume $\rho(a)\in \overline{B\phi(a)B}$ for all $a\in A_+$. Let $\mathscr C$ be the convex cone of all c.p.~maps $A\to B$ which are approximately dominated by $\phi$. Clearly every map in $\mathscr C$ is nuclear. As $\mathscr C$ is point-norm closed, it follows by a Hahn--Banach separation argument that it is point-weakly closed. Hence, to show that $\rho \in \mathscr C$ it suffices to show that given $a_1,\dots,a_n\in A$, $\epsilon >0$ and $f_1,\dots,f_n\in B^\ast$ then there is a $\psi \in \mathscr C$ such that
 \[
  | f_i(\rho(a_i)) - f_i(\psi(a_i))| < \epsilon, \quad \text{for } i=1,\dots, n.
 \]
 By the Radon--Nikodym theorem for $C^\ast$-algebras, see e.g.~\cite[Lemma 7.17 (i)]{KirchbergRordam-absorbingOinfty}, we may find a cyclic representation $\pi \colon  B \to \mathbb B(\mathcal H)$ with cyclic vector $\xi \in \mathcal H$, and elements $c_1, \dots, c_n \in \pi( B)' \cap \mathbb B(\mathcal H)$ such that $f_i(b) = \langle \pi(b) c_i \xi, \xi\rangle$ for $i=1,\dots, n$. Let $C = C^\ast(c_1,\dots,c_n)$ and $\iota \colon C \hookrightarrow \mathbb B(\mathcal H)$ be the inclusion.
 For any c.p.~map $\eta \colon A \to B$ there is an induced positive linear functional on $A \otimes_{\max{}} C$ given by the composition
 \[
  A \otimes_{\max{}} C \xrightarrow{\eta \otimes id_{C}}  B \otimes_{\max{}}  C \xrightarrow{\pi \times \iota} \mathbb B(\mathcal H) \xrightarrow{ \omega_\xi} \mathbb C,
 \]
 where $\omega_\xi$ is the vector functional induced by $\xi$, i.e.~$\omega_\xi(T) = \langle T\xi , \xi\rangle$. 
 If $\eta$ is nuclear, then $\eta \otimes id_{C}$ above factors through the spatial tensor product $A \otimes C$, see e.g.~\cite[Lemma 3.6.10]{BrownOzawa-book-approx},
 so if $\eta$ is nuclear it induces a positive linear functional $\gamma_\eta$ on $A \otimes  C$. Thus, as $\rho$ and any $\psi\in \mathscr C$ are nuclear, we get induced positive linear functionals $\gamma_\rho$ and $\gamma_\psi$ on $A\otimes C$.
 
 Let $\mathscr K$ be the weak$^\ast$ closure of $\{ \gamma_{\psi} : \psi \in \mathscr C\} \subseteq ( A \otimes C)^\ast$. 
 It suffices to show that $\gamma_\rho \in \mathscr K$. In fact, if $|\gamma_\rho(a_i \otimes c_i) -  \gamma_\psi(a_i \otimes c_i)| < \epsilon$ for some $\psi \in \mathscr C$, then
 \[
  f_i(\rho(a_i)) = \langle \pi(\rho(a_i)) c_i \xi , \xi \rangle = \gamma_\rho(a_i \otimes c_i) \approx_\epsilon \gamma_\psi(a_i \otimes c_i) = \langle \pi(\psi(a_i)) c_i \xi , \xi \rangle = f_i(\psi(a_i)),
 \]
 for $i=1,\dots, n$, which is what we want to prove. It is easily verified (e.g.~by checking on elementary tensors $a\otimes c$) that $\gamma_{\psi_1} + \gamma_{\psi_2} = \gamma_{\psi_1 + \psi_2}$, 
 and that $t \gamma_\psi = \gamma_{t\psi}$ for $t \in \mathbb R_+$. Hence $\mathscr K$ is a weak$^\ast$ closed convex cone of positive linear functionals.
 
 We want to show that if $\gamma \in \mathscr K$ and $d\in  A \otimes  C$ then $d^\ast \gamma d := \gamma(d^\ast(-)d) \in \mathscr K$.
 Since $\mathscr K$ is weak$^\ast$ closed it suffices to show this for $\gamma = \gamma_\psi$ where $\psi = \sum_{l=i}^m e_i^\ast \phi(-)e_i$ for $e_1,\dots,e_m \in B$, and $d= \sum_{j=1}^k x_j \otimes y_j$ 
 where $x_1,\dots,x_k\in A$ and $y_1 , \dots,y_k\in C$. Let $\mathcal F \subset A\otimes C$ be a finite set of elementary tensors and let $\delta>0$ so that we wish to find $\psi_0\in \mathscr C$ such that $d^\ast \gamma_\psi d(a\otimes c) \approx_\delta \gamma_{\psi_0}(a\otimes c)$ for $a\otimes c\in \mathcal F$. 
 Since $\xi$ is cyclic for $\pi$ we may find $b_1,\dots,b_k\in B$ such that $\| \pi(b_j)\xi - y_j \xi \|$ is as small, and $N$ sufficiently large, such that
 \begin{eqnarray*}
  \gamma_{\psi}(d^\ast(a\otimes c)d) &=& \sum_{j,l=1}^k \gamma_{\psi}((x_j^\ast a x_l) \otimes (y_j^\ast c y_l)) \\
  &=& \sum_{j,l=1}^k \langle \pi(\psi(x_j^\ast a x_l)) c y_l \xi, y_j \xi\rangle \\ 
  &\approx_{\delta/2}& \sum_{j,l=1}^k \langle \pi(\psi(x_j^\ast a x_l)) c \pi(b_l) \xi, \pi(b_j) \xi\rangle \\
&=& \langle \pi(\sum_{j,l=1}^k b_j^\ast\psi(x_j^\ast a x_l) b_l) c \xi , \xi\rangle \\
&=& \langle \pi(\sum_{i=1}^m\sum_{j,l=1}^k b_j^\ast e_i^\ast \phi(x_j^\ast a x_l)e_i b_l) c \xi , \xi\rangle \\
&\stackrel{\eqref{eq:bimodule}}{\approx_{\delta/2}}& \langle\pi(\sum_{i=1}^m\sum_{j,l=1}^k b_j^\ast e_i^\ast \phi^{1/N}(x_j^\ast) \phi(a) \phi^{1/N} (x_l)e_i b_l) c \xi , \xi\rangle \\
&=& \gamma_{\psi_0}(a\otimes c)
 \end{eqnarray*}
 for all $a\otimes c\in \mathcal F$, where
 \[
 \psi_0(-) = \sum_{i=1}^m\sum_{j,l=1}^k b_j^\ast e_i^\ast \phi^{1/N}(x_j^\ast) \phi(-) \phi^{1/N} (x_l)e_i b_l = \sum_{i=1}^m f_i^\ast \phi(-) f_i \in \mathscr C
 \]
for $f_i := \sum_{j=1}^k \phi^{1/N} (x_j)e_i b_j$. 
 Thus $d^\ast \gamma d \in \mathscr K$ for any $\gamma \in \mathscr K$ and $d\in A \otimes  C$. 
 Let $J$ be the subset of $ A \otimes  C$ consisting of elements $d$ such that $\gamma(d^\ast d) = 0$ for all $\gamma \in \mathscr K$. 
 By \cite[Lemma 7.17 (ii)]{KirchbergRordam-absorbingOinfty} it follows that $ J$ is a closed two-sided ideal in $ A \otimes  C$, and that $\gamma_\rho \in \mathscr K$ if $\gamma_\rho(d^\ast d) = 0$ for all $d\in J$. In other words, it suffices to show that $J$ is contained in the left kernel $L:=L_{\gamma_\rho}$ of $\gamma_{\rho}$.
 
 Since $A$ is exact it follows that $J = \overline{\mathrm{span}}\{a\otimes c : a\in A, c\in C, a\otimes c \in J\}$, see e.g.~\cite[Corollary 9.4.6]{BrownOzawa-book-approx}. As the left kernel $L$ of $\gamma_\rho$ is a closed linear subspace of $A\otimes C$, it suffices to show that $a\otimes c \in L$ for all elementary tensors $a\otimes c \in J$, so fix such $a\in A$ and $c\in C$.
 
By assumption $\rho(a^\ast a) \in \overline{ B\phi(a^\ast a) B}$. Thus for any $\delta>0$ we may choose $b_1,\dots,b_m \in  B$ such that
 \[
  \| \rho(a^\ast a) - \sum_{j=1}^m b_j^\ast \phi(a^\ast a) b_j\| < \delta.
 \]
 Define $\psi = \sum_{j=1}^m b_j^\ast \phi(-) b_j$ which is in $\mathscr C$, and note that $\| \rho(a^\ast a) - \psi(a^\ast a)\| < \delta$. Since $\gamma_\psi(a^\ast a \otimes c^\ast c) = 0$ we get that
 \begin{eqnarray*}
  |\gamma_\rho(a^\ast a \otimes c^\ast c) | &=& | \gamma_\rho(a^\ast a \otimes c^\ast c) - \gamma_\psi(a^\ast a \otimes c^\ast c)| \\
  &=& | \langle \pi(\rho(a^\ast a)-\psi(a^\ast a)) c \xi, c \xi\rangle | \\
  &<& \delta \| c \xi\|^2.
 \end{eqnarray*}
 Since $\delta$ was arbitrary we get that $\gamma_\rho(a^\ast a \otimes c^\ast c) = 0$ so $a\otimes c\in L$, which finishes the proof.
\end{proof}


\subsection{Murray--von Neumann equivalence of $\ast$-homomorphisms}

It is customary to consider approximate or asymptotic unitary equivalence of $\ast$-homo\-morphisms as the correct equivalence relation for $\ast$-homomorphisms (at least when it comes to classification). However, this does not always seem to be the right framework to consider for not necessarily unital maps in general, in the same sense as unitary equivalence of projections has its downsides compared to Murray--von Neumann equivalence. 

For instance, if $B$ is a unital, infinite $C^\ast$-algebra with a non-unitary isometry $v$ and $\phi \colon A \to B$ is a unital $\ast$-homomorphism, then $\phi$ and $v \phi(-) v^\ast$ are clearly not approximately unitary equivalent as one map is unital and the other is not. However, these two maps will agree on all usual invariants used for classification (e.g.~ideal related $KK$-theory or the Cuntz semigroup), see also Corollary \ref{c:functorinv}. 

Thus, if one's hope is to classify (not necessarily unital) $\ast$-homomorphisms, approximate (or asymptotic) unitary equivalence is often too strong of an equivalence relation.  This motivates the weaker version which I call approximate/asymptotic Murray--von Neumann equivalence.\footnote{I would be surprised if this definition has never appeared before but I know of no reference to such a definition.} 

In many cases approximate/asymptotic Murray--von Neumann equivalence is the same as approximate/asymptotic unitary equivalence, cf.~Proposition \ref{p:MvNvsu}.

\begin{definition}\label{d:MvN}
Let $A$ and $B$ be $C^\ast$-algebras, and let $\phi, \psi \colon A \to B$ be two $\ast$-homo\-morphisms. 
We say that $\phi$ and $\psi$ are \emph{approximately Murray--von Neumann equivalent} if for any finite set $\mathcal F\subset A$, and any $\epsilon>0$, there is $u\in B$ such that
\[
\| u^\ast \phi(a) u - \psi(a)\| < \epsilon , \qquad \| u \psi(a) u^\ast - \phi(a) \| < \epsilon , \qquad a\in \mathcal F.
\]
When $A$ is separable we say that $\phi$ and $\psi$ are \emph{asymptotically Murray--von Neumann equivalent} if there is a contractive\footnote{A proof similar to that of Lemma \ref{l:MvNcontractive} implies that we do not need to assume that $(u_t)$ is contractive or even bounded. Contractivity will be assumed for convenience.} continuous path $(u_t)_{t\in \mathbb R_+}$ in $B$ (where $\mathbb R_+ := [0,\infty)$) such that
\[
\lim_{t\to \infty} u_t^\ast \phi(a) u_t =  \psi(a), \qquad \lim_{t\to\infty} u_t \psi(a) u_t^\ast = \phi(a), \qquad a\in A. 
\]
\end{definition}

A standard argument shows that two $\ast$-homomorphisms $\phi,\psi \colon \mathbb C \to B$ are approximately Murray--von Neumann equivalent if and only if $\phi(1)$ and $\psi(1)$ are Murray--von Neumann equivalent.

More generally, if $A$ is unital then $u$ in the definition of approximate Murray--von Neumann equivalence can be taken to be a partial isometry such that $u^\ast u = \psi(1_A)$ and $uu^\ast = \phi(1_A)$. 

Recall that maps are \emph{approximately/asymptotically unitary equivalent} if they satisfy the condition in Definition \ref{d:MvN} with $u$ (resp.~each $u_t$) a unitary in $\multialg{B}$.

\begin{lemma}\label{l:MvNcontractive}
In the definition of approximate Murray--von Neumann equivalence one may always pick $u$ to be contractive.
\end{lemma}
\begin{proof}
Suppose $\phi$ and $\psi$ are approximately Murray--von Neumann equivalent. Given $\mathcal F\subset A$ finite, and $\epsilon>0$, pick a positive contraction $a\in A$ such that $\| axa - x\| < \epsilon/2$ for all $x\in \mathcal F$. Find $v\in B$ such that
\[
\| v^\ast \phi(x) v - \psi(x)\| < \epsilon/2, \qquad \| v \psi(x) v^\ast - \phi(x) \| < \epsilon/2 , \qquad x\in \mathcal F \cup \cup a\mathcal F a \cup \{ a^2\}. 
\]
Let $u:= \phi(a) v$. It easily follows that
\[
\| u^\ast \phi(x) u - \psi(x)\| < \epsilon, \qquad \| u \psi(x) u^\ast - \phi(x) \| < \epsilon , \qquad x\in \mathcal F.
\]
Thus such $u$ implement an approximate Murray--von Neumann equivalence. Also, as
\[
\| u \|^2 = \| v^\ast \phi(a^2) v\| \leq \| v^\ast \phi(a^2) v - \psi(a^2)\| + \| \psi(a^2)\| \leq 1 + \epsilon/2,
\]
it clearly follows that we may pick $u$ to be contractive.
\end{proof}

\begin{remark}
For a $C^\ast$-algebra $B$ let $B_\infty := \prod_\mathbb{N} B/\bigoplus_\mathbb{N} B$ be the \emph{sequence algebra}, and let $B_\as := C_b(\mathbb R_+, B)/C_0(\mathbb R_+ ,B)$ be the \emph{path algebra} (or asymptotic corona) where $\mathbb R_+ = [0,\infty)$. Clearly $B$ embeds into both $B_\infty$ and $B_\as$ as constant sequences/paths, so we consider $B$ as a $C^\ast$-subalgebra of both $B_\infty$ and $B_\as$. 
\end{remark}

 Due to Lemma \ref{l:MvNcontractive} in the approximate case, the following observation is immediate.

\begin{observation}
If $A$ is separable then $\phi, \psi \colon A \to B$ are approximately (resp.~asymptotically) Murray--von Neumann equivalent if and only if there is a contraction $v\in B_\infty$ (resp.~$v\in B_\as$) such that $v^\ast \phi(-) v = \psi$ and $v\psi(-)v^\ast = \phi$.
\end{observation}

The above characterisation of approximate/asymptotic Murray--von Neumann equivalence will be used throughout the paper without reference. The following lemma indicates how this characterisation will be applied by replacing $D$ below with either $B_\infty$ or $B_\as$.

\begin{lemma}\label{l:conjv}
Let $\phi,\psi\colon A \to D$ be $\ast$-homomorphisms and suppose that there is a contraction $v\in D$ such that $v^\ast\phi(a) v = \psi(a)$ for all $a\in A$. Then
\begin{itemize}
\item[$(i)$] $vv^\ast \in D \cap \phi(A)'$,
\item[$(ii)$] $v^\ast v \psi(a) = \psi(a)$ for all $a\in A$,
\item[$(iii)$] $\phi(a) v = v \psi(a)$ for all $a\in A$.
\end{itemize}
\end{lemma}
\begin{proof}
$(i)$ is a standard trick: using that the map $v^\ast \phi(-) v$ is multiplicative one easily checks $((1-vv^\ast)^{1/2}\phi(a) vv^\ast)^\ast ((1-vv^\ast)^{1/2}\phi(a) vv^\ast) =0$ for any $a\in A$. Thus $(1-vv^\ast)^{1/2} \phi(a) vv^\ast=0$, so $\phi(a) vv^\ast = vv^\ast \phi(a) vv^\ast$ for any $a\in A$. By symmetry $\phi(a) vv^\ast = vv^\ast \phi(a)$.

$(ii)$: Given $a\in A$, pick $b,c\in A$ such that $a=bc$. Then
\[
\psi(a) = \psi(b) \psi(c) = v^\ast \phi(b) vv^\ast \phi(c) v \stackrel{(i)}{=} v^\ast v v^\ast \phi(bc) v = v^\ast v \psi(a).
\]

$(iii)$: By $(i)$, $\phi(a) vv^\ast v = vv^\ast \phi(a) v = v \psi(a)$. Also, we have
\[
(1-v^\ast v) v^\ast \phi(a^\ast) \phi( a) v (1-v^\ast v) = (1-v^\ast v) \psi(a^\ast a) (1-v^\ast v) \stackrel{(ii)}{=} 0,
\]
so $\phi(a) v(1-v^\ast v) = 0$, and thus $\phi(a) v = \phi(a) v v^\ast v = v \psi(a)$.
\end{proof}

\begin{remark}
If $\phi \colon A \to B$ is a $\ast$-homomorphism let 
\[
\mathrm{Ann}(\phi(A)) := \{ b \in B_\infty : b  \phi(A) + \phi(A) b = \{0\}\}
\]
be the annihilator of $\phi(A)$ in $B_\infty$. Then $\mathrm{Ann}(\phi(A))$ is an ideal in the relative commutant $B_\infty \cap \phi(A)'$. If $A$ is $\sigma$-unital and $(a_n)_{n\in \mathbb N}$ is an approximate unit for $A$, then the image of $(\phi(a_n))_{n\in \mathbb N}$ in $B_\infty$ induces a unit in $B_\infty\cap \phi(A)' /\mathrm{Ann}(\phi(A))$, so this $C^\ast$-algebra is unital. 

The analogous results hold for $B_\as$ (taking a continuous path $\mathbb R_+ \ni t \mapsto a_t \in A$ which is an approximate unit), and we again use the notation $\mathrm{Ann}(\phi(A))$ for the annihilator of $\phi(A)$ in $B_\as$.
\end{remark}

Considering relative commutants of the above form is inspired by work of Kirchberg in \cite{Kirchberg-Abel}.

The equivalence of $(i)$ and $(ii)$ in the proposition below is a generalisation of the fact that two projections $p,q$ are Murray--von Neumann equivalent if and only if $p \oplus 0$ and $q\oplus 0$ are unitary equivalent. Condition $(iii)$ is a version of a $2\times 2$-matrix trick of Connes \cite{Connes-classtypeIII}, and $(i)\Leftrightarrow (iii)$ below shows that approximate/asymptotic Murray--von Neumann equivalence is exactly the equivalence relation on $\ast$-homomorphisms for which this ``trick'' is applicable. It will be used in the main result of this section Theorem \ref{t:O2HB}.

My inspiration for considering such matrix tricks comes from recent work of Matui and Sato \cite{MatuiSato-decrankUHF} and by Bosa, Brown, Sato, Tikuisis, White and Winter \cite{BBSTWW-2coloured}.

It is not hard to see that one may replace $M_2(B)_\infty$ below with $M_2(B)_\omega$ for some/any free filter $\omega$ on $\mathbb N$. However, in this paper it suffices to work with $M_2(B)_\infty$. 

I would like to thank the referee for suggesting improving item $(ii)$ below to have the unitaries in $M_2(B)^\sim$, which is the optimal version of approximate/asymptotic unitary equivalence.

\begin{proposition}\label{p:MvNeq}
Let $A$ and $B$ be $C^\ast$-algebras with $A$ separable, and let $\phi, \psi \colon A \to B$ be $\ast$-homomorphisms. The following are equivalent.
\begin{itemize}
\item[$(i)$] $\phi$ and $\psi$ are approximately Murray--von Neumann equivalent,
\item[$(ii)$] $\phi\oplus 0, \psi \oplus 0 \colon A \to M_2(B)$ are approximately unitary equivalent (with unitaries in the minimal unitisation $M_2(B)^\sim$),
\item[$(iii)$] The two projections
\[
\left( \begin{array}{cc} 1 & 0 \\ 0 & 0 \end{array} \right) , \left( \begin{array}{cc} 0 & 0 \\ 0 & 1 \end{array} \right) \in \frac{M_2(B)_\infty \cap (\phi \oplus \psi)(A)'}{\mathrm{Ann}(\phi\oplus \psi)(A)}
\]
are Murray--von Neumann equivalent.
\end{itemize}
The above is also true if one replaces ``approximately'' with ``asymptotically'', and ``$M_2(B)_\infty$'' with ``$M_2(B)_\as$''.
\end{proposition}

Note that the projections in $(iii)$ can be represented by $(\phi(a_n) \oplus 0)_{n\in \mathbb N}$ and $(0\oplus \psi(a_n))_{n\in \mathbb N}$ respectively where $(a_n)_{n\in \mathbb N}$ is an approximate identity in $A$, and similarly in the asymptotic case. Also, it is crucial that the Murray--von Neumann equivalence in $(iii)$ happens in the specified $C^\ast$-algebra.

\begin{proof}[Proof of Proposition \ref{p:MvNeq}]
The asymptotic version of the proof is identical to the approximate version, so we only prove the approximate one.

$(i)\Rightarrow (iii)$: By Lemma \ref{l:MvNcontractive} we find a $u\in B_\infty$ such that $u^\ast \phi(-) u = \psi$ and $u\psi(-)u^\ast = \phi$. Using Lemma \ref{l:conjv}(i) it follows that
\[
(\phi(a)\oplus \psi(a) ) (u \otimes e_{12}) = \phi(a)u \otimes e_{12} = u\psi(a) \otimes e_{12} = (u\otimes e_{12})(\phi(a) \oplus \psi(a)),
\]
so $u\otimes e_{12} \in M_2(B)_\infty \cap (\phi\oplus \psi)(A)'$. Moreover, by Lemma \ref{l:conjv}(ii) we get
\[
(u \otimes e_{12})(u\otimes e_{12})^\ast (\phi(a) \oplus \psi(a)) = uu^\ast \phi(a) \oplus 0 = \phi(a) \oplus 0 = (1\oplus 0) (\phi(a) \oplus \psi(a)),
\]
and similarly $(u \otimes e_{12})^\ast (u\otimes e_{12}) (\phi(a) \oplus \psi(a)) = (0 \oplus 1) (\phi(a) \oplus \psi(a))$.
Hence
\begin{equation}\label{eq:partiso}
V:= u \otimes e_{12} + \mathrm{Ann}((\phi \oplus \psi)(A)) \in \frac{M_2(B)_\infty \cap (\phi \oplus \psi)(A)'}{\mathrm{Ann}((\phi\oplus \psi)(A))}
\end{equation}
is a partial isometry for which $VV^\ast = 1\oplus 0$ and $V^\ast V= 0\oplus 1$.

$(iii) \Rightarrow (ii)$: Note that the inclusion $M_2(B) \hookrightarrow M_2(B)^\sim$ induces an isomorphism
\[
\frac{M_2(B)_\infty \cap (\phi \oplus \psi)(A)'}{\mathrm{Ann}((\phi\oplus \psi)(A))} \xrightarrow \cong \frac{(M_2(B)^\sim)_\infty \cap (\phi \oplus \psi)(A)'}{\mathrm{Ann}((\phi\oplus \psi)(A))}
\]
as the unit of the left hand side is represented by an element in $M_2(B)_\infty$. Let $V$ be a partial isometry with $VV^\ast = 1\oplus 0$ and $V^\ast V = 0\oplus 1$. Then $U = V + V^\ast$ is a unitary such that $U^\ast (1\oplus 0) U = 0\oplus 1$. 

Since $U$ is a self-adjoint unitary we may lift $U$ to a unitary $u \in (M_2(B)^\sim)_\infty \cap (\phi \oplus \psi)(A)'$. We get
\begin{eqnarray*}
u^\ast (\phi(a) \oplus 0) u &=& u^\ast (1\oplus 0 ) (\phi(a) \oplus \psi(a)) u \\
&=& u^\ast (1\oplus 0 ) u (\phi(a) \oplus \psi(a)) \\
&=& (0\oplus 1) (\phi(a) \oplus \psi(a)) \\
&=& 0 \oplus \psi(a)
\end{eqnarray*}
for all $a \in A$.\footnote{At this point it would be tempting to use a rotation unitary to conclude that $\phi \oplus 0$ and $\psi \oplus 0$ are approximately unitary equivalent. However, then the resulting sequence of unitaries would be in $M_2(\widetilde B)$ and not in $M_2(B)^\sim$.} As the inclusion $M_2(B)^\sim \to M_2(\widetilde B)$ induces an isomorphism
\[
\frac{(M_2(B)^\sim)_\infty \cap (\psi \oplus \psi)(A)'}{\mathrm{Ann}((\psi\oplus \psi)(A))} \xrightarrow \cong \frac{M_2(\widetilde B)_\infty \cap (\psi \oplus \psi)(A)'}{\mathrm{Ann}((\psi\oplus \psi)(A))},\footnote{Note that we use $\psi \oplus \psi$ here and not $\phi \oplus \psi$.}
\]
we may lift the self-adjoint unitary $W = e_{12} + e_{21} \in \frac{M_2(\widetilde B)_\infty \cap (\psi \oplus \psi)(A)'}{\mathrm{Ann}((\psi\oplus \psi)(A))}$ to a unitary $w \in (M_2(B)^\sim)_\infty \cap (\psi \oplus \psi)(A)'$. One checks exactly as above that
\[
w^\ast(0 \oplus \psi(a)) w = \psi(a) \oplus 0
\]
for all $a\in A$. Hence $uw \in (M_2(B)^\sim)_\infty$ is a unitary such that $w^\ast u^\ast (\phi(a) \oplus 0) u w = \psi(a) \oplus 0$.

$(ii)\Rightarrow (i)$: Let $(a_n)_{n\in \mathbb N}$ be an approximate identity in $A$. If $u\in M_2(E)_\infty$ is a unitary such that $u^\ast (\phi \oplus 0) u = \psi \oplus 0$, then 
\[
v := [(\phi(a_n) \oplus 0)_{n\in \mathbb N}] u [(\psi(a_n) \oplus 0)_{n\in \mathbb N}] \in (B \oplus 0)_\infty = B_\infty
\]
 induces an approximate Murray--von Neumann equivalence.
\end{proof}

The following essentially states that any functor used for classification which does not take the class of the unit in $K_0$ or similar into consideration, will not be able to distinguish $\ast$-homomorphisms which are approximately/asymptotically Murray--von Neumann equivalent. Hence, if one wants uniqueness results of (not necessarily unital) $\ast$-homomorphisms, approximate/asymptotic Murray--von Neumann equivalence is a more appropriate equivalence relation than approximate/asymp\-totic unitary equivalence.

\begin{corollary}\label{c:functorinv}
Let $F$ be any functor for which the domain is the category of separable $C^\ast$-algebras. Suppose that $F$ is invariant under approximate/asymptotic unitary equivalence and is $M_2$-stable, i.e.~$F(id_B \oplus 0) \colon F(B) \xrightarrow \cong F(M_2(B))$ is an isomorphism for any $B$. Then $F$ is invariant under approximate/asymptotic Murray--von Neumann equivalence.
\end{corollary}
\begin{proof}
Suppose $\phi,\psi \colon A \to B$ are approximately/asymptotically Murray--von Neumann equivalent. By Proposition \ref{p:MvNeq}, $\phi \oplus 0$ and $\psi\oplus 0$ are approximately/asymptotically unitary equivalent, so
\[
F(\phi) = F(id_B \oplus 0)^{-1} \circ F(\phi \oplus 0) = F(id_B \oplus 0)^{-1} \circ F(\psi \oplus 0 ) = F(\psi). \qedhere
\]
\end{proof}

\begin{lemma}\label{l:stableasym}
Let $\mathcal H$ be a infinite dimensional, separable Hilbert space with orthonormal basis $(\xi_n)_{n\in \mathbb N}$, and let $T_1 \in \mathbb B(\mathcal H)$ be given by $T_1 \xi_n = \xi_{2n-1}$ for $n\in \mathbb N$. Then the two $\ast$-homomorphisms $id_{\mathbb K(\mathcal H)}, T_1(-) T_1^\ast \colon \mathbb K(\mathcal H) \to \mathbb K(\mathcal H)$ are asymptotically unitary equivalent.
\end{lemma}
\begin{proof}
We construct a norm-continuous path $(U_t)_{t\in [1,\infty)}$ of unitaries in $\mathbb B(\mathcal H)$ implementing the asymptotic unitary equivalence as follows: Let $U_1 = 1$ and suppose that we have constructed $U_t$ for $t\in [1, k-1]$ for some integer $k\geq 2$. We may fix a continuous path $(V_{k,t})_{t\in [k-1,k]}$ of unitaries such that $V_{k,k-1} = 1$, $V_{k,t}\xi_n = \xi_n$ for $n\neq k, 2k-1$ and all $t\in [k-1,k]$, and such that $V_{k,k}\xi_{k} = \xi_{2k-1}$ and $V_{k,k}\xi_{2k-1} = \xi_{k}$. Let $V_k := V_{k,k}$ and define $U_t := V_{k,t} U_{k-1} = V_{k,t} V_{k-1} \dots V_{1}$ for $t\in [k-1,k]$. This gives a norm-continuous path $(U_t)_{t\in [1,\infty)}$ of unitaries in $\mathbb B(\mathcal H)$.

By how each $V_{l,t}$ is constructed it holds that $V_l\xi_{2n-1} = \xi_{2n-1}$ for $l=1,\dots, n-1$. Hence for $t\geq n$ we get
\[
U_t T_1 \xi_n = V_{\lceil t \rceil, t} V_{\lfloor t \rfloor} \dots V_n \dots V_1 \xi_{2n-1} = V_{\lceil t \rceil, t} V_{\lfloor t \rfloor} \dots V_n \xi_{2n-1} = V_{\lceil t \rceil , t} V_{\lfloor t \rfloor} \dots V_{n+1} \xi_{n} = \xi_n.
\]
Hence if $\theta_{\xi,\xi'}$ for $\xi, \xi' \in \mathcal H$ is the rank 1 operator $\theta_{\xi,\xi'}\eta = \xi \langle \eta , \xi'\rangle$, we get
\[
U_t T_1 \theta_{\xi_n, \xi_m} T_1^\ast U_t^\ast = \theta_{U_t T_1 \xi_n, U_t T_1 \xi_m} = \theta_{\xi_n , \xi_m}, \qquad n,m \in \mathbb N, t \geq \max\{n,m\}.
\]
As $\span_{n,m\in \mathbb N} \theta_{\xi_n,\xi_m}$ is dense in $\mathbb K(\mathcal H)$ it follows that $id_{\mathbb K(\mathcal H)}$ and $T_1(-)T_1^\ast$ are asymptotically unitary equivalent.
\end{proof}

\begin{proposition}\label{p:MvNvsu}
Let $A$ and $B$ be $C^\ast$-algebras with $A$ separable, and let $\phi,\psi \colon A \to B$ be $\ast$-homomorphisms. Consider the following conditions.
\begin{itemize}
\item[$(a)$] $A,B,\phi$ and $\psi$ are all unital,
\item[$(b)$] $B$ is stable,
\item[$(c)$] $B$ has stable rank 1.
\end{itemize}

If $\phi$ and $\psi$ are approximately Murray--von Neumann equivalent and $(a)$, $(b)$ or $(c)$ holds, then $\phi$ and $\psi$ are approximately unitary equivalent.

If $\phi$ and $\psi$ are asymptotically Murray--von Neumann equivalent and $(a)$ or $(b)$ holds, then $\phi$ and $\psi$ are asymptotically unitary equivalent.
\end{proposition}
\begin{proof}
Case $(a)$ is obvious for both approximate and asymptotic equivalences. 

Case $(c)$ follows since if $v\in B_\infty$ implements the approximate Murray--von Neumann equivalence, we may by stable rank 1 decompose $v = u|v|$ with $u\in (\tilde B)_\infty$ a unitary. Hence
\[
u \psi(a) u^\ast \stackrel{\textrm{Lem.}\ref{l:conjv}(ii)}{=} u (v^\ast v)^{1/2} \psi(a) (v^\ast v)^{1/2} u^\ast = v \psi(a) v^\ast = \phi(a)
\]
for all $a\in A$, so $\phi$ and $\psi$ are approximately unitary equivalent.

For case $(b)$ we only prove the asymptotic version, as the approximate version is essentially identical. We may replace $B$ with $B \otimes \mathbb K(\mathcal H)$ where $\mathcal H$ is an infinite dimensional, separable Hilbert space. Let $(\xi_n)_{n\in \mathbb N}$ be an orthonormal basis of $\mathcal H$, and let $T_1,T_2 \in \mathbb B(\mathcal H)$ be the isometries $T_1\xi_n = \xi_{2n-1}$ and $T_2\xi_n = \xi_{2n}$. Let $s_i := 1_{\multialg{B}} \otimes T_i \in \multialg{B\otimes \mathbb K(\mathcal H)}$. By Lemma \ref{l:stableasym} it follows that $\phi = id \circ \phi$ is asymptotically unitary equivalent to $s_1 \phi(-) s_1^\ast$, and that $\psi$ is asymptotically unitary equivalent to $s_1\psi(-)s_1^\ast$. As $s_1,s_2$ are isometries with $s_1s_1^\ast + s_2s_2^\ast=1$, there is an isomorphism $M_2(B \otimes \mathbb K(\mathcal H)) \xrightarrow \cong B\otimes \mathbb K(\mathcal H)$ given by $b \otimes e_{ij} \mapsto s_i b s_j^\ast$. Hence, if $\phi$ and $\psi$ are approximately/asymptotically Murray--von Neumann equivalent, then $s_1\phi(-)s_1^\ast$ and $s_1\psi(-) s_1^\ast$ are approximately/asymptotically unitary equivalent by Proposition \ref{p:MvNeq}. This proves the result in case $(b)$.
\end{proof}

The following corollary shows that approximate Murray--von Neumann equivalence is still a strong enough equivalence relation to get classification up to stable isomorphism.

\begin{corollary}\label{c:MvNstableiso}
Let $A$ and $B$ be separable $C^\ast$-algebras, and let $\phi\colon A \to B$ and $\psi \colon B \to A$ be $\ast$-homomorphisms. If $\psi \circ  \phi$ and $id_A$ are approximately Murray--von Neumann equivalent, and $\phi \circ \psi$ and $id_B$ are approximately Murray--von Neumann equivalent, then $A\otimes \mathbb K \cong B \otimes \mathbb K$.
\end{corollary}
\begin{proof}
By Proposition \ref{p:MvNvsu} the maps $(\psi \circ \phi) \otimes id_{\mathbb K} = (\psi \otimes id_{\mathbb K}) \circ (\phi \otimes id_{\mathbb K})$ and $id_A \otimes id_{\mathbb K} = id_{A \otimes \mathbb K}$ are approximately unitary equivalent, and $(\phi \otimes id_{\mathbb K})\circ (\psi \otimes id_{\mathbb K}) $ and $id_{B\otimes \mathbb K}$ are approximately unitary equivalent. Hence $A\otimes \mathbb K \cong B \otimes \mathbb K$ by an intertwining argument a la Elliott, see e.g.~\cite[Corollary 2.3.4]{Rordam-book-classification}.
\end{proof}

\begin{remark}
Let $\phi \colon \mathcal O_2 \to \mathcal O_2\otimes \mathbb K$ and $\psi \colon \mathcal O_2\otimes \mathbb K \to \mathcal O_2$ be injective $\ast$-homomorphisms. By Theorem \ref{t:O2HB} below, $\psi \circ \phi$ and $id_{\mathcal O_2}$, and $\phi \circ \psi$ and $id_{\mathcal O_2\otimes \mathbb K}$ are approximately Murray--von Neumann equivalent. This shows that $A\not \cong B$ in general in Corollary \ref{c:MvNstableiso}.
\end{remark}


\subsection{$\mathcal O_\infty$-stable and $\mathcal O_2$-stable $\ast$-homomorphisms}

\begin{definition}\label{d:Dstable}
Let $\mathcal D$ be either $\mathcal O_2$ or $\mathcal O_\infty$. Let $A$ and $B$ be $C^\ast$-algebras with $A$ separable, and let $\phi \colon A \to B$ be a $\ast$-homomorphism. We say that $\phi$ is 
\begin{itemize}
\item \emph{$\mathcal D$-stable} if $\mathcal D$ embeds unitally in $(B_\infty \cap \phi(A)')/\mathrm{Ann}(\phi(A))$.
\item \emph{strongly $\mathcal D$-stable} if $\mathcal D$ embeds unitally in $(B_\as \cap \phi(A)')/\mathrm{Ann}(\phi(A))$.
\end{itemize}
\end{definition}

\begin{remark}
In Corollary \ref{c:McDuff} it is shown that $\mathcal O_2$- and $\mathcal O_\infty$-stable $\ast$-homo\-morphisms have a McDuff type property, which really is the motivation for why the maps have been given this name.

There is an obvious generalisation of $\mathcal D$-stable maps for any strongly self-absorbing $C^\ast$-algebra $\mathcal D$. However, with this definition it seems unlikely that they satisfy the McDuff type property and therefore I do not believe that this is the correct generalisation of $\mathcal D$-stable $\ast$-homomorphisms for more general strongly self-absorbing $C^\ast$-algebras.
\end{remark}

In \cite[Definition 8.2.4]{Rordam-book-classification} Rørdam introduces $\mathcal O_2$-absorbing and $\mathcal O_\infty$-absorbing $\ast$-homo\-morphisms. These are unital $\ast$-homomorphisms $\phi \colon A \to B$ such that $\mathcal O_2$ (resp.~$\mathcal O_\infty$) embeds unitally in the commutant $B \cap \phi(A)'$. I emphasise that although these notions are closely related to Definition \ref{d:Dstable}, they are \emph{not} the same. For instance, the identity map $id_{\mathcal O_2}$ is $\mathcal O_2$-stable (in the above sense), but not $\mathcal O_2$-absorbing in the sense of Rørdam.

\begin{proposition}\label{p:Dstablealgebras}
Let $A$ and $B$ be $C^\ast$-algebras with $A$ separable, and let $\phi \colon A \to B$ be a $\ast$-homomorphism. Let $\mathcal D$ be either $\mathcal O_2$ or $\mathcal O_\infty$. If either $A$ or $B$ is $\mathcal D$-stable then $\phi$ is strongly $\mathcal D$-stable. 
\end{proposition}
\begin{proof}
Let $\iota_n \colon \mathcal D \to \mathcal D \otimes \mathcal D \otimes \dots =: \mathcal D^{\otimes \infty}$ be the embedding into the $n$'th tensor. By \cite[Theorem 2.2]{DadarlatWinter-KKssa}, $\iota_n$ and $\iota_{n+1}$ are asymptotically unitary equivalent with unitaries in $\mathcal D_{n,n+1} := 1\otimes \dots \otimes 1 \otimes \mathcal D \otimes \mathcal D \otimes 1 \otimes \dots$, where the two $\mathcal D$'s are the $n$'th and $(n+1)$'st tensors. Hence there are unital homotopies $\sigma_t \colon \mathcal D \to C([n,n+1], \mathcal D^{\otimes \infty})$ from $\iota_n$ to $\iota_{n+1}$, such that $\sigma_t(d)\in \mathcal D_{n,n+1}$. We get an induced unital embedding of $\sigma_t \colon \mathcal D \to C_b(\mathbb R_+, \mathcal D)$, as $\mathcal D \cong \mathcal D^{\otimes \infty}$. By construction, $\| [x , \sigma_t(y)]\| \to 0$ for all $x,y\in \mathcal D$.

Let $(a_t)$ be a continuous approximate unit of $A$. If $\alpha \colon A \otimes \mathcal D \xrightarrow \cong A$ is an isomorphism, then $\eta \colon \mathcal D \to C_b(\mathbb R_+,B)$ given by $\eta(d)(t) = \phi(\alpha(a_t \otimes \sigma_t(d)))$ induces strong $\mathcal D$-stability.

Suppose $\beta \colon B \otimes \mathcal D \xrightarrow \cong B$ is an isomorphism. Then $\eta \colon \mathcal D \to C_b(\mathbb R_+,B)$ given by $\eta(d)(t) =\phi(a_t) \multialg{\beta}(1_{\multialg{B}} \otimes \sigma_t(d))$ induces strong $\mathcal D$-stability.
\end{proof}

\begin{proposition}
Let $\mathcal D$ be either $\mathcal O_2$ or $\mathcal O_\infty$. A separable $C^\ast$-algebra $A$ is $\mathcal D$-stable if and only if $id_A$ is $\mathcal D$-stable.
\end{proposition}
\begin{proof}
If $A$ is $\mathcal D$-stable then $id_A$ is $\mathcal D$-stable by Proposition \ref{p:Dstablealgebras}. If $id_A$ is $\mathcal D$-stable then $A$ is $\mathcal D$-stable by \cite[Proposition 4.4(4,5)]{Kirchberg-Abel}.\footnote{Alternatively, let $\sigma \colon \mathcal D \to (A_\infty \cap A')/\mathrm{Ann}(A)$ be a unital $\ast$-homomorphism, and $\overline \sigma \colon \mathcal D \to A_\infty \cap A'$ be any map that lifts $\sigma$. We obtain a $\ast$-homomorphism $A \otimes \mathcal D \to A_\infty$ given on elementary tensors by $a \otimes d \mapsto a \overline \sigma (d)$. In particular, $a\otimes 1_{\mathcal D} \mapsto a$, so \cite[Theorem 2.3]{TomsWinter-ssa} implies that $A \otimes \mathcal D \cong A$.}
\end{proof}

\begin{lemma}\label{l:Dcompose}
Let $A, B$ and $C$ be $C^\ast$-algebras with $A$ separable, and let $\phi \colon A \to B$, $\psi \colon B \to C$ be $\ast$-homo\-morphisms. Let $\mathcal D$ be either $\mathcal O_2$ or $\mathcal O_\infty$. 
\begin{itemize}
 \item[$(i)$] If $\phi$ is (strongly) $\mathcal D$-stable then $\psi \circ \phi$ is (strongly) $\mathcal D$-stable.
 \item[$(ii)$] If $B$ is separable and $\psi$ is (strongly) $\mathcal D$-stable then $\psi \circ \phi$ is (strongly) $\mathcal D$-stable.
 \end{itemize}
\end{lemma}
\begin{proof}
We only do the $\mathcal D$-stable cases, as the strongly $\mathcal D$-stable cases are virtually identical.

$(i)$: It is straightforward to check that $\psi_\infty \colon B_\infty \to C_\infty$ induces a $\ast$-homomorphism
\[
\psi_\infty \colon \frac{B_\infty \cap \phi(A)'}{\mathrm{Ann}(\phi(A))} \to \frac{C_\infty \cap \psi\circ\phi(A)'}{\mathrm{Ann}(\psi\circ\phi(A))}.
\]
Let $(a_n)_{n\in \mathbb N}$ be an approximate identity for $A$. Then the unit of the left (resp.~right) hand side above is represented by $(\phi(a_n))_{n\in \mathbb N})$ (resp.~$(\psi\circ\phi(a_n))_{n\in \mathbb N} = \psi_\infty((\phi(a_n))_{n\in \mathbb N})$). Hence the map above is unital. Thus, if $\mathcal D$ embeds unitally into the left hand side above then it embeds unitally into the right hand side.

$(ii)$: There is an obvious map
\[
\frac{C_\infty \cap \psi(B)'}{\mathrm{Ann}(\psi(B))} \to \frac{C_\infty \cap \psi \circ\phi(A)'}{\mathrm{Ann}(\psi\circ\phi(A))}.
\]
As in $(i)$, it suffices to show that this map is unital. The unit of the left hand side is represented by $(\psi(b_n))_{n\in \mathbb N}$, where $(b_n)$ is an approximate identity in $B$. For any $a\in A$, we have that 
\[
(1-\psi(b_n)) \psi \phi(a) = \psi(\phi(a) - b_n \phi(a)) \to 0 , \quad n \to \infty,
\]
so $(\psi(b_n))_{n\in \mathbb N}$ also induces the unit in the right hand side above, so the map above is unital.
\end{proof}

\begin{remark}
By Proposition \ref{p:Dstablealgebras}, $\mathcal O_2$-stable $\ast$-homomorphisms occur naturally as $\ast$-homo\-morphisms between $\mathcal O_2$-stable $C^\ast$-algebras. However, there are also many interesting examples where this is not the case. A nice example comes from the study of $KK$-theory and $\Ext$-theory. Given any $\ast$-homomorphism $\phi \colon A \to \multialg{B}$ one can form the infinite repeat by $\phi \otimes 1_{\multialg{\mathbb K}} \colon A \to \multialg{B \otimes \mathbb K}$. Since such an infinite repeat factors through $\multialg{B} \otimes \mathcal O_2 \subseteq \multialg{B \otimes \mathbb K}$, such infinite repeats are always $\mathcal O_2$-stable by Lemma \ref{l:Dcompose}. Thus, all types of Weyl--von Neumann--Voiculescu theorems such as those due to Kasparov \cite{Kasparov-Stinespring}, Kirchberg \cite{Kirchberg-simple} and Elliott--Kucerovsky \cite{ElliottKucerovsky-extensions} are essentially about characterising when (full, weakly nuclear) $\ast$-homomorphisms $A \to \multialg{B\otimes \mathbb K}$ are $\mathcal O_2$-stable.
\end{remark}

The following is a Stinespring type theorem for (strongly) $\mathcal O_\infty$-stable maps. Recall that we consider $B$ as a $C^\ast$-subalgebra of $B_\infty$ and of $B_\as$.

\begin{theorem}\label{t:OinftyHB}
Let $A$ and $B$ be $C^\ast$-algebras with $A$ separable and exact, let $\phi \colon A \to B$ be a nuclear $\ast$-homomorphism, and let $\rho \colon A \to B$ be a nuclear c.p.~map such that $\mathcal I(\rho) \leq \mathcal I(\phi)$. 
\begin{itemize}
\item[$(i)$] If $\phi$ is $\mathcal O_\infty$-stable then there is an element $v\in B_\infty$ of norm $\|\rho\|^{1/2}$ such that $v^\ast \phi(a) v = \rho(a)$ for all $a\in A$.
\item[$(ii)$] If $\phi$ is strongly $\mathcal O_\infty$-stable then there is an element $v\in B_\as$ of norm $\| \rho\|^{1/2}$ such that $v^\ast \phi(a) v = \rho(a)$ for all $a\in A$.
\end{itemize}
\end{theorem}
\begin{proof}
$(i)$: By Theorem \ref{t:HB}, $\phi$ approximately dominates $\rho$. We first show that we may take $n=1$ in the definition of approximate domination, Definition \ref{d:approxdom}. Given $\mathcal F \subset A$ finite and $\epsilon>0$, find $b_1,\dots, b_n$ such that
\[
\| \rho(a) - \sum_{k=1}^n b_k^\ast \phi(a) b_k\| < \epsilon, \qquad a\in \mathcal F.
\]
Passing to $B_\infty$, fix elements $s_1,s_2,\dots$ in $B_\infty \cap \phi(A)'$ such that $s_1 + \mathrm{Ann}(\phi(A)), s_2 + \mathrm{Ann}(\phi(A)),\dots$  are isometries with orthogonal range projections. Then
\[
\sum_{k=1}^n b_k^\ast \phi(a) b_k = \sum_{k=1}^n b_k^\ast s_k^\ast \phi(a) s_k b_k = c^\ast \phi(a) c
\]
where $c= \sum_{k=1}^n s_k b_k$. Let $(s_k^{(m)})_{m=1}^\infty$ be a lift of $s_k$ for each $k$, and let $c_m = \sum_{k=1}^n s_k^{(m)} b_k$. For large $m$ we have
\[
\| \rho(a) - c_m^\ast \phi(a) c_m\| < \epsilon , \qquad a\in \mathcal F.
\]
Hence we obtain approximate domination with $n=1$. By separability of $A$, there is a sequence $(d_n)$ in $B$ such that $d_n^\ast \phi(a) d_n \to \rho(a)$ for all $a\in A$.  Let $(a_n)_{n\in \mathbb N}$ be an approximate identity for $A$. By passing to a subsequence of $(d_n)$ we may assume that $\| d_n^\ast\phi(a_n^2) d_n - \rho(a_n^2)\| \to 0$. Let $v_n = \phi(a_n) d_n$. As $\| \rho(a_n^2)\| \to \| \rho\|$ it follows that $\limsup_{n\to \infty} \|v_n\| = \|\rho\|^{1/2}$, so $(v_n)$ induces an element $v\in B_\infty$ of norm at most $\|\rho\|^{1/2}$ satisfying $v^\ast \phi(a) v = \rho(a)$ for all $a\in A$.

$(ii)$: By part $(i)$ we find a sequence $(w_n)$ in $B$ of norm $\|\rho\|^{1/2}$ such that $w_n^\ast \phi(a) w_n \to \rho(a)$ for all $a\in A$. As $\mathcal O_\infty$ embeds unitally in $(B_\as \cap \phi(A)')/\mathrm{Ann}(\phi(A))$ we may find a sequence of contractions $s_1,s_2,\dots\in C_b(\mathbb R_+,B)$, which induces such a unital copy of $\mathcal O_\infty$ (so that the $s_i$ induce isometries, with orthogonal range projections).

Let $(a_k)_{k\in \mathbb N}$ be a dense sequence in $A$. We will find a continuous function $r \colon \mathbb R_+ \to \mathbb R_+$ such that $\| s_j(r(t))^\ast \phi(a_k) s_j(r(t)) - \phi(a_k)\| < 1/n$ and $\|s_i(r(t))^\ast \phi(a_k) s_j(r(t))\| < 1/n$ for $i,j,k=1,\dots,n+1$, $i\neq j$ and all $t\geq n$. We construct such an $r$ recursively on the intervals $[n-1,n]$. Let $r(0) = 0$ and suppose that we have constructed $r(t)$ on $[0,n-1]$. Pick $r(n-1) \leq R_n \in \mathbb R_+$ such that $\| s_j(x)^\ast \phi(a_k) s_j(x) - \phi(a_k)\| < 1/n$ and $\|s_i(x)^\ast \phi(a_k) s_j(x)\| < 1/n$ for $i,j,k=1,\dots, n+1$, $i\neq j$ and all $x\geq R_n$. Let $r(t) = r(n-1) (n - t) + (t - n+1)R_n$ for $t\in [n-1,n]$. This gives a continuous function $r\colon \mathbb R_+ \to \mathbb R_+$ with the desired property.

Now, define $v_t := (n+1-t)^{1/2} s_n(r(t)) w_n + (t - n)^{1/2} s_{n+1}(r(t)) w_{n+1}$ for $t\in [n,n+1]$, $n\in \mathbb N$. Clearly $\limsup_{t\to \infty} \| v_t \| = \|\rho\|^{1/2}$. For every $k\in \mathbb N$, every $n \geq k$ and every $t\in [n,n+1]$, we have
\begin{eqnarray*}
&& \| v_t^\ast \phi(a_k) v_t - \rho(a_k) \| \\
&\leq & (n+1-t)\Big( \| w_n^\ast \phi(a_k) w_n - \rho(a_k) \| + \|\rho\| \|s_n^\ast(r(t)) \phi(a_k) s_n(r(t)) - \phi(a_k)\| \Big) \\
&& + (t-n) \Big( \| w_{n+1}^\ast \phi(a_k) w_{n+1} - \rho(a_k)\| + \|\rho\| \|s_{n+1}(r(t))^\ast \phi(a_k) s_{n+1}(r(t)) - \phi(a_k) \| \Big) \\
&& + (n+1-t)^{1/2} (t-n)^{1/2} \|\rho\| \| s_n^\ast(r(t)) \phi(a_k) s_{n+1}(r(t)) \| \\
&& +  (n+1-t)^{1/2} (t-n)^{1/2} \| \rho \| \| s_{n+1}^\ast(r(t)) \phi(a_k) s_n \| \\
&\leq&  \| w_n^\ast \phi(a_k) w_n - \rho(a_k) \| + \| w_{n+1}^\ast \phi(a_k) w_{n+1} - \rho(a_k)\| + 4\|\rho \|/n
\end{eqnarray*}
From this it easily follows that $v_t^\ast \phi(a) v_t \to \rho(a)$ for any $a\in A$.
\end{proof}

\begin{theorem}\label{t:O2HB}
Let $A$ and $B$ be $C^\ast$-algebras with $A$ separable and exact, and let $\phi, \psi \colon A \to B$ be nuclear $\ast$-homomorphisms.
\begin{itemize}
\item[$(i)$] If $\phi$ and $\psi$ are $\mathcal O_2$-stable, then $\phi$ and $\psi$ are approximately Murray--von Neumann equivalent if and only if $\mathcal I(\phi) = \mathcal I(\psi)$.
\item[$(ii)$] If $\phi$ and $\psi$ are strongly $\mathcal O_2$-stable, then $\phi$ and $\psi$ are asymptotically Murray--von Neumann equivalent if and only if $\mathcal I(\phi) = \mathcal I(\psi)$.
\end{itemize}
Moreover, if either $A,B, \phi$ and $\psi$ are all unital, or if $B$ is stable, then we may replace ``approximately/asymptotically Murray--von Neumann equivalent'' with ``approximately/asymp\-totically unitary equivalent'' above.
\end{theorem}
\begin{proof}
The proof of the two statements are almost identical simply by interchanging $B_\infty$ and $B_\as$. So we only prove $(i)$.

If $\phi$ and $\psi$ are approximately Murray--von Neumann equivalent, then clearly $\mathcal I(\phi) = \mathcal I(\psi)$. Conversely, suppose $\mathcal I(\phi) = \mathcal I(\psi)$.
Let
\[
D:= \frac{M_2(B)_\infty \cap (\phi\oplus \psi)(A)'}{\mathrm{Ann}(\phi\oplus \psi)(A)}.
\]
By Theorem \ref{t:OinftyHB} there is a $v\in B_\infty$ such that $v^\ast \phi(-) v = \psi$. By Lemma \ref{l:conjv} it follows that $V := v \otimes e_{12} + \mathrm{Ann}(\phi \oplus \psi)(A) \in D$ is well-defined (as $\phi(-)v = v\psi(-)$), that $VV^\ast \leq 1\oplus 0$, and $V^\ast V = 0\oplus 1$ (as $v^\ast v \psi(-) = \psi$). Similarly, $1\oplus 0$ is subequivalent to $0\oplus 1$, so the projections $1\oplus 0$ and $0\oplus 1$ generate the same ideal in $D$. As their sum is $1_D$ it follows that $1\oplus 0$ and $0\oplus 1$ are both full projections in $D$. As
\[
(1\oplus 0) D (1\oplus 0) = \frac{B_\infty \cap \phi(A)'}{\mathrm{Ann}(\phi(A))} \oplus 0,
\]
it follows from $\mathcal O_2$-stability of $\phi$, that $1\oplus 0$ is properly infinite and $[1\oplus 0 ]_0 =0$ in $K_0(D)$. The same holds for $0\oplus 1$.  Thus, by a result of Cuntz \cite{Cuntz-K-theoryI}, $1\oplus 0$ and $0\oplus 1$ are Murray--von Neumann equivalent. Proposition \ref{p:MvNeq} implies that $\phi$ and $\psi$ are approximately Murray--von Neumann equivalent.

The ``moreover'' part follows from Proposition \ref{p:MvNvsu}.
\end{proof}

\begin{remark}
Using the Kirchberg--Phillips theorem and a result of Lin, one can deduce that $\mathcal O_2$-stable $\ast$-homomorphisms are not always strongly $\mathcal O_2$-stable. Let $A,B$ be unital Kirchberg algebras in the UCT class and suppose that $[1_A]_0 = 0\in K_0(A)$ and $[1_B]_0 = 0\in K_0(B)$. Any unital $\ast$-homomorphism $\theta \colon A \to B$ that factors through $\mathcal O_2$ is strongly $\mathcal O_2$-stable and $KK(\theta) =0$. By Theorem \ref{t:O2HB}, any strongly $\mathcal O_2$-stable, unital $\ast$-homomorphism $\phi \colon A \to B$ is asymptotically unitary equivalent to $\theta$, so in particular $KK(\phi) = 0$ (the converse is also true by the uniqueness part in the Kirchberg--Phillips theorem).

Now, if $\mathrm{Pext}(K_\ast(A) , K_{1-\ast}(B)) \neq 0$, apply the UMCT \cite{DadarlatLoring-UMCT} and the Kirchberg--Phillips theorem \cite{Kirchberg-simple}, \cite{Phillips-classification} to find a unital $\ast$-homomorphism $\phi \colon A \to B$ such that $KK(\phi)\neq 0$ but such that $KL(\phi) = 0$.\footnote{The $KL$-groups were originally defined by Rørdam in \cite{Rordam-classsimple} in the presence of a universal coefficient theorem, and were later treated by Dadarlat in \cite{Dadarlat-KKtop} in the general case.} Then $\phi$ is not strongly $\mathcal O_2$-stable, but as $KL(\phi) = KL(\theta)$ it follows from \cite[Theorem 4.10]{Lin-stableapproxuniqueness} that $\phi$ and $\theta$ are approximately unitary equivalent, so $\phi$ is $\mathcal O_2$-stable.
\end{remark}

\begin{question}
Are $\mathcal O_\infty$-stable $\ast$-homomorphisms always strongly $\mathcal O_\infty$-stable?
\end{question}


\section{From approximate morphisms to $\ast$-homomorphisms}\label{s:infty}

Recall that $B_\infty := \prod_\mathbb{N} B / \bigoplus_{\mathbb N} B$ for a $C^\ast$-algebra $B$. In this section a characterisation is given of when a $\ast$-homomorphism $A \to B_\infty$ is unitary equivalent to a constant $\ast$-homomorphism, i.e.~a $\ast$-homomorphism factoring through $B$.

\begin{lemma}\label{l:approxinfty}
Let $A$ and $B$ be $C^\ast$-algebras with $A$ separable and $B$ unital, and let $\phi, \psi \colon A \to B_\infty$ be continuous maps. If $\phi$ and $\psi$ are approximately unitary equivalent, then they are unitary equivalent.
\end{lemma}
\begin{proof}
The proof is a standard ``diagonal'' argument. Let $a_1,a_2,\dots$ be a dense sequence in $A$. For each $n\in \mathbb N$ find a sequence $u_n\in B_\infty$ of unitaries such that $u_n^\ast \phi(a_i) u_n \approx_{1/n} \psi(a_i)$ for $i=1,\dots,n$. Let $(\phi_j), (\psi_j) \colon A \to \prod_{\mathbb N} B$ be (set-theoretical) lifts of $\phi$ and $\psi$ respectively, and let $(u_n^{(j)})_{j\in \mathbb N}\in \prod_{\mathbb N} B$ be a unitary lift of $u_n$ for each $n\in \mathbb N$. Let $k_1\in \mathbb N$ be such that
\[
\| u_1^{(k)\ast} \phi_k(a_1) u_1^{(k)} - \psi_k(a_1) \| \leq 1, \qquad k\geq k_1.
\]
Having found $k_{n-1}$ we let $k_{n}>k_{n-1}$ be such that
\[
\| u_{n}^{(k)\ast} \phi_k(a_i) u_n^{(k)} - \psi_k(a_i) \| \leq 1/n, \qquad i=1,\dots, n, \textrm{ and } k\geq k_n.
\]
Let $v_k = 1_B$ for $k < k_1$, and $v_k=u_{n}^{(k)}$ if $k_n \leq k < k_{n+1}$. We let $v$ be the induced unitary in $B_\infty$. Then $v^\ast \phi(a_i) v = \psi(a_i)$ for all $i \in \mathbb N$ so by continuity of $\phi$ and $\psi$, $v^\ast \phi(a) v = \psi(a)$ for all $a\in A$.
\end{proof}

Note that whenever $\eta \colon \mathbb N \to \mathbb N$ is a map for which $\lim_{k\to \infty}\eta(k) = \infty$, then there is an induced $\ast$-endomorphism $\eta^\ast \colon B_\infty \to B_\infty$ given by
\[
\eta^\ast([(b_1,b_2,\dots)]) = [(b_{\eta(1)} , b_{\eta(2)}, \dots)].
\]

\begin{lemma}\label{l:seqlem}
Let $A$ and $B$ be $C^\ast$-algebras with $A$ separable and $B$ unital. Suppose $\phi \colon A \to B_\infty$ is a $\ast$-homomorphism with the following property: for any map $\eta \colon \mathbb N \to \mathbb N$ for which $\lim_{k\to \infty}\eta(k)=\infty$, the maps $\phi$ and $\eta^\ast \circ \phi$ are approximately unitary equivalent as maps into $B_\infty$.

Let $(\phi_n) \colon A \to \prod_{\mathbb N} B$ be any (not necessarily linear) lift of $\phi$. For every finite $\mathcal F \subset A$, every $\epsilon >0$, and every $m\in \mathbb N$, there is an integer $k\geq m$ such that for every integer $n\geq k$ there is a unitary $u\in B$ for which
\[
\| u^\ast \phi_n(a) u - \phi_k(a) \| < \epsilon, \qquad a\in \mathcal F.
\]
\end{lemma}
\begin{proof}
Suppose for contradiction that the lemma is false. Then there is a finite set $\mathcal F \subset A$, an $\epsilon >0$ and an $m\in \mathbb N$, such that for any integer $k\geq m$ there exists an integer $n_k \geq k$ for which
\begin{equation}\label{eq:unitaryseq}
\max_{a\in \mathcal F} \| u_k^\ast \phi_{n_k}(a) u_k - \phi_k(a)\|  \geq \epsilon,
\end{equation}
for every unitary $u_k\in B$. Let $\eta \colon \mathbb N \to \mathbb N$ be the map $\eta(k) = n_k$ whenever $k\geq m$ and $\eta (k) = 1$ for $k< m$. As $n_k\geq k$ for $k\geq m$, it follows that $\lim_{k\to \infty} \eta(k) = \infty$. As $\phi$ and $\eta^\ast \circ \phi$ are approximately unitary equivalent in $B_\infty$, there is a unitary $u\in B_\infty$ for which
\[
\| u^\ast (\eta^\ast \circ \phi (a)) u - \phi(a) \| < \epsilon, \qquad a\in \mathcal F.
\]
Let $(u_k)_{k\in \mathbb N} \in \prod_{\mathbb N}B$ be a unitary lift of $u$. It follows that
\[
\limsup_{k\to \infty} \| u_{k}^\ast \phi_{n_k}(a) u_{k} - \phi_{k}(a) \| = \| u^\ast (\eta^\ast \circ \phi (a)) u - \phi(a) \| < \epsilon
\]
for all $a\in \mathcal F$. However, this contradicts \eqref{eq:unitaryseq} so the lemma is true.
\end{proof}

The following is essentially a discrete version of \cite[Proposition 1.3.7]{Phillips-classification} and the proofs are very similar.

\begin{theorem}\label{t:lifthom}
Let $A$ and $B$ be $C^\ast$-algebras with $A$ separable and $B$ unital, and let $B_\infty=\prod_\mathbb{N} B / \bigoplus_{\mathbb N} B$. Suppose $\phi \colon A \to B_\infty$ is a $\ast$-homomorphism. Then $\phi$ is unitary equivalent to a $\ast$-homomorphism $\psi \colon A \to B\subseteq B_\infty$ if and only if for any map $\eta \colon \mathbb N \to \mathbb N$ with $\lim_{k\to \infty}\eta(k)=\infty$, the maps $\phi$ and $\eta^\ast \circ \phi$ are approximately unitary equivalent.
\end{theorem}
\begin{proof}
``Only if'': Suppose $u\in B_\infty$ is a unitary such that $\Ad u \circ \psi = \phi$, where $\psi$ factors through $B$. As $\eta^\ast \circ \psi = \psi$ for all $\eta\colon \mathbb N \to \mathbb N$ with $\lim_{n\to \infty}\eta(n) = \infty$, we get
\[
\eta^\ast \circ \phi = \eta^\ast \circ \Ad u \circ \psi = \Ad \eta^\ast(u) \circ \psi = \Ad (\eta^\ast(u) u^\ast) \circ \phi.
\]
``If'': let $\mathcal F_1 \subseteq \mathcal F_2 \subseteq \dots \subseteq A$ be finite sets such that $\bigcup \mathcal F_n$ is dense in $A$. Let $(\phi_k)\colon A \to \prod_\mathbb{N} B$ be any function which is a lift of $\phi$. Pick $k_0 := 1 < k_1<k_2<\dots$ recursively by applying Lemma \ref{l:seqlem} to $\mathcal F = \mathcal F_n$, $\epsilon = 1/2^n$ and $m = k_{n-1}+1$ and then obtain $k = k_n$.
For every $n\in \mathbb N$ we may find a unitary $u_n\in B$ for which
\[
\| u_n^\ast \phi_{k_n}(a) u_n - \phi_{k_{n-1}}(a) \| < 1/2^n, \qquad a \in \mathcal F_n.
\]
Let $v_n = u_n u_{n-1} \cdots u_1$. We claim that $\psi(a) := \lim_{n\to \infty} v_n^\ast \phi_{k_n}(a) v_n$ is a well-defined $\ast$-homomorphism which is approximately unitary equivalent to $\phi$ as maps into $B_\infty$.

To see that the map is well-defined it suffices to check that $(v_n^\ast\phi_{k_n}(a) v_n)_{n\in \mathbb N}$ is a Cauchy sequence for every $a\in A$. Given $\epsilon>0$, pick an $n\in \mathbb N$ and $b\in \mathcal F_n$ such that $\|a-b\| < \epsilon/3$. For $m>l \geq n$ we have
\[
v_m^\ast \phi_{k_m}(b) v_m \approx_{2^{-m}} v_{m-1}^\ast \phi_{k_{m-1}}(b) v_{m-1} \approx_{2^{-m+1}} \dots \approx_{2^{-l-1}} v_l^\ast \phi_{k_l}(b) v_l,
\]
so $\| v_m^\ast \phi_{k_m} (b) v_m - v_l^\ast\phi_{k_l}(b) v_l\| < \sum_{k=l+1}^m 2^{-k} < \sum_{k=l+1}^\infty2^{-k}$. As $\limsup_{k\to \infty} \| \phi_k(a) - \phi_k(b) \| < \epsilon/3$, and since $k_l\to \infty$, we may pick $N\geq n$ such that $\| \phi_{k_l} (a) - \phi_{k_l}(b) \| < \epsilon /3$ and $\sum_{k=l+1}^\infty 2^{-k} < \epsilon /3$ for $l \geq N$. It easily follows that
\[
\| v_m^\ast \phi_{k_m} (a) v_m - v_l^\ast\phi_{k_l}(a) v_l \| < \epsilon
\]
for $m,l\geq N$, so $(v_m^\ast\phi_{k_m}(a) v_m)_{m\in \mathbb N}$ is a Cauchy sequence for every $a\in A$. Hence $\psi\colon A \to B$ is a well-defined map.

To see that $\psi$ is approximately unitary equivalent to $\phi$ in $B_\infty$, let $v \in B_\infty$ be the unitary induced by $(v_1,v_2,\dots)$ and $\eta \colon \mathbb N \to \mathbb N$ be the map $\eta(n) = k_n$. Then as maps into $B_\infty$
\[
\psi = v^\ast (\eta^\ast \circ \phi(-)) v,
\]
and thus $\psi$ and $\phi$ are approximately unitary equivalent in $B_\infty$ by assumption.
By Lemma \ref{l:approxinfty} they are also unitary equivalent. As $\psi$ composed with the embedding into $B_\infty$ is a $\ast$-homomorphism, so is $\psi$.
\end{proof}

The above theorem can be applied to give a nice McDuff type characterisation of $\mathcal O_2$- and $\mathcal O_\infty$-stable $\ast$-homomorphisms. 

Recall from \cite{Cuntz-K-theoryI} that $K_0(\mathcal O_2) \cong K_1(\mathcal O_2) \cong K_1(\mathcal O_\infty) \cong 0$, and that $K_0(\mathcal O_\infty) \cong \mathbb Z$ for which $[1_{\mathcal O_\infty}]_0$ is a generator. This is essentially computed by realising each $\mathcal O_n\otimes \mathbb K$ as a crossed product by an extendible endomorphism on an AF algebra, an argument that also implies that $\mathcal O_n$ satisfies the universal coefficient theorem (UCT) of Rosenberg and Schochet in \cite{RosenbergSchochet-UCT}.
In particular, it follows from the UCT that $\mathcal O_2$ is $KK$-equivalent to zero, and that the unital inclusion $\mathbb C \to \mathcal O_\infty$ is a $KK$-equivalence.

The following well-known consequence of Kirchberg's version of the Kirchberg--Phillips theorem is needed to give the McDuff type characterisation.\footnote{I emphasise that these results will not be needed when proving the classification of $\mathcal O_2$-stable $C^\ast$-algebras, Theorem \ref{t:O2class}. So overall, the proof of this main theorem is still pretty self-contained with the main exception of a result of Kirchberg and Rørdam \cite[Theorem 6.11]{KirchbergRordam-zero} used in Proposition \ref{p:Wliftcp}.}

\begin{proposition}\label{p:Dunitalembedding}
Let $\mathcal D$ be either $\mathcal O_2$ or $\mathcal O_\infty$. Any two unital embeddings of $\mathcal D$ are asymptotically unitary equivalent.
\end{proposition}
\begin{proof}
Let $\phi , \psi \colon \mathcal D \to C$ be unital embeddings. By replacing $C$ with $C^\ast(\phi(\mathcal D), \psi(\mathcal D))$ we may assume that $C$ is separable. It holds that $\phi$ and $\psi$ define the same element in $KK$-theory. In fact, the case $\mathcal D = \mathcal O_2$ follows since $\mathcal O_2$ is $KK$-equivalent to zero. In the case $\mathcal D = \mathcal O_\infty$, one uses that the unital inclusion is $\iota \colon \mathbb C \to \mathcal O_\infty$ is a $KK$-equivalence to note that $KK(\phi)= KK(\psi)$ if and only if $KK(\phi \circ \iota) = KK(\psi \circ \iota)$. The latter is obviously true since $\phi \circ \iota = \psi \circ \iota \colon \mathbb C \to C$, so $KK(\phi) = KK(\psi)$.

Hence by \cite[Proposition 2.8(i) and Theorem 2.9(ii)]{Dadarlat-htpyKirchbergalg} (relying on Kirchberg's version of the Kirchberg--Phillips theorem \cite{Kirchberg-simple}), $\phi$ and $\psi$ are asymptotically unitary equivalent.
\end{proof}

In the case $\mathcal D=\mathcal O_2$, Theorem \ref{t:O2HB} provides an alternative proof of the above result.

As a corollary, the following McDuff type characterisation of $\mathcal O_2$- and $\mathcal O_\infty$-stable $\ast$-homomorphisms is obtained.

\begin{corollary}\label{c:McDuff}
Let $A$ and $B$ be $C^\ast$-algebras with $A$ separable and for which $\multialg{B}$ is properly infinite, let $\phi \colon A \to B$ be a $\ast$-homomorphism, and let $\mathcal D$ be either $\mathcal O_2$ or $\mathcal O_\infty$. The following are equivalent.
\begin{itemize}
\item[$(i)$] $\phi$ is $\mathcal D$-stable,
\item[$(ii)$] there exists a $\ast$-homomorphism $\psi \colon A \otimes \mathcal D \to B$ such that $\phi$ and $\psi \circ (id_A \otimes 1_\mathcal D)$ are approximately Murray--von Neumann equivalent.
\end{itemize}
\end{corollary}
\begin{proof}
$(ii)\Rightarrow (i)$: Let $v\in B_\infty$ be such that $v^\ast \phi(a) v = \psi(a\otimes 1)$ and $v\psi(a\otimes 1) v^\ast = \phi(a)$. As $v\psi(a\otimes 1) = \phi(a) v$ for all $a\in A$ by Lemma \ref{l:conjv} it easily follows that the c.p.~map $v^\ast(-) v \colon B_\infty \to B_\infty$ restricts to $B_\infty \cap \phi(A)' \to B_\infty \cap \psi(A \otimes 1)'$, which in turn descends to an isomorphism
\[
 v^\ast(-) v \colon \frac{B_\infty \cap \phi(A)'}{\mathrm{Ann}(\phi(A))} \xrightarrow \cong \frac{B_\infty \cap \psi(A \otimes 1)'}{\mathrm{Ann}(\psi(A\otimes 1))}.
\]
The map $\psi \circ (1_A \otimes id_\mathcal D)$ induces a unital embedding of $\mathcal D$ in $B_\infty \cap \psi(A \otimes 1)'/\mathrm{Ann}(\psi(A \otimes 1))$, so it follows that $\phi$ is $\mathcal D$-stable.

$(i)\Rightarrow (ii)$: Fix isometries $s_1, s_2\in \multialg{B}$ with orthogonal range projections and $p= s_1s_1^\ast + s_2 s_2^\ast$. Then $s_1,s_2$ implement an isomorphism $pBp \cong M_2(B)$ such that $s_1 \phi(-) s_1^\ast \colon A \to pBp$ corresponds to the map $\phi \oplus 0$ via this identification. As $\phi$ and $s_1 \phi(-) s_1^\ast$ are approximately Murray--von Neumann equivalent, it suffices to construct $\psi \colon A \otimes \mathcal D \to M_2(B)$ such that $\phi \oplus 0$ and $\psi\circ (id_A \otimes 1)$ are approximately Murray--von Neumann equivalent.

Let $\theta \colon \mathcal D \to \frac{B_\infty \cap \phi(A)'}{\mathrm{Ann}(\phi(A))}$ be a unital embedding, and let $\eta \colon \mathbb N \to \mathbb N$ be a map such that $\lim_{n\to \infty} \eta(n) = \infty$. Then $\eta^\ast \colon B_\infty \to B_\infty$ induces a $\ast$-homomorphism
\[
 \eta^\ast \colon \frac{B_\infty \cap \phi(A)'}{\mathrm{Ann}(\phi(A))} \to \frac{B_\infty \cap \phi(A)'}{\mathrm{Ann}(\phi(A))}.
\]
By Proposition \ref{p:Dunitalembedding}, $\theta$ and $\eta^\ast \circ \theta$ are asymptotically unitary equivalent, so let $u_n \in \frac{B_\infty \cap \phi(A)'}{\mathrm{Ann}(\phi(A))}$ be a sequence of unitaries such that $u_n^\ast \theta(d) u_n \to \eta^\ast \circ \theta(d)$ for all $d\in \mathcal D$. Let $\overline \theta \colon A \to B_\infty \cap \phi(A)'$ be any map that lifts $\theta$, and $v_n \in B_\infty \cap \phi(A)'$ be a lift of $u_n$ for each $n$.

Let $\phi \times \theta \colon A \otimes \mathcal D \to B_\infty$ be the induced $\ast$-homomorphism given on elementary tensors by $\phi \times \theta(a\otimes x) = \phi(a) \overline \theta(x)$. Then 
\[
v_n^\ast (\phi\times \theta)(a \otimes d) v_n = \phi(a)v_n^\ast \overline \theta(d) v_n \to \phi(a) \eta^\ast(\overline \theta(d)) = \eta^\ast \circ (\phi \times \theta)(a\otimes d)
\]
for all $a\in A$ and $d\in \mathcal D$. Thus $\phi \times \theta$ and $\eta^\ast \circ (\phi \times \theta)$ are approximately Murray--von Neumann equivalent. By Proposition \ref{p:MvNeq}, $(\phi \times \theta) \oplus 0$ and 
\[
(\eta^\ast \circ ( \phi \times \theta)) \oplus 0 = \eta^\ast \circ ((\phi \times \theta) \oplus 0)
\]
are approximately unitary equivalent with unitaries in $(M_2( B)_\infty)^\sim \subseteq (M_2(B)^\sim)_\infty$. By Theorem \ref{t:lifthom} there is a $\ast$-homomorphism $\psi \colon A \otimes \mathcal D \to M_2(B)$ which is approximately unitary equivalent to $(\phi \times \theta) \oplus 0$. Thus the result follows.
\end{proof}

If $\phi \colon A \to B$ is $\mathcal O_\infty$-stable it is not hard to see that for any $a\in A_+$ the element $\phi(a)$ is properly infinite in $B$ in the sense of Kirchberg and Rørdam \cite{KirchbergRordam-purelyinf}. The following is a special case of \cite[Question 3.4]{KirchbergRordam-purelyinf}. An affirmative answer would imply that the requirement that $\multialg{B}$ is properly infinite is not needed in the above corollary.

\begin{question}
Let $A$ be separable and let $\phi \colon A \to B$ be an $\mathcal O_\infty$-stable $\ast$-homomorphism. Is $\multialg{\overline{\phi(A)B\phi(A)}}$ properly infinite?
\end{question}


\section{Existence results for generalised $Cu$-morphisms}\label{s:W}
 
 In this section existence results are produced for lifting generalised $Cu$-morphisms on ideal lattices of $C^\ast$-algebras to c.p.~maps. The main idea is to use Michael's selection theorem to produce large enough c.p.~maps $\phi$ into commutative $C^\ast$-algebras which preserve the ideal structure. This general method was first used by Blanchard \cite{Blanchard-deformations} and has later been used by Harnisch and Kirchberg \cite{HarnischKirchberg-primitive} to produce similar results.
 
 Basically everything in this section is contained in \cite{Gabe-cplifting} but in the language of actions on $C^\ast$-algebras instead of generalised $Cu$-morphisms. I include self-contained proofs here for the sake of completeness.
 
 An important tool is a version of Michael's selection theorems \cite[Theorem 1.2]{Michael-aselectionthm} which is a slight variation of his iconic selection theorem \cite[Theorem 3.2'']{Michael-selection}. I will recall the statement in the special case which will be needed, as well as the required terminology. 

Let $Y$ and $Z$ be topological spaces. A \emph{carrier} from $Y$ to $Z$ is a map $\Gamma \colon Y \to 2^Z$ where $2^Z$ denotes the set of non-empty subsets of $Z$. The purpose of Michael's selection theorems is to find \emph{continuous selections} of carriers $\Gamma$, i.e.~continuous maps $\gamma \colon Y \to Z$ such that $\gamma(y) \in \Gamma(y)$ for all $y\in Y$. 
 
A carrier $\Gamma$ is called \emph{lower semicontinuous} if for every $U \subseteq Z$ open, the set
\[
\{ y \in Y : \Gamma(y) \cap U \neq \emptyset\}
\]
is an open subset of $Y$. The following is a special case of \cite[Theorem 1.2]{Michael-aselectionthm}.\footnote{This follows immediately since the closed unit ball of the dual of a separable Banach space is compact and metrisable in the weak$^\ast$ topology.}

\begin{theorem}[Michael's selection theorem]\label{t:Michael}
Let $Y$ be a compact Hausdorff space, let $A$ be a separable Banach space and equip the dual space $A^\ast$ with the weak$^\ast$ topology. Let $\Gamma \colon Y \to 2^{A^\ast}$ be a lower semicontinuous carrier such that $\Gamma(y)$ is a closed convex subset of the closed unit ball of $A^\ast$ for each $y\in Y$. Then there exists a continuous selection of $\Gamma$, i.e.~there is a weak$^\ast$ continuous map $\gamma \colon Y \to A^\ast$ such that $\gamma(y) \in \Gamma(y)$ for all $y\in Y$.
\end{theorem}
 
The sets $P(A) \subseteq QS(A) \subseteq A^\ast$ are the sets of pure states, quasi-states\footnote{A \emph{quasi-state} is a positive linear functional of norm at most $1$.} and the dual space of $A$ respectively, all equipped with the weak$^\ast$ topology.
 
 \begin{lemma}\label{l:lsc}
 Let $A$ be a $C^\ast$-algebra, let $Y$ be a compact Hausdorff space, and let $\Phi \colon \mathcal I(A) \to \mathcal I(C(Y))$ be a generalised $Cu$-morphism. For every $y\in Y$, define 
 \[
  I_{\Phi,y} := \overline{\sum_{I \in \mathcal I(A): \Phi(I) \subseteq C_0(Y\setminus\{y\})} I} \in \mathcal I(A).
 \]
The carrier $\Gamma \colon Y \to 2^{A^\ast}$ given by
\[
  \Gamma(y) = \{ \eta \in QS(A) : \eta(I_{\Phi,y}) = 0\}, \qquad y\in Y,
\]
is lower semicontinuous.
 \end{lemma}
 \begin{proof}
 Define the carrier $\Gamma' \colon Y \to 2^{A^\ast}$ by
 \[
 \Gamma' (y) = \{ \eta \in P(A) \cup \{0\} : \eta(I_{\Phi,y}) = 0\} , \qquad y\in Y.
 \]
Note that $\Gamma(y)$ (resp.~$\Gamma'(y)$) can be naturally identified with the set of states (resp.~pure states) in the forced unitisation $(A/I_{\Phi,y})^\dagger$. Hence $\Gamma(y)$ is the weak$^\ast$ closure of the convex hull of $\Gamma'(y)$ for each $y\in Y$. By \cite[Propositions 2.3 and 2.6]{Michael-selection} it thus follows that $\Gamma$ is lower semicontinuous if $\Gamma'$ is lower semicontinuous. Thus it suffices (in order to finish the proof) to show that $\Gamma'$ is lower semicontinuous.
 
Let $U\subseteq A^\ast$ be open. Suppose that $0\in U$. As $0\in \Gamma'(y)$ for every $y$, it follows that
\[
\{ y \in Y : \Gamma'(y) \cap U \neq \emptyset\} = Y
\]
which is open. So suppose that $0 \notin U$. By \cite[Theorem 4.3.3]{Pedersen-book-automorphism}, the continuous map $F \colon P(A) \to \Prim A$, $F(\eta) = \ker \pi_ \eta$ is open where $\pi_\eta$ is the GNS representation of $\eta$. So there is an induced map $\mathbb O(P(A)) \xrightarrow{F} \mathbb O(\Prim A)$. Also, there is a canonical order isomorphism $\mathbb O(\Prim A) \cong \mathcal I(A)$ given by $U \mapsto \bigcap_{\mathfrak p \in \Prim A\setminus U} \mathfrak p$. The inverse map is given by $I \mapsto \{ \mathfrak p \in \Prim A : I \nsubseteq \mathfrak p\}$. Let $\Psi$ be the composition
\[
\mathbb O(P(A)) \xrightarrow{F} \mathbb O(\Prim A) \xrightarrow \cong \mathcal I(A) \xrightarrow \Phi \mathcal I(C(Y)) \cong \mathbb O(Y).
\]
We claim that
\[
 \{ y \in Y : \Gamma'(y) \cap U \neq \emptyset\} = \Psi ( U \cap P(A)),
\]
which will imply that $\Gamma'$ is lower semicontinuous as $\Psi(U \cap P(A)) \subseteq Y$ is open. 

As generalised $Cu$-morphisms of ideal lattices preserve arbitrary suprema by Remark \ref{r:sup}, it follows that $\Phi(I_{\Phi,y}) \subseteq C_0(Y\setminus\{y\})$, and that $I_{\Phi,y}$ is the \emph{largest} ideal in $A$ with this property. Hence, for any $I\in \mathcal I(A)$ we have $I \subseteq I_{\Phi,y}$ if and only if $\Phi(I) \subseteq C_0(Y\setminus \{y\})$. Thus we get the following chain of reasoning for $y\in Y$:
\begin{eqnarray*}
y \notin \Psi(U\cap P(A)) &\Leftrightarrow& \Phi\left( \bigcap_{\mathfrak p \in \Prim A \setminus F(U\cap P(A))} \mathfrak p \right) \subseteq C_0(Y\setminus \{y\}) \\
&\Leftrightarrow&  \bigcap_{\mathfrak p \in \Prim A \setminus F(U\cap P(A))} \mathfrak p \subseteq I_{\Phi,y} \\
&\Leftrightarrow& F(U\cap P(A)) \subseteq \{ \mathfrak p\in \Prim A : I_{\Phi,y} \nsubseteq \mathfrak p \} \\
& \Leftrightarrow & \textrm{ for every }\eta\in U\cap P(A) \textrm{ we have } I_{\Phi,y} \nsubseteq \ker \pi_\eta \\
& \Leftrightarrow & \textrm{ for every }\eta\in U\cap P(A) \textrm{ we have } \eta(I_{\Phi,y}) \neq 0 \\
&\Leftrightarrow& \Gamma'(y) \cap U \cap P(A) = \emptyset \\
&\Leftrightarrow& \Gamma'(y) \cap U = \emptyset.
\end{eqnarray*}
This means that
\[
\{ y\in Y : \Gamma'(y) \cap U \neq \emptyset \} = \Psi(U \cap P(A)).
\]
As $\Psi(U \cap P(A))$ is open, $\Gamma'$ is lower semicontinuous. This finishes the proof since implies that $\Gamma$ is lower semicontinuous (as seen early on in the proof).
\end{proof}

\begin{lemma}\label{l:selcpc}
With the same setup as in Lemma \ref{l:lsc}, suppose that $\gamma\colon Y \to A^\ast$ is a continuous selection of $\Gamma$. Then the map $\phi \colon A \to C(Y)$ given by
\[
\phi(a)(y) = \gamma(y)(a), \qquad a\in A, y\in Y,
\]
is a contractive c.p.~map satisfying $\mathcal I(\phi) \leq \Phi$.
\end{lemma}
\begin{proof}
Recall that positive $\ast$-linear maps into $C(Y)$ are completely positive.\footnote{This follows as an element in $C(Y,M_n)$ is positive exactly when its evaluation in $y$ is positive for every $y\in Y$, and since positive linear functionals are completely positive, see \cite[Example 1.5.2]{BrownOzawa-book-approx}.} As $\gamma$ is weak$^\ast$ continuous and each $\gamma(y)$ has norm at most $1$, it follows that $y \mapsto \gamma(y)(a)$ is a continuous map of norm at most $\| a\|$ for every $a\in A$, and thus the map $\hat \phi \colon A \to C(Y)$ given by $\hat \phi(a) (y) = \gamma(y)(a)$ is a contractive c.p.~map. 

To see that $\mathcal I(\phi) \leq \Phi$, fix $I\in \mathcal I(A)$. We want to show that $\phi(I) \subseteq \Phi(I)$. Let $U_I \subseteq Y$ be the open subset such that $\Phi(I) = C_0(U_I)$. As $\Phi(I) = \bigcap_{y\notin U_I} C_0(Y\setminus \{y\})$, it suffices to show that $\phi(I) \subseteq C_0(Y\setminus\{y\})$ for any $y \in Y\setminus U_I$. Fix such a $y$.

As $\Phi(I) \subseteq C_0(Y\setminus\{y\})$ it follows that $I \subseteq I_{\Phi , y}$ by how $I_{\Phi,y}$ was constructed. Thus, for any $a\in I \subseteq I_{\Phi,y}$ we have $\phi(a)(y) = \gamma(y)(a) = 0$. Hence $\phi(I) \subseteq C_0(Y\setminus \{y\})$, so $\mathcal I(\phi) \leq \Phi$.
\end{proof}

\begin{lemma}\label{l:selcom}
Let $A$ be a separable $C^\ast$-algebra, let $C$ be a separable, commutative $C^\ast$-algebra and let $\Phi \colon \mathcal I(A) \to \mathcal I(C)$ be a generalised $Cu$-morphism. Then there exists a c.p.~map $\phi \colon A \to C$ such that $\mathcal I(\phi) = \Phi$.
\end{lemma}
\begin{proof}
If $C$ is not unital, let $\iota \colon C \to \widetilde C$ be the inclusion. If we can find a c.p.~map $\overline \phi \colon A \to \widetilde C$ such that $\mathcal I(\phi) = \mathcal I(\iota) \circ \Phi$, then $\phi$ corestricts to a c.p.~map $\phi \colon A \to C$ such that $\mathcal I(\phi) = \Phi$. Hence we may assume that $C$ is unital. We may assume that $C = C(Y)$ for a compact, metrisable space $Y$. 

We use the same notation as in Lemma \ref{l:lsc}. Since $A$ is separable, it follows from \cite[Corollary 4.3.4]{Pedersen-book-automorphism} that $\mathcal I(A)$ has a countable basis $(I_m)_{m\in \mathbb N}$. Let $U_m$ be the open subset of $Y$ such that $\Phi(I_m) = C_0(U_m)$. Note that $I_m \not \subseteq I_{\Phi, y}$ whenever $y\in U_m$. In fact, as $\Phi$ preserves suprema, it follows from the definition of $I_{\Phi,y}$ that $\Phi(I_{\Phi,y}) \subseteq C_0(Y\setminus\{y\})$. Thus, if $I_m \subseteq I_{\Phi, y}$ then 
\[
\Phi(I_m) \subseteq C_0(U_m) \cap C_0(Y \setminus \{y\}) = C_0(U_m \setminus \{y\})
\]
which is a contradiction if $y\in U_m$. 

Hence, for every pair $(y,m)$ such that $\Phi(I_m) \neq 0$ and $y\in U_m$, we may pick a positive contraction $a_{y,m} \in I_m$ such that $\| a_{y,m} + I_{\Phi,y} \| = 1$. Moreover, we may fix a quasi-state $\eta_{y,m} \in \Gamma(y)$ such that $\eta_{y,m}(a_{y,m}) = 1$. 

For each such pair $(y,m)$ where $\Phi(I_m) \neq 0$ and $y\in U_m$, we construct carriers $\Gamma_{y,m} \colon Y \to 2^{A^\ast}$ by
\[
\Gamma_{y,m}(z) = \left\{ \begin{array}{ll} \{ \eta_{y,m}\} , \textrm{ if }z = y \\ \Gamma(z) , \textrm{ otherwise.} \end{array} \right.
\]
As $\Gamma$ is lower semicontinuous by Lemma \ref{l:lsc}, it follows from \cite[Example 1.3*]{Michael-selection} that $\Gamma_{y,m}$ is lower semicontinuous. By Michael's selection theorem, Theorem \ref{t:Michael}, there are continuous maps $\gamma_{y,m} \colon Y \to A^\ast$ for each $y,m\in \mathbb N$ such that $\gamma_{y,m}(z) \in \Gamma_{y,m}(z)$ for every $z\in Y$. As each $\gamma_{y,m}$ is also a selection for $\Gamma$, it follows from Lemma \ref{l:selcpc} that we may construct contractive c.p.~maps $\phi_{y,m} \colon A \to C(Y)$ by
\[
\phi_{y,m}(a)(z) = \gamma_{y,m}(z)(a), \qquad a\in A, z\in Y,
\]
which satisfy $\mathcal I(\phi_{y,m}) \leq \Phi$. Note that $\phi_{y,m}(a)(y) = \eta_{y,m}(a)$ for every $a\in A$.

For each $(y,m)$ as above, the set $V_{y,m} := \{ x\in Y :  \phi_{y,m}(a_{y,m})(x) > 0 \}$ is an open neighbourhood of $y$ by construction. If $x \notin U_m$, then $\Phi(I_m) = C_0(U_m) \subseteq C_0(Y\setminus \{x\})$, so $I_m \subseteq I_{\Phi, x}$ by definition of $I_{\Phi,x}$. Hence $a_{y,m} \in I_{\Phi,x}$, so $\phi_{y,m}(a_{y,m})(x) = 0$ whenever $x\notin U_m$. Thus $V_{y,m} \subseteq U_m$, so $(V_{y,m})_{y\in U_m}$ is an open cover of $U_m$.

As $U_m$ is $\sigma$-compact, we may find a countable sequence $(y_{(n,m)})_{n\in \mathbb N}$ in $U_m$ such that $(V_{y_{(n,m)},m})_{n\in \mathbb N}$ is an open cover of $U_m$.

Let 
\[
\phi = \sum_{n,m\in \mathbb N, \Phi(I_m) \neq 0} 2^{-n-m} \phi_{n,m}  \qquad \textrm{(point-wise convergence)}.
\]
 We claim that $\mathcal I(\phi) = \Phi$. As $\mathcal I(\phi_{n,m}) \leq \Phi$ it easily follows that $\mathcal I(\phi) \leq \Phi$, so it remains to prove the other inequality.

As $\Phi$ and $\mathcal I(\phi)$ are both generalised $Cu$-morphisms they both preserve suprema, so it suffices to show that $\Phi(I_m) = C_0(U_m) \subseteq \mathcal I(\phi)(I_m)$ for each $m\in \mathbb N$.

If $\Phi(I_m) = 0$ this is obvious, so we assume that $\Phi(I_m) \neq 0$. Let $W_m \subseteq Y$ be the open set such that $\mathcal I(\phi)(I_m) = C_0(W_m)$, and recall that
\[
 W_m = \{ y \in Y : f(y) \neq 0 \textrm{ for some } f\in \mathcal I(\phi)(I_m)\}.
\]
For $y\in U_m$ we will find $f\in \mathcal I(\phi)(I_m)$ such that $f(y) \neq 0$ which implies $y\in W_m$. Pick an $n$ such that $y \in V_{y_{(n,m)},m}$. In particular, $a_{y_{(n,m)} , m}\in I_m$ satisfies
\[
\phi(a_{y_{(n,m)} , m})(y) \geq 2^{-n-m} \phi_{n,m} (a_{y_{(n,m)} , m})(y) > 0.
\]
Hence $f = \phi(a_{y_{(n,m)}, m}) \in \mathcal I(\phi)(I_m)$ satisfies $f(y) \neq 0$, so $y\in W_m$. This implies that $U_m \subseteq W_m$ and thus $\Phi(I_m) = C_0(U_m) \subseteq C_0(W_m) =\mathcal I(\phi)(I_m)$, which completes the proof.
\end{proof}

\begin{proposition}\label{p:Wliftcp}
Let $A$ and $B$ be separable $C^\ast$-algebras with $B$ nuclear, and let $\Phi \colon \mathcal I(A) \to \mathcal I(B)$ be a generalised $Cu$-morphism. Then there exists a c.p.~map $\phi \colon A \to B$ such that $\mathcal I(\phi) = \Phi$.
\end{proposition}
\begin{proof}
Say that $B$ has \emph{Property $(\Diamond)$} if the following holds: \emph{there exist a separable, commutative $C^\ast$-algebra $C$, a c.p.~map $\psi \colon C \to B$, and a generalised $Cu$-morphism $\Psi \colon \mathcal I(B) \to \mathcal I(C)$, such that $\mathcal I(\psi) \circ \Psi = id_{\mathcal I(B)}$}.

 We will finish the proof assuming $B$ has Property $(\Diamond)$ and afterwards apply a result of Kirchberg and Rørdam \cite[Theorem 6.11]{KirchbergRordam-zero} to conclude that any separable, nuclear $C^\ast$-algebra $B$ has this property.

So suppose that $B$ has Property $(\Diamond)$ and let $C$, $\psi$ and $\Psi$ be given as above. As $\Psi \circ \Phi \colon \mathcal I(A) \to \mathcal I(C)$ is a generalised $Cu$-morphism we may apply Lemma \ref{l:selcom} to obtain a c.p.~map $\phi_0 \colon A \to C$ such that $\mathcal I(\phi_0) = \Psi \circ \Phi$. Let $\phi = \psi \circ \phi_0$. By Lemma \ref{l:idealcom} we get
\[
\mathcal I(\phi) = \mathcal I(\psi) \circ \mathcal I(\phi_0) = \mathcal I(\psi) \circ \Psi \circ \Phi = \Phi.
\]
It remains to show that any separable, nuclear $B$, has Property $(\Diamond)$. By \cite[Theorem 6.11]{KirchbergRordam-zero}, $B\otimes \mathcal O_2$ contains a commutative $C^\ast$-subalgebra which separates the ideals of $B\otimes \mathcal O_2$ and such that $(I\cap C) + (J\cap C) = (I+J)\cap C$ for every $I,J\in \mathcal I(B\otimes \mathcal O_2)$. It follows that the map $\Psi' \colon \mathcal I(B\otimes \mathcal O_2) \to \mathcal I(C)$ given by $\Psi'(J) = J\cap C$ is a generalised $Cu$-morphism.\footnote{For any $C^\ast$-subalgebra $C\subseteq B$, the map $J \mapsto J \cap C$ will preserve zero, order and increasing suprema, but it will in general not be additive.} Let $\iota \colon C \hookrightarrow B \otimes \mathcal O_2$ be the inclusion. As $C$ separates the ideals it follows that $\mathcal I(\iota) \circ \Psi' = id_{\mathcal I(B\otimes \mathcal O_2)}$. 

Let $\eta$ be a faithful state on $\mathcal O_2$ and $\lambda_\eta \colon B\otimes \mathcal O_2 \to B$ be the induced slice map. Clearly $\mathcal I(\lambda_\eta)$ is the inverse of the isomorphism $\Theta \colon \mathcal I(B) \xrightarrow \cong \mathcal I(B\otimes \mathcal O_2)$, $\Theta(I) = I\otimes \mathcal O_2$. Let $\Psi := \Psi' \circ \Theta$ and $\psi := \lambda_\eta \circ \iota$. Then
\[
\mathcal I(\psi) \circ \Psi \stackrel{\textrm{Lem.}~\ref{l:slice}}{=} \mathcal I(\lambda_\eta) \circ \mathcal I(\iota) \circ \Psi' \circ \Theta = id_{\mathcal I(B)}.
\]
Hence $B$ has Property $(\Diamond)$ thus finishing the proof.
\end{proof}

\begin{remark}
It would be desirable to have a more elementary proof of the above proposition without having to go through the deep structural result of Kirchberg and Rørdam. As emphasised in the proof, $B$ does not need to be separable and nuclear for the proof to work. What is really needed is that $B$ has Property $(\Diamond)$ as defined in the proof above.

By the Dauns--Hofmann theorem any $\sigma$-unital $B$ with $\Prim B$ second countable, Hausdorff has Property $(\Diamond)$. This easily follows by letting $C = C_0(\Prim B)$, $\psi = h \iota (-) h$, and $\Psi = \mathcal I(\psi)^{-1}$, where $\iota \colon  C_0(\Prim B) \to \multialg{B}$ is the canonical $\ast$-homomorphism coming from the Dauns-Hofmann theorem, and where $h\in B$ is a strictly positive element.

Similarly, if $\Prim B$ is second countable and zero dimensional (i.e.~has a basis of compact open sets), then it is not hard to apply the construction of \cite{BratteliElliott-structurespacesII} to find an AF $C^\ast$-subalgebra $D \subseteq B \otimes \mathcal O_2$ such that the inclusion induces an isomorphism $\mathcal I(D) \cong \mathcal I(B\otimes \mathcal O_2)$. One easily sees that AF algebras have Property $(\Diamond)$ so by slicing away $\mathcal O_2$ using Lemma \ref{l:slice} it follows that any such $B$ has Property $(\Diamond)$. Hence, in these cases one does not have to rely on the result of Kirchberg and Rørdam.
\end{remark}

For the final result of this section, recall (a possible construction of) the Kasparov--Stinespring dilation, cf.~\cite[Theorem 3]{Kasparov-Stinespring}. Given a contractive c.p.~map $\phi \colon A \to B$, where $A$ is separable and $B$ is $\sigma$-unital and stable, one constructs the (countably generated, right) Hilbert $B$-module $E:=\widetilde A \otimes_{\widetilde \phi} B$ where $\widetilde \phi \colon \widetilde A \to \multialg{B}$ is the minimal unitisation. There is an induced $\ast$-homomorphism $\phi'_0 \colon A \to \mathbb B(E) \subseteq \mathbb B(E\oplus B)$ given by multiplication on the left tensor of $E$. As $B$ is stable, $B \otimes \ell^2(\mathbb N) \cong B$ as Hilbert $B$-modules. Thus, by Kasparov's stabilisation theorem \cite[Theorem 2]{Kasparov-Stinespring}, there is a unitary $u\in \mathbb B(B, E\oplus B)$. Define $\phi_0 \colon A \to \mathbb B(B) = \multialg{B}$ by $\phi_0 := u^\ast \phi'_0(-) u$ which is a $\ast$-homomorphism. By letting $W\in \mathbb B(B, E\oplus B)$ be given by $W (b) = (1\otimes b,0)$, and $V:= u^\ast W\in \multialg{B}$, one obtains $V^\ast \phi_0(-) V = \phi$. I will refer to $(\phi_0, V)$ as a \emph{Kasparov--Stinespring dilation} of $\phi$. 

\begin{corollary}\label{c:multimorphism}
Let $A$ and $B$ be separable $C^\ast$-algebras with $B$ nuclear and stable, and let $\Phi \colon \mathcal I(A) \to \mathcal I(B)$ be a generalised $Cu$-morphism. Then there exists a $\ast$-homomorphism $\phi_0 \colon A \to \multialg{B}$ such that $\Phi(I) = \overline{B \phi_0(I) B}$ for all $I\in \mathcal I(A)$.
\end{corollary}
\begin{proof}
As $B$ is separable and nuclear, Proposition \ref{p:Wliftcp} implies the existence of a c.p.~map $\phi \colon A \to B$ for which $\mathcal I(\phi) = \Phi$.\footnote{This is the only place where separability and nuclearity of $B$ is used, although we do need $\sigma$-unitality to apply the Kasparov--Stinespring dilation. So if one can construct a c.p.~map $\phi \colon A \to B$ for which $\mathcal I(\phi) = \Phi$, separability and nuclearity of $B$ in the statement of the corollary can be replaced with $\sigma$-unitality of $B$.} We may assume that $\phi$ is contractive. Let $(\phi_0,V)$ be a Kasparov--Stinespring dilation as described above. For $I\in \mathcal I(A)$ we have
\[
\Phi(I) = \overline{B \phi(I) B} = \overline{BV^\ast\phi_0(I)VB} \subseteq \overline{B \phi_0(I) B}.
\]
We wish to show that $\overline{B \phi_0(I) B} \subseteq \Phi(I)$. As the canonical inner product on $\overline{\phi_0(I) B}$ is given by $\langle a,b\rangle = a^\ast b$, we have $\overline{B \phi_0(I) B} = \langle \overline{\phi_0(I)B}, \overline{\phi_0(I)B} \rangle$, so it suffices to show that $\langle \overline{\phi_0(I)B}, \overline{\phi_0(I)B} \rangle \subseteq \Phi(I)$. Since
\[
\langle \overline{\phi_0(I)B}, \overline{\phi_0(I) B} \rangle = \langle \overline{\phi_0'(I) E}\oplus 0, \overline{\phi_0'(I) E} \oplus 0\rangle,
\]
it suffices to check that $\langle y \oplus 0, \phi_0'(x) y \oplus 0 \rangle \in \Phi(I)$ for any $x\in I$, and any $y\in E$ induced from the algebraic tensor product $\widetilde A \otimes_\mathbb{C} B$. So let $a_1,\dots, a_n\in \widetilde A$ and $b_1,\dots, b_n\in B$ be such that $y$ is induced by $\sum a_i \otimes b_i$. Then
\[
\langle y\oplus 0, \phi_0'(x) y \oplus 0\rangle = \sum_{i,j=1}^n b_i^\ast \phi(a_i x a_j) b_j \in \overline{B \phi(I) B} = \mathcal I(\phi)(I) = \Phi(I).\qedhere
\]
\end{proof}


\section{An ideal related $\mathcal O_2$-embedding theorem}\label{s:O2}

The main goal of this section is to prove the following ideal related $\mathcal O_2$-embedding result of Kirchberg, which essentially should be considered as an existence result.

\begin{theorem}\label{t:O2embedding}
Let $A$ be a separable, exact $C^\ast$-algebra, let $B$ be a separable, nuclear, $\mathcal O_\infty$-stable $C^\ast$-algebra, and let $\Phi \colon \mathcal I(A) \to \mathcal I(B)$ be a $Cu$-morphism. Then there exists a strongly $\mathcal O_2$-stable $\ast$-homomorphism $\phi \colon A \to B$ such that $\mathcal I(\phi) = \Phi$.
\end{theorem}

\begin{remark}
If in the above theorem one takes $B= \mathcal O_2$ and $\Phi$ to be the map $\Phi(I) = \mathcal O_2$ for $I\neq 0$ and $\Phi(0) = 0$, then one obtains Kirchberg's classical $\mathcal O_2$-embedding theorem: for any separable, exact $C^\ast$-algebra $A$ there is an injective $\ast$-homomorphism $A \hookrightarrow \mathcal O_2$.
\end{remark}

\begin{remark}
Kirchberg and Rørdam have shown in \cite{KirchbergRordam-absorbingOinfty} that a separable, nuclear $C^\ast$-algebra is $\mathcal O_\infty$-stable if and only if it is \emph{strongly purely infinite}.\footnote{Technically one should also use \cite[Proposition 4.4(4,5)]{Kirchberg-Abel} or \cite[Corollary 3.2]{TomsWinter-ssa} to reduce from the stable case to the general case.} As strong pure infiniteness in general is weaker than $\mathcal O_\infty$-stability, it would be more natural to replace $\mathcal O_\infty$-stability in the above theorem with strong pure infiniteness. However, to keep the proof more self contained the results will be stated with $\mathcal O_\infty$-stability instead.
\end{remark}

\begin{remark}
Here is the main idea of how the proof of Theorem \ref{t:O2embedding} goes: 

Apply Corollary \ref{c:multimorphism} to find a $\ast$-homomorphism $\phi_0 \colon A \to \multialg{B}$ such that $\overline{B\phi_0(I) B} = \Phi(I)$ for all $I\in \mathcal I(A)$. We pick a suitably well-behaved positive element in $B_\infty \cap \phi_0(A)'$ with spectrum $[0,1]$. By considering $\multialg{B} \subseteq \multialg{B}_\infty$ and $B_\infty \subseteq \multialg{B}_\infty$, there is an induced $\ast$-homo\-morphism $C_0(0,1) \otimes A \to \multialg{B}_\infty$ which factors through $B_\infty$. The uniqueness result Theorem \ref{t:O2HB} will be used to show that there is a unitary in $\multialg{B}_\infty$ implementing a certain automorphism $\alpha$ on $C_0(0,1) \otimes A$, so there is an induced $\ast$-homomorphism $(C_0(0,1) \otimes A)\rtimes_\alpha \mathbb Z \to B_\infty$. The automorphism $\alpha$ is chosen such that the crossed product is isomorphic to $C(\mathbb T) \otimes \mathbb K \otimes A$. As $A$ embeds into $C(\mathbb T)\otimes \mathbb K \otimes A$, this will produce a $\ast$-homomorphism $\psi \colon A\to B_\infty$.\footnote{This is the same trick as in the proof of the $\mathcal O_2$-embedding theorem due to Kirchberg and Phillips in \cite{KirchbergPhillips-embedding}.} One now combines the uniqueness result Theorem \ref{t:O2HB} with Theorem \ref{t:lifthom} to produce a $\ast$-homomorphism $\phi \colon A \to B$ which is unitary equivalent to $\psi$ in $B_\infty$. This $\phi$ will satisfy $\mathcal I(\phi) = \Phi$.
\end{remark}

Obviously compact containment of ideals must play an important part in the proof of Theorem \ref{t:O2embedding}, as any generalised $Cu$-morphism which lifts to a $\ast$-homomorphism must necessarily preserve compact containment by Lemma \ref{l:idealmorphisms}$(iii)$. The following proposition says that one can compute certain ideals in $B_\infty$ by using compact containment whenever $B$ is weakly purely infinite. 

Recall from \cite{BlanchardKirchberg-Hausdorff},\footnote{This differs slightly from the definition of $n$-purely infinite in \cite{KirchbergRordam-absorbingOinfty}, but the definitions of weakly purely infinite are still the same by \cite[Proposition 4.12]{BlanchardKirchberg-Hausdorff}.} that a $C^\ast$-algebra $B$ is \emph{$n$-purely infinite} for $n\in \mathbb N$ if $\ell^{\infty}(B)$ has no quotients of dimension $\leq n^2$, and if for any positive $a,b\in B$ such that $a\in \overline{BbB}$, and any $\epsilon >0$, there are $d_1,\dots, d_n \in B$ such that $\| a - \sum_{k=1}^n d_k^\ast b d_k \| < \epsilon$. A $C^\ast$-algebra is \emph{weakly purely infinite} if it is $n$-purely infinite for some $n$, and is \emph{purely infinite} if it is $1$-purely infinite. 

For any $C^\ast$-algebra $B$, $B\otimes \mathcal O_2$ and $B\otimes \mathcal O_\infty$ are purely infinite by \cite[Theorem 5.11]{KirchbergRordam-purelyinf}.

A similar statement as the one below (with virtually the same proof) holds for the sequence algebra $B_\omega$ for any free filter $\omega$ on $\mathbb N$.

\begin{proposition}\label{p:idealultra}
Let $B$ be a weakly purely infinite $C^\ast$-algebra and let $I$ be an ideal of $B$. Then
\[
\overline{B_\infty I B_\infty} = \overline{\bigcup_{J \cpct I} J_\infty}
\]
where the union above is taken over ideals $J$ in $B$ which are compactly contained in $I$.
\end{proposition}
\begin{proof}
``$\subseteq$'': To show $\overline{B_\infty I B_\infty} \subseteq \overline{\bigcup_{J \cpct I} J_\infty}$ it suffices to check that whenever $c\in I$ is positive and $\epsilon >0$, then $(c-\epsilon)_+ \in \overline{\bigcup_{J \cpct I} J_\infty}$. However, this is obvious as $\overline{B(c-\epsilon)_+ B} \cpct \overline{BcB} \subseteq I$ by Lemma \ref{l:cpct}, so
\[
(c-\epsilon)_+ \in (\overline{B (c-\epsilon)_+ B})_\infty \subseteq \overline{\bigcup_{J \cpct I} J_\infty}.
\]

``$\supseteq$'': To show $\overline{B_\infty I B_\infty} \supseteq \overline{\bigcup_{J \cpct I} J_\infty}$, it suffices to check that $J_\infty \subseteq \overline{B_\infty I B_\infty}$ for every ideal $J \cpct I$. Fix such a $J$. By Lemma \ref{l:cpct} there is a positive $c\in I$ and $\epsilon >0$, such that $J \subseteq \overline{B(c-\epsilon)_+ B}$. 
Let $x\in J_\infty$ be a positive contraction and let $(x_n)_{n\in \mathbb N} \in \prod_{\mathbb N} J$ be a lift of $x$ for which each $x_n$ is a positive contraction. As $B$ is weakly purely infinite it is $m$-purely infinite for some $m\in \mathbb N$. Thus, as each $x_n \in \overline{B(c-\epsilon)_+B}$, we may find $d_n^{(1)},\dots, d_n^{(m)}\in B$ such that
\[
\| x_n - \sum_{k=1}^m d_n^{(k)\ast} (c-\epsilon)_+ d_n^{(k)}\| < 1/n, \qquad \textrm{ for all }n\in \mathbb N.
\]
Let $y_n^{(k)} = (c-\epsilon)_+^{1/2} d_n^{(k)}$ for $n\in \mathbb N$, $k=1,\dots,m$. These elements are clearly bounded, as $\|x_n\| \leq 1$ and $\sum_{k=1}^m y_n^{(k)\ast}y_n^{(k)}$ is $1/n$-close to $x_k$. Let $y^{(k)}$ be the image of $(y_n^{(k)})_{n\in \mathbb N}$ in $B_\infty$. Pick an element $c_0 \in C^\ast(c)$ such that $c_0 (c-\epsilon)_+ = (c-\epsilon)_+$. Then
\[
x = \sum_{k=1}^m y^{(k)\ast} y^{(k)} = \sum_{k=1}^m y^{(k)\ast} c_0 y^{(k)} \in \overline{B_\infty c B_\infty} \subseteq \overline{B_\infty I B_\infty}.
\]
As $x\in J_\infty$ was an arbitrary positive contraction it follows that $J_\infty \subseteq \overline{B_\infty I B_\infty}$.
\end{proof}

\begin{lemma}\label{l:fullmulti}
Let $A$ and $B$ be $C^\ast$-algebras with $B$ stable, let $\phi \colon A \to B_\infty$ be a contractive c.p.~map, and let $\tilde \phi \colon \tilde A \to \multialg{B}_\infty$ be the induced unital c.p.~map. For $I\in \mathcal I(\tilde A)$ we have
\[
\mathcal I(\tilde \phi)(I) = \left\{ \begin{array}{ll} \mathcal I(\phi)(I) , & \text{ if } I \subseteq A \\ \multialg{B}_\infty, & \text{ otherwise.} \end{array}\right.
\]
\end{lemma}
\begin{proof}
Clearly $\mathcal I(\tilde \phi)(I) = \mathcal I(\phi)(I)$ whenever $I\subseteq A$. If $I \nsubseteq A$, then $I$ contains an element of the form $1-a$ with $a\in A$. Lift $\phi(a)$ to a bounded sequence $(b_n)_{n\in \mathbb N} \in \prod_\mathbb{N} B$. As $B$ is stable we may for each $n$ find an isometry $v_n\in \multialg{B}$ such that $\|v_n^\ast b_n v_n\| < 1/n$. Let $v$ be the isometry in $\multialg{B}_\infty$ induced by $(v_n)$. Then $v^\ast \phi(a) v=0$, so $v^\ast\tilde \phi(1-a) v = 1$, and thus $\tilde \phi(1-a)$ is full. Hence $\mathcal I(\tilde \phi)(I) = \multialg{B}_\infty$.
\end{proof}

\begin{lemma}\label{l:O2stableseq}
Let $A$ and $B$ be $C^\ast$-algebras with $A$ separable, and suppose that $B$ is $\mathcal O_2$-stable. For any $\ast$-homomorphism $\phi \colon A \to B_\infty$, the unitisation $\tilde \phi \colon \tilde A \to \multialg{B}_\infty$ is $\mathcal O_2$-stable.
\end{lemma}
\begin{proof}
Lift $\phi$ to a map $(\phi_n)_{n\in \mathbb N} \colon A \to \prod_{\mathbb N} B$, and let $a_1,a_2,\dots \in A$ be dense. As $B$ is $\mathcal O_2$-stable, we may find unital $\ast$-homomorphisms $\psi_n\colon \mathcal O_2 \to \multialg{B}$ such that $\| [\psi_n(s_i), \phi_n(a_j)]\| < 1/n$ for $i=1,2$ and $j=1,\dots, n$. The induced unital $\ast$-homomorphism $\psi \colon \mathcal O_2 \to \multialg{B}_\infty$ commutes with $\phi(A)$ and thus also with $\tilde \phi(\tilde A)$, so $\tilde \phi$ is $\mathcal O_2$-stable.
\end{proof}

For convenience of the reader I include a few lemmas on nuclear maps which should be well-known to experts. These are not stated in their most general forms, but just in a form which is applicable in the proof of Theorem \ref{t:O2embedding} and easy to prove using only well-known results.

\begin{lemma}\label{l:nucmulti}
Let $A$ be a separable, exact $C^\ast$-algebra, and let $B$ be a separable, nuclear, stable $C^\ast$-algebra. 
Then any c.p.~map $A \to \multialg{B}$ is nuclear.
\end{lemma}
\begin{proof}
Let $\eta\colon A \to \multialg{B}$ be a c.p.~map. By normalising and unitising we may assume $A$ and $\eta$ are unital. Pick a unital embedding $A \hookrightarrow \multialg{\mathbb K}$ such that $A \cap \mathbb K = \{0\}$. As $B$ is stable we may find unital embeddings $A \hookrightarrow \multialg{\mathbb K} \xrightarrow{1_{\multialg{B}} \otimes id_{\multialg{\mathbb K}}} \multialg{B \otimes \mathbb K} \cong \multialg{B}$. As $A$ is exact it is nuclearly embeddable, see for instance \cite[Theorem 3.9.1]{BrownOzawa-book-approx}, so the composition of these maps is nuclear. By Kasparov's Weyl--von Neumann--Voiculescu type theorem \cite[Theorem 5]{Kasparov-Stinespring}, this composition approximately 1-dominates $\eta$. Hence $\eta$ is nuclear.
\end{proof}

\begin{lemma}\label{l:nuctensor}
Let $A$, $B$ and $C$ be $C^\ast$-algebras with $C$ nuclear, and let $\phi\colon A \to B$ and $\psi \colon C \to B$ be c.p.~maps with commuting images. If $\phi$ is nuclear then the induced c.p.~map $\phi \times \psi \colon A\otimes C \to B$ is nuclear.
\end{lemma}
\begin{proof}
Let $D$ be a unital $C^\ast$-algebra. As in \cite[Corollary 3.8.8]{BrownOzawa-book-approx} it suffices to show that $id_D \otimes_{\mathrm{alg}} (\phi \times \psi) \colon D \otimes_{\mathrm{alg}} (A\otimes C) \to D \otimes_{\max} B$ is $\|\cdot \|_{\min}$-continuous, i.e.~continuous with respect to the minimal tensor product norm. By nuclearity of $\phi$, the c.p.~map $id_D \otimes \phi \colon D \otimes A \to D \otimes_{\max} B$ is well-defined, and as $1_D \otimes \psi \colon C \to D \otimes_{\max} B$ commutes with the image of $id_D \otimes \phi$ there is an induced c.p.~map $\eta := (id_D \otimes \phi) \times (1_D \otimes \psi) \colon (D\otimes A) \otimes C \to D \otimes_{\max} B$. Under the canonical identification $(D\otimes A) \otimes C \cong D \otimes (A\otimes C)$, $\eta$ is a c.p.~extension of $id_D \otimes_{\mathrm{alg}} (\phi \times \psi)$ so this map is $\| \cdot \|_{\min}$-continuous.
\end{proof}

\begin{lemma}\label{l:nuccrossed}
Let $G$ be a countable, discrete, amenable group, let $A$ be a unital $G$-$C^\ast$-algebra, let $B$ be a $C^\ast$-algebra, and let $\eta \colon A \rtimes G \to B$ be a $\ast$-homomorphism.\footnote{An application of \eqref{eq:bimodule} implies that the result also holds for order zero maps.} If $\eta|_A$ is nuclear, then $\eta$ is nuclear.
\end{lemma}
\begin{proof}
We use \cite[Lemma 4.2.3]{BrownOzawa-book-approx}. For any finite set $F\subset G$ they construct c.p.~maps $\phi_F \colon A \rtimes G \to A \otimes M_F(\mathbb C)$ and $\psi_F \colon A \otimes M_F(\mathbb C) \to A\rtimes G$ such that $\psi_{F_n} \circ \phi_{F_n} \to id$ point-norm for any Følner sequence $(F_n)$ for $G$. Let $\alpha$ denote the $G$-action on $A$ and $\lambda_g$ be the canonical unitaries in $A\rtimes G$ for $g\in G$. The map $\psi_F$ is given by
\[
\psi_F(a \otimes e_{g,h}) = \tfrac{1}{|F|} \alpha_g(a) \lambda_{gh^{-1}} = \tfrac{1}{|F|} \lambda_g a \lambda_h^\ast, \qquad a\in A, g,h\in F.
 \]
Thus, under the canonical identification of $A \otimes M_F(\mathbb C) \cong M_F(A)$, one has
\[
\psi_F = \tfrac{1}{|F|} (\lambda_g)_{g\in F} (-) (\lambda_g)_{g\in F}^\ast \colon M_F(A) \to A\rtimes \Gamma
\]
where $(\lambda_g)_{g\in F}$ is considered as a row vector. The amplification $(\eta|_A)^{(F)} \colon M_F(A) \to M_F(B)$ of $\eta|_A$ is nuclear, so $\eta \circ \psi_F = \tfrac{1}{|F|} (\eta(\lambda_g))_{g\in F} (\eta|_A)^{(F)}(-) (\eta(\lambda_g))^\ast_{g\in F}$ is nuclear. For any Følner sequence $(F_n)$, $\eta \circ \psi_{F_n} \circ \phi_{F_n}$ is thus nuclear and converges point-norm to $\eta$ which is therefore nuclear.
\end{proof}

\begin{proof}[Proof of Theorem \ref{t:O2embedding}]
By Proposition \ref{p:Dstablealgebras} any $\ast$-homomorphism into $B\otimes \mathcal O_2 \otimes \mathbb K$ is strongly $\mathcal O_2$-stable. As the composition of a strongly $\mathcal O_2$-stable $\ast$-homomorphism with any $\ast$-homo\-morphism is again strongly $\mathcal O_2$-stable by Lemma \ref{l:Dcompose}, and as any inclusion $\mathcal O_2\otimes \mathbb K \hookrightarrow \mathcal O_\infty$ induces an isomorphism of ideal lattices 
\[
\mathcal I(B\otimes \mathcal O_2\otimes \mathbb K) \xrightarrow \cong \mathcal I(B\otimes \mathcal O_\infty) \cong \mathcal I(B)
\]
we may assume that $B$ is stable and $\mathcal O_2$-stable. We write $B = B_1 \otimes \mathcal O_2$ with $B_1 \cong B$. Let $\Phi_1 \colon \mathcal I(A) \to \mathcal I(B_1)$ be the $Cu$-morphism induced by $\Phi$ and the obvious identification $\mathcal I(B_1) \cong \mathcal I(B)$, i.e.~$\Phi(I) = \Phi_1(I) \otimes \mathcal O_2$ for all $I \in \mathcal I(A)$. By Corollary \ref{c:multimorphism} we pick a $\ast$-homomorphism $\phi_1 \colon A \to \multialg{B_1}$ such that $\overline{B_1\phi_1(I) B_1} = \Phi_1(I)$ for all $I\in \mathcal I(A)$. Let $\phi_0 = \phi_1 \otimes 1_{\mathcal O_2} \colon A \to \multialg{B_1 \otimes \mathcal O_2} = \multialg{B}$. Clearly $\overline{B \phi_0(I) B} = \Phi(I)$ for $I\in \mathcal I(A)$. By Lemma \ref{l:nucmulti}, $\phi_0$ is nuclear. Let $(b_n)_{n\in \mathbb N}$ be a countable approximate identity (of positive contractions) for $B_1$ which is quasi-central with respect to $\phi_1(A)$, and let $h\in \mathcal O_2$ be a positive element with spectrum $[0,1]$. Let $b = \pi_\infty((b_n)) \in (B_1)_\infty$. As $b\otimes h\in (B_1)_\infty \otimes \mathcal O_2 \subseteq B_\infty$ is a positive contraction which commutes with the image of $\phi_1 \otimes 1_{\mathcal O_2} = \phi_0$, there is a $\ast$-homomorphism $\Psi \colon C_0(0,1) \otimes A \to \multialg{B}_\infty$ given on elementary tensors by
\[
\Psi(f\otimes a) =  f(b \otimes h)\phi_0(a), \qquad f \in C_0(0,1),  a \in A.
\]
The map $\Psi$ clearly factors through $B_\infty$ as each $f(b\otimes h) \in B_\infty$, and $\Psi$ is nuclear by Lemma \ref{l:nuctensor}. Let $\iota \colon B \hookrightarrow B_\infty$ be the canonical inclusion and let $\hat \Phi = \mathcal I(\iota)\circ \Phi$, i.e.~$\hat \Phi(I) = \overline{B_\infty \Phi(I) B_\infty}$.

\textbf{Claim:} $\Psi(f \otimes a)$ is full in $\hat \Phi(\overline{AaA})$ for all positive $a\in A$ and all positive, non-zero $f\in C_0(0,1)$.\footnote{This implies that $\mathcal I(\Psi) = \hat \Phi \circ \mathcal I(\rho_\mu)$, where $\rho_\mu \colon C_0(0,1) \otimes A \to A$ is the right slice map with respect to a faithful state $\mu$ on $C_0(0,1)$.} So fix such $a,f$.

To see that $\Psi(f\otimes a)\in \hat \Phi(\overline{AaA})$, it suffices to check that $\Psi(f \otimes (a-\epsilon)_+) \in \hat \Phi(\overline{AaA})$ for any $\epsilon>0$. Fix such an $\epsilon$. As $B$ is purely infinite, it follows from Proposition \ref{p:idealultra} that 
\begin{equation}\label{eq:B0ideal}
\hat \Phi(\overline{AaA}) = \overline{\bigcup_{J \cpct \Phi(\overline{AaA})} J_\infty}.
\end{equation}
Recall that $\overline{B\phi_0(\overline{A(a-\epsilon)_+A})B} = \Phi(\overline{A(a-\epsilon)_+A})$, and that $\overline{A(a-\epsilon)_+A} \cpct \overline{AaA}$ by Lemma \ref{l:cpct}. As $\Phi$ preserves compact containment it follows that $\overline{B\phi_0(\overline{A(a-\epsilon)_+A})B} \cpct \Phi(\overline{AaA})$. Thus, as $f(b\otimes h) \in B_\infty$, we have
\[
\Psi(f \otimes (a-\epsilon)_+) = f(b \otimes h) \phi_0(a-\epsilon)_+ \in \left( \overline{B \phi_0(\overline{A(a-\epsilon)_+A})B} \right)_\infty \stackrel{\eqref{eq:B0ideal}}{\subseteq} \hat \Phi(\overline{AaA}).
\]
So 
\begin{equation}\label{eq:psiinhat}
\Psi(f\otimes a) \in \hat \Phi(\overline{AaA}).
\end{equation}

Next we show that $\hat \Phi(\overline{AaA}) \subseteq \overline{B_\infty \Psi(f\otimes a) B_\infty}$, which finishes the proof of the claim. As $B_1$ is separable, $\Phi_1(\overline{AaA})$ contains a full, positive element, say $c$.\footnote{Separability of $B_1$ is actually not needed to conclude that $\Phi_1(\overline{AaA})$ has a full element. In fact, an easy consequence of Corollary \ref{c:Cuntzsglemma} is that $Cu$-morphisms map ideals with full elements to ideals with full elements.} In particular, as $\mathcal O_2$ is simple, $f$ is non-zero, and $h$ has spectrum $[0,1]$, $c\otimes f(h)$ is full in $\Phi_1(\overline{AaA}) \otimes \mathcal O_2 = \Phi(\overline{AaA})$ and thus also full in $\hat \Phi(\overline{AaA})$. So it suffices to show that $(c-\epsilon)_+ \otimes f(h)$ is in the ideal generated by $\Psi(f\otimes a)$ for every $\epsilon>0$. Fix such an $\epsilon$.

Recall that $B = B_1 \otimes \mathcal O_2$. Note that $b\otimes 1$ and $1\otimes h$ are in the relative commutant $\multialg{B}_\infty \cap (B_1 \otimes 1_{\mathcal O_2})'$, and that $b\otimes 1 + \mathrm{Ann}(B _1 \otimes 1_{\mathcal O_2})$ is the unit of $(\multialg{B}_\infty \cap (B_1 \otimes 1)')/\mathrm{Ann}(B_1 \otimes 1)$. Thus 
\[
b \otimes h + \mathrm{Ann}(B_1 \otimes 1_{\mathcal O_2}) = 1 \otimes h + \mathrm{Ann}(B_1 \otimes 1_{\mathcal O_2}),
\]
so in particular
\[
f(b \otimes h) + \mathrm{Ann}(B_1 \otimes 1_{\mathcal O_2}) = 1 \otimes f(h) + \mathrm{Ann}(B_1 \otimes 1_{\mathcal O_2}).
\]
Hence, for any $d\in B_1$ we have 
\begin{equation}\label{eq:f(bh)}
f(b\otimes h)(d\otimes 1) = (d \otimes 1) f(b\otimes h) =  (d \otimes 1)(1 \otimes f(h)).
\end{equation}
As $c \in \overline{B_1\phi_1(a)B_1}$ we may find $d_1,\dots,d_n\in B$ such that $(c-\epsilon)_+ = \sum_{k=1}^n d_k^\ast \phi_1(a) d_k$. It follows that
\begin{eqnarray*}
(c-\epsilon)_+  \otimes f(h) &=& \sum_{k=1}^n (d_k^\ast \otimes 1) (\phi_1(a)  \otimes 1)( d_k  \otimes 1)(1 \otimes f(h)) \\
&\stackrel{\eqref{eq:f(bh)}}{=} &  \sum_{k=1}^n (d_k^\ast \otimes 1) \phi_0(a)f(b\otimes h)( d_k \otimes 1) \\
&=& \sum_{k=1}^n (d_k^\ast \otimes 1) \Psi(f\otimes a) ( d_k \otimes 1).
\end{eqnarray*}
As $\epsilon>0$ was arbitrary, it follows that $c\otimes f(h)$ is in the ideal generated by $\Psi(f\otimes a)$, so it follows that 
\begin{equation}\label{eq:hatinpsi}
\hat \Phi(\overline{AaA}) \subseteq \overline{B_\infty \Psi(f\otimes a) B_\infty}.
\end{equation}
Combining this with equation \eqref{eq:psiinhat} proves our claim above.

Let $\alpha$ be an automorphism on $C_0(0,1)$ such that $C_0(0,1) \rtimes_\alpha \mathbb Z \cong C(\mathbb T) \otimes \mathbb K$,\footnote{For instance, the automorphism $\sigma$ on $C_0(\mathbb R)$ which shifts the variable by $1$, satisfies $C_0(\mathbb R)\rtimes_{\sigma} \mathbb Z \cong C(\mathbb T) \otimes \mathbb K$.} let $\beta = \alpha \otimes id_A$ be the induced automorphism on $C_0(0,1) \otimes A$, and let $\tilde \beta$ be the unitisation which is an automorphism on $(C_0(0,1) \otimes A)^\sim$.

For any $a\in A_+$ and any non-zero $f\in C_0(0,1)_+$ it follows from what we showed above that both $\Psi(f\otimes a)$ and $\Psi \circ \beta (f\otimes a) = \Psi(\alpha(f) \otimes a)$ are full in $\hat \Phi(\overline{AaA})$. Hence $\mathcal I(\Psi)$ and $\mathcal I(\Psi \circ \beta)$ agree on all ideals generated by elements of the form $f\otimes a$. As such ideals form a basis for $\mathcal I(C_0(0,1) \otimes A)$, and as $\mathcal I(\Psi)$ and $\mathcal I(\Psi \circ \beta)$ preserve suprema by Lemma \ref{l:idealmorphisms} and Remark \ref{r:sup}, it follows that $\mathcal I(\Psi) = \mathcal I(\Psi \circ \beta)$. 

Consider the unital extension $\tilde \Psi \colon (C_0(0,1)\otimes A)^\sim \to \multialg{B}_\infty$ of $\Psi$. As unitisations of nuclear maps are nuclear (cf.~\cite[Proposition 2.2.4]{BrownOzawa-book-approx}), $\tilde \Psi$ is nuclear. Moreover, as $\Psi$ factors through $B_\infty$, it follows from Lemma \ref{l:O2stableseq} that $\tilde \Psi$ is $\mathcal O_2$-stable. 
Also, as $\Psi$ takes values in $B_\infty$, and as $\mathcal I(\Psi) = \mathcal I(\Psi \circ \beta)$, it follows from Lemma \ref{l:fullmulti} that $\mathcal I(\tilde \Psi) = \mathcal I(\tilde \Psi \circ \tilde \beta)$. 

As $\tilde \Psi$ is unital, nuclear and $\mathcal O_2$-stable, so is $\tilde \Psi \circ \tilde \beta$. Thus by Theorem \ref{t:O2HB}, $\tilde \Psi \circ \tilde \beta$ and $\tilde \Psi$ are approximately unitary equivalent. By Lemma \ref{l:approxinfty} there exists a unitary $u\in \multialg{B}_\infty$ such that $u \tilde \Psi (-) u^\ast = \tilde \Psi \circ \tilde \beta$. Thus there is an induced $\ast$-homomorphism
\[
\psi_1 \colon (C_0(0,1) \otimes A)^\sim \rtimes_{\tilde \beta} \mathbb Z \to \multialg{B}_\infty
\]
given by $\psi_1(x v^n) = \tilde \Psi(x) u^n$ for $x\in (C_0(0,1)\otimes A)^\sim$ and $n\in \mathbb Z$, where $v$ is the canonical unitary in the crossed product. As the restriction of $\psi_1$ to $(C_0(0,1) \otimes A)^\sim$ is $\tilde \Psi$ which is nuclear, it follows from Lemma \ref{l:nuccrossed} that $\psi_1$ is nuclear. As $\psi_1(xv^n) = \Psi(x) v^n \in B_\infty$ for any $x\in C_0(0,1) \otimes A$ and $n\in \mathbb Z$, it follows that $\psi_1$ restricts to a $\ast$-homomorphism 
\[
\psi_0 \colon (C_0(0,1) \otimes A)\rtimes_\beta \mathbb Z \to B_\infty.
\]
As $\psi_0$ is nuclear when considered as a map into $\multialg{B}_\infty$ (as this is just the restriction of $\psi_1$), and as $B_\infty$ is an ideal in $\multialg{B}_\infty$, it follows that $\psi_0$ is nuclear.

Since $\beta = \alpha \otimes id_A$ we have a natural isomorphism
\[
 \theta \colon (C_0(0,1)\rtimes_\alpha \mathbb Z) \otimes A \xrightarrow \cong (C_0(0,1) \otimes A)\rtimes_\beta \mathbb Z.
\]
As $C_0(0,1)\rtimes_\alpha \mathbb Z \cong C(\mathbb T) \otimes \mathbb K$ we may fix a full projection $p\in C_0(0,1) \rtimes_\alpha \mathbb Z$. We get an induced $\ast$-homomorphism 
\[
\psi \colon A \to B_\infty, \qquad \psi(a) = \psi_0(\theta(p \otimes a))
\]
for $a\in A$. As $\psi_0$ is nuclear, so is $\psi$. We will show that $\psi(a)$ is full in $\hat \Phi(\overline{AaA})$ for every positive $a\in A_+$, so fix such an $a$.

Let $w\in \multialg{C_0(0,1)\rtimes_\alpha \mathbb Z}$ be the canonical unitary.  Then
\[
\psi_0(\theta(fw^n \otimes a)) = \Psi(f \otimes a) u^n \stackrel{\eqref{eq:psiinhat}}{\in} \hat \Phi(\overline{AaA}), \qquad f\in C_0(0,1), n \in \mathbb Z.
\]
It follows that
\[
\psi(a) = \psi_0(\theta(p\otimes a)) \in \hat \Phi(\overline{AaA}).
\]
Let $f\in C_0(0,1)$ be positive and non-zero. As $p$ is full in $C_0(0,1) \rtimes_\alpha \mathbb Z$, it follows that
\[
\hat \Phi(\overline{AaA}) \stackrel{\eqref{eq:hatinpsi}}{\subseteq} \overline{B_\infty \Psi(f \otimes a) B_\infty} = \overline{B_\infty \psi_1(\theta(f \otimes a)) B_\infty} \subseteq \overline{B_\infty \psi_1(\theta(p\otimes a)) B_\infty} = \overline{B_\infty \psi(a) B_\infty}.
\]
Thus $\psi(a)$ is full in $\hat \Phi(\overline{AaA})$ for all positive $a\in A$, so $\mathcal I(\psi) = \hat \Phi$.

 We wish to apply Theorem \ref{t:lifthom}, so let $\eta \colon \mathbb N \to \mathbb N$ be a map such that $\lim_{n\to \infty}\eta(n) = \infty$, and let $\eta^\ast \colon \multialg{B}_\infty \to \multialg{B}_\infty$ be the induced $\ast$-homomorphism. Let $I\in \mathcal I(A)$, and let $c\in B_+$ be such that $\hat \Phi(I) = \overline{B_\infty c B_\infty}$. We get
\[
\mathcal I(\eta^\ast)(\hat \Phi(I)) = \overline{\multialg{B}_\infty \eta^\ast(c) \multialg{B}_\infty} = \overline{B_\infty c B_\infty} = \hat \Phi(I).
\]
It follows, as $\mathcal I(\psi) = \hat \Phi$, that 
\[
\mathcal I(\eta^\ast \circ \psi) \stackrel{\textrm{Prop }\ref{p:ozfunctorial}}{=} \mathcal I(\eta^\ast) \circ \mathcal I(\psi) = \mathcal I(\eta^\ast) \circ \hat \Phi = \hat \Phi = \mathcal I(\psi).
\]
Let $\tilde \psi \colon \tilde A \to \multialg{B}_\infty$ be the unital extension of $\psi$ which is nuclear. Then $\eta^\ast \circ \tilde \psi$ is the unital extension of $\eta^\ast \circ \psi$ which is also nuclear. It follows from Lemma \ref{l:O2stableseq} that both $\tilde \psi$ and $\eta^\ast \circ \tilde \psi$ are $\mathcal O_2$-stable, and by Lemma \ref{l:fullmulti}, $\mathcal I(\tilde \psi) = \mathcal I(\eta^\ast \circ \tilde \psi)$. So by our uniqueness result Theorem \ref{t:O2HB}, $\tilde \psi$ and $\eta^\ast \circ \tilde \psi$ are approximately unitary equivalent. By Theorem \ref{t:lifthom} it follows that there is a $\ast$-homomorphism $\tilde \phi \colon \tilde A \to \multialg{B}$ such that $\tilde \phi$ and $\tilde \psi$ are unitary equivalent as maps into $\multialg{B}_\infty$. 

It follows that 
\[
\mathcal I(\iota) \circ \mathcal I(\phi) = \mathcal I(\iota \circ \phi) = \mathcal I(\psi) = \hat \Phi = \mathcal I(\iota) \circ \Phi,
\]
where $\iota \colon B \hookrightarrow B_\infty$ is the constant inclusion. Obviously the map $\mathcal I(\iota)$ is injective, so it follows that $\mathcal I(\phi) = \Phi$, thus finishing the proof.
\end{proof}


\subsection{Applications}

Let $\Hom_{\mathcal O_2}(A,B)$ and $\Hom_{s\mathcal O_2}(A,B)$ denote the sets of $\mathcal O_2$-stable and strongly $\mathcal O_2$-stable $\ast$-homomorphisms from $A$ to $B$ respectively (Definition \ref{d:Dstable}), and let $\sim_\mathrm{aMvN}$ and $\sim_\mathrm{asMvN}$ denote approximate and asymptotic Murray--von Neumann equivalence respectively (Definition \ref{d:MvN}). Let $\mathbf{Cu}(\mathcal I(A), \mathcal I(B))$ denote the semigroup of $Cu$-morphisms from $\mathcal I(A) \to \mathcal I(B)$ (Definition \ref{d:Cumorphism}). 

\begin{corollary}\label{c:classhom}
Let $A$ be a separable, exact $C^\ast$-algebra and let $B$ be a separable, nuclear, $\mathcal O_\infty$-stable $C^\ast$-algebra. Then the natural maps
\[
\Hom_{s\mathcal O_2}(A,B)/\!\!\sim_\mathrm{asMvN} \; \to \Hom_{\mathcal O_2}(A,B)/\!\!\sim_\mathrm{aMvN} \; \to \mathbf{Cu}(\mathcal I(A), \mathcal I(B)) \\
\]
are both bijective.

In particular, if either $A$ or $B$ is $\mathcal O_2$-stable, then the natural maps
\[
\Hom(A,B)/\!\!\sim_\mathrm{asMvN} \; \to \Hom(A,B)/\!\!\sim_\mathrm{aMvN} \; \to \mathbf{Cu}(\mathcal I(A), \mathcal I(B)) \\
\]
are both bijective.

Moreover, if $B$ is also stable, then $\sim_\mathrm{asMvN}$ and $\sim_\mathrm{aMvN}$ may be replaced with asymptotic and approximate unitary equivalence respectively (with unitaries in $\multialg{B}$).
\end{corollary}
\begin{proof}
Injectivity is Theorem \ref{t:O2HB}, and surjectivity is Theorem \ref{t:O2embedding}. The ``in particular'' part follows from Proposition \ref{p:Dstablealgebras}.
\end{proof}

Similarly, suppose $A$ and $B$ are unital. Let $\Hom_{((s)\mathcal O_2)}(A,B)_1$ denote the set of all ((strongly) $\mathcal O_2$-stable) unital $\ast$-homomorphisms, let $\sim_\mathrm{asu}$ and $\sim_\mathrm{au}$ denote asymptotic and approximate unitary equivalence respectively, and let $\mathbf{Cu}(\mathcal I(A) , \mathcal I(B))_1$ denote the set of $Cu$-morphisms $\Phi \in \mathbf{Cu}(\mathcal I(A), \mathcal I(B))$ such that $\Phi(A) = B$.

\begin{corollary}\label{c:classhom1}
Let $A$ be a separable, exact, unital $C^\ast$-algebra and let $B$ be a separable, nuclear, unital, $\mathcal O_\infty$-stable $C^\ast$-algebra such that $[1_B]_0 = 0 \in K_0(B)$. Then the natural maps
\[
\Hom_{s\mathcal O_2}(A,B)_1/\!\!\sim_\mathrm{asu} \; \to \Hom_{\mathcal O_2}(A,B)_1/\!\!\sim_\mathrm{au} \; \to \mathbf{Cu}(\mathcal I(A), \mathcal I(B))_1 \\
\]
are both bijective.

In particular, if either $A$ or $B$ is $\mathcal O_2$-stable, then the natural maps
\[
\Hom(A,B)_1/\!\!\sim_\mathrm{asu} \; \to \Hom(A,B)_1/\!\!\sim_\mathrm{au} \; \to \mathbf{Cu}(\mathcal I(A), \mathcal I(B))_1 \\
\]
are both bijective.
\end{corollary}
\begin{proof}
Injectivity again follows from Theorem \ref{t:O2HB}.

For surjectivity let $\Phi \in \mathbf{Cu}(\mathcal I(A), \mathcal I(B))_1$. Use Theorem \ref{t:O2embedding} to construct a strongly $\mathcal O_2$-stable $\ast$-homo\-morphism $\phi_0 \colon A \to B$ such that $\mathcal I(\phi_0) = \Phi$. As $\Phi(A) = B$, $p:=\phi_0(1_A)$ is a full projection in $B$. As $\phi_0$ is $\mathcal O_2$-stable, $\mathcal O_2$ embeds unitally in $(pBp)_\infty \cap \phi_0(A)'$. By semiprojectivity of $\mathcal O_2$, this embedding lifts to a unital $\ast$-homomorphism $\mathcal O_2 \to (pBp)_\infty$, so $\mathcal O_2$ embeds unitally in $pBp$. Hence $p\in B$ is a full, properly infinite projection with $[p]_0 =0 \in K_0(B)$. Thus, a result of Cuntz \cite{Cuntz-K-theoryI} implies that $1_B$ and $\phi_0(1_A)$ are Murray--von Neumann equivalent. Let $v\in B$ be an isometry with $v v^\ast = \phi_0(1_A)$. Then $\phi := v \phi_0(-) v^\ast \colon A \to B$ is unital, strongly $\mathcal O_2$-stable and satisfies $\mathcal I(\phi) = \Phi$.

``In particular'' is again Proposition \ref{p:Dstablealgebras}.
\end{proof}

The main application is the following strong classification result which was originally proved by Kirchberg, see \cite{Kirchberg-non-simple}.

\begin{theorem}\label{t:O2class}
Let $A$ and $B$ be separable, nuclear, $\mathcal O_2$-stable $C^\ast$-algebras which are either both stable or both unital. 
\begin{itemize}
\item[$(a)$] If $\Phi \colon \mathcal I(A) \xrightarrow \cong \mathcal I(B)$ is an order isomorphism, then there exists an isomorphism $\phi \colon A \xrightarrow \cong B$ such that $\mathcal I(\phi) = \Phi$, i.e.~such that $\phi(I) = \Phi(I)$ for all $I\in \mathcal I(A)$.
\item[$(b)$] If $f \colon \Prim A \xrightarrow \cong \Prim B$ is a homeomorphism, then there exists an isomorphism $\phi \colon A \xrightarrow \cong B$ such that $\phi(I) = f(I)$ for all $I \in \Prim A$.
\end{itemize}
\end{theorem}
\begin{proof}
$(a)$: The stable (resp.~unital) case follows from Corollary \ref{c:classhom} (resp.~Corollary \ref{c:classhom1}) and an intertwining argument a la Elliott, see \cite[Corollary 2.3.4]{Rordam-book-classification}.

$(b)$: By \cite[Theorem 4.1.3]{Pedersen-book-automorphism} there is an induced order isomorphism $\Phi \colon \mathcal I(A) \xrightarrow \cong \mathcal I(B)$ such that $\Phi(I) = f(I)$ for all $I\in \Prim A$. Hence the result follows from part $(a)$. 
\end{proof}

\begin{corollary}
Let $A$ and $B$ be separable, nuclear $C^\ast$-algebras which are either both stable or both unital. The following are equivalent.
\begin{itemize}
\item[$(i)$] $A\otimes \mathcal O_2$ and $B\otimes \mathcal O_2$ are isomorphic,
\item[$(ii)$] $\mathcal I(A)$ and $\mathcal I(B)$ are order isomorphic,
\item[$(iii)$] $\Prim A$ and $\Prim B$ are homeomorphic.
\end{itemize}
\end{corollary}
\begin{proof}
$(ii) \Leftrightarrow (iii)$: For any separable $C^\ast$-algebra $C$, $\Prim C$ is sober\footnote{A topological space $X$ is \emph{sober} (or spectral, or point complete) if the prime open subsets $U$ are exactly the sets of the form $X\setminus\overline{\{x\}}$ for a unique $x\in X$. An open set $U$ is prime if whenever $V,W$ are open and $V \cap W \subseteq U$, then $V\subseteq U$ or $W\subseteq U$.} by \cite[Proposition 4.3.6]{Pedersen-book-automorphism}. It is well-known (and easily verified) that two sober spaces are homeomorphic
if and only if their lattices of open subsets are order isomorphic. Since $\mathcal I(C)$ is order
isomorphic to the lattice of open subsets of $\Prim C$ by \cite[Theorem 4.1.3]{Pedersen-book-automorphism}, we get that $\mathcal I(A) \cong \mathcal I(B)$ if and only if $\Prim A \cong \Prim B$.

$(i) \Rightarrow (ii)$: This follows since $I \mapsto I \otimes \mathcal O_2$ is an isomorphism $\mathcal I(C) \cong \mathcal I(C\otimes \mathcal O_2)$.

$(ii) \Rightarrow (i)$: This follows from Theorem \ref{t:O2class}.
\end{proof}

\newcommand{\etalchar}[1]{$^{#1}$}


\begin{thebibliography}{EGLN15}

\bibitem[APT11]{AraPereraToms-K-theoryclass}
P. Ara, F. Perera, and A.~S. Toms.
\newblock {$K$}-theory for operator algebras. {C}lassification of
  {$C^*$}-algebras.
\newblock In {\em Aspects of operator algebras and applications}, volume 534 of
  {\em Contemp. Math.}, pages 1--71. Amer. Math. Soc., Providence, RI, 2011.

\bibitem[BBS{\etalchar{+}}15]{BBSTWW-2coloured}
J. Bosa, N.~P. Brown, Y. Sato, A. Tikuisis, S. White, and
  W. Winter.
\newblock Covering dimension of {C}$^*$-algebras and $2$-coloured
  classification.
\newblock {\em Mem. Amer. Math. Soc.}, to appear.

\bibitem[BE78]{BratteliElliott-structurespacesII}
O. Bratteli and G.~A. Elliott.
\newblock Structure spaces of approximately finite-dimensional {$C^{\ast}
  $}-algebras. {II}.
\newblock {\em J. Funct. Anal.}, 30(1):74--82, 1978.

\bibitem[BCS15]{BrownClarkSierakowski-purelyinfgroupoids}
J. Brown, L.~O. Clark, and A. Sierakowski.
\newblock Purely infinite {$C^\ast$}-algebras associated to \'etale groupoids.
\newblock {\em Ergodic Theory Dynam. Systems}, 35(8):2397--2411, 2015.

\bibitem[Bla96]{Blanchard-deformations}
E. Blanchard.
\newblock D\'eformations de {$C^*$}-alg\`ebres de {H}opf.
\newblock {\em Bull. Soc. Math. France}, 124(1):141--215, 1996.

\bibitem[BK04]{BlanchardKirchberg-Hausdorff}
E. Blanchard and E. Kirchberg.
\newblock Non-simple purely infinite {$C^*$}-algebras: the {H}ausdorff case.
\newblock {\em J. Funct. Anal.}, 207(2):461--513, 2004.

\bibitem[BO08]{BrownOzawa-book-approx}
N.~P. Brown and N. Ozawa.
\newblock {\em {$C^*$}-algebras and finite-dimensional approximations},
  volume~88 of {\em Graduate Studies in Mathematics}.
\newblock American Mathematical Society, Providence, RI, 2008.

\bibitem[Con73]{Connes-classtypeIII}
A. Connes.
\newblock Une classification des facteurs de type {${\rm III}$}.
\newblock {\em Ann. Sci. \'Ecole Norm. Sup. (4)}, 6:133--252, 1973.

\bibitem[CEI08]{CowardElliottIvanescu-Cuntzsemigroupinv}
K.~T. Coward, G.~A. Elliott, and C. Ivanescu.
\newblock The {C}untz semigroup as an invariant for {$C^*$}-algebras.
\newblock {\em J. Reine Angew. Math.}, 623:161--193, 2008.

\bibitem[Cun81]{Cuntz-K-theoryI}
J. Cuntz.
\newblock {$K$}-theory for certain {$C^{\ast} $}-algebras.
\newblock {\em Ann. of Math. (2)}, 113(1):181--197, 1981.

\bibitem[Dad05]{Dadarlat-KKtop}
M. Dadarlat.
\newblock On the topology of the {K}asparov groups and its applications.
\newblock {\em J. Funct. Anal.}, 228(2):394--418, 2005.

\bibitem[Dad07]{Dadarlat-htpyKirchbergalg}
M. Dadarlat.
\newblock The homotopy groups of the automorphism group of {K}irchberg
  algebras.
\newblock {\em J. Noncommut. Geom.}, 1(1):113--139, 2007.

\bibitem[Dad09]{Dadarlat-fiberwiseKK}
M. Dadarlat.
\newblock Fiberwise {$KK$}-equivalence of continuous fields of
  {$C^*$}-algebras.
\newblock {\em J. K-Theory}, 3(2):205--219, 2009.

\bibitem[DL96]{DadarlatLoring-UMCT}
M. Dadarlat and T.~A. Loring.
\newblock A universal multicoefficient theorem for the {K}asparov groups.
\newblock {\em Duke Math. J.}, 84(2):355--377, 1996.

\bibitem[DW09]{DadarlatWinter-KKssa}
M. Dadarlat and W. Winter.
\newblock On the {$KK$}-theory of strongly self-absorbing {$C^*$}-algebras.
\newblock {\em Math. Scand.}, 104(1):95--107, 2009.

\bibitem[Ell93]{Elliott-classrr0}
G.~A. Elliott.
\newblock On the classification of {$C^*$}-algebras of real rank zero.
\newblock {\em J. Reine Angew. Math.}, 443:179--219, 1993.

\bibitem[EGLN15]{ElliottGongLinNiu-classfindec}
G.~A. Elliott, G. Gong, H. Lin, and Z. Niu.
\newblock On the classification of simple {$C^\ast$}-algebras with finite
  decomposition rank, {II}.
\newblock {\em ArXiv:1507.03437v2}, 2015.

\bibitem[EK01]{ElliottKucerovsky-extensions}
G.~A. Elliott and D. Kucerovsky.
\newblock An abstract {V}oiculescu-{B}rown-{D}ouglas-{F}illmore absorption
  theorem.
\newblock {\em Pacific J. Math.}, 198(2):385--409, 2001.

\bibitem[Gab16]{Gabe-cplifting}
J. Gabe.
\newblock Lifting theorems for completely positive maps.
\newblock {\em Preprint, ArXiv:1508.00389v3}, 2016.

\bibitem[GLN15]{GongLinNiu-classZ-stable}
G. Gong, H. Lin, and Z. Niu.
\newblock Classification of finite simple amenable ${\mathcal Z}$-stable
  {$C^*$}-algebras.
\newblock {\em ArXiv:1501.00135v4}, 2015.

\bibitem[HK05]{HarnischKirchberg-primitive}
H. Harnisch and E. Kirchberg.
\newblock The Inverse Problem for Primitive Ideal Spaces.
\newblock Preprintreihe SFB 478 - Geometrische Strukturen in der Mathematik.
  SFB478, 2005.

\bibitem[Kas80]{Kasparov-Stinespring}
G.~G. Kasparov.
\newblock Hilbert {$C^{\ast} $}-modules: theorems of {S}tinespring and
  {V}oiculescu.
\newblock {\em J. Operator Theory}, 4(1):133--150, 1980.

\bibitem[Kir94]{Kirchberg-simple}
E. Kirchberg.
\newblock The classification of purely infinite {$C^\ast$}-algebras using
  {K}asparov's theory.
\newblock 1994.

\bibitem[Kir95]{Kirchberg-normalizer}
E. Kirchberg.
\newblock On restricted perturbations in inverse images and a description of
  normalizer algebras in {$C^*$}-algebras.
\newblock {\em J. Funct. Anal.}, 129(1):1--34, 1995.

\bibitem[Kir00]{Kirchberg-non-simple}
E. Kirchberg.
\newblock Das nicht-kommutative {M}ichael-{A}uswahlprinzip und die
  {K}lassifikation nicht-einfacher {A}lgebren.
\newblock In {\em {$C^*$}-algebras ({M}\"unster, 1999)}, pages 92--141.
  Springer, Berlin, 2000.

\bibitem[Kir06]{Kirchberg-Abel}
E. Kirchberg.
\newblock Central sequences in {$C^*$}-algebras and strongly purely infinite
  algebras.
\newblock In {\em Operator {A}lgebras: {T}he {A}bel {S}ymposium 2004}, volume~1
  of {\em Abel Symp.}, pages 175--231. Springer, Berlin, 2006.

\bibitem[Kir]{Kirchberg-book}
E. Kirchberg.
\newblock {\em The {C}lassification of {P}urely {I}nfinite {C}$^*$-{A}lgebras {U}sing {K}asparov’s {T}heory}
\newblock Book in preparation.

\bibitem[KP00]{KirchbergPhillips-embedding}
E. Kirchberg and N.~C. Phillips.
\newblock Embedding of exact {$C^*$}-algebras in the {C}untz algebra {$\mathcal
  O_2$}.
\newblock {\em J. Reine Angew. Math.}, 525:17--53, 2000.

\bibitem[KR00]{KirchbergRordam-purelyinf}
E. Kirchberg and M. R{\o}rdam.
\newblock Non-simple purely infinite {$C^\ast$}-algebras.
\newblock {\em Amer. J. Math.}, 122(3):637--666, 2000.

\bibitem[KR02]{KirchbergRordam-absorbingOinfty}
E. Kirchberg and M. R{\o}rdam.
\newblock Infinite non-simple {$C^*$}-algebras: absorbing the {C}untz algebras
  {$\mathcal{O}_\infty$}.
\newblock {\em Adv. Math.}, 167(2):195--264, 2002.

\bibitem[KR05]{KirchbergRordam-zero}
E. Kirchberg and M. R{\o}rdam.
\newblock Purely infinite {$C^*$}-algebras: ideal-preserving zero homotopies.
\newblock {\em Geom. Funct. Anal.}, 15(2):377--415, 2005.

\bibitem[KS17a]{KirchbergSierakowski-spicrossed}
E. Kirchberg and A. Sierakowski.
\newblock Strong pure infiniteness of crossed products.
\newblock {\em Ergodic Theory and Dynamical Systems}, page 1–24, 2017.

\bibitem[KS17b]{KwasniewskiSzymanski-pureinffell}
B.~K. Kwa\'sniewski and W. Szyma\'nski.
\newblock Pure infiniteness and ideal structure of {$C^*$}-algebras associated
  to {F}ell bundles.
\newblock {\em J. Math. Anal. Appl.}, 445(1):898--943, 2017.

\bibitem[Lin02]{Lin-stableapproxuniqueness}
H. Lin.
\newblock Stable approximate unitary equivalence of homomorphisms.
\newblock {\em J. Operator Theory}, 47(2):343--378, 2002.

\bibitem[MS14]{MatuiSato-decrankUHF}
H. Matui and Y. Sato.
\newblock Decomposition rank of {UHF}-absorbing {$\mathrm{C}^*$}-algebras.
\newblock {\em Duke Math. J.}, 163(14):2687--2708, 2014.

\bibitem[Mic56]{Michael-selection}
E. Michael.
\newblock Continuous selections. {I}.
\newblock {\em Ann. of Math. (2)}, 63:361--382, 1956.

\bibitem[Mic66]{Michael-aselectionthm}
E. Michael.
\newblock A selection theorem.
\newblock {\em Proc. Amer. Math. Soc.}, 17:1404--1406, 1966.

\bibitem[Ped79]{Pedersen-book-automorphism}
G.~K. Pedersen.
\newblock {\em {$C^{\ast} $}-algebras and their automorphism groups}, volume~14
  of {\em London Mathematical Society Monographs}.
\newblock Academic Press Inc. [Harcourt Brace Jovanovich Publishers], London,
  1979.

\bibitem[Phi00]{Phillips-classification}
N.~C. Phillips.
\newblock A classification theorem for nuclear purely infinite simple
  {$C^*$}-algebras.
\newblock {\em Doc. Math.}, 5:49--114 (electronic), 2000.

\bibitem[R{\o}r94]{Rordam-O2tensorO2}
M. R{\o}rdam.
\newblock A short proof of {E}lliott's theorem: {${\mathcal
  O}_2\otimes{\mathcal O}_2\cong{\mathcal O}_2$}.
\newblock {\em C. R. Math. Rep. Acad. Sci. Canada}, 16(1):31--36, 1994.

\bibitem[R{\o}r95]{Rordam-classsimple}
M. R{\o}rdam.
\newblock Classification of certain infinite simple {$C^*$}-algebras.
\newblock {\em J. Funct. Anal.}, 131(2):415--458, 1995.

\bibitem[R{\o}r02]{Rordam-book-classification}
M. R{\o}rdam.
\newblock Classification of nuclear, simple {$C^*$}-algebras.
\newblock In {\em Classification of nuclear {$C^*$}-algebras. {E}ntropy in
  operator algebras}, volume 126 of {\em Encyclopaedia Math. Sci.}, pages
  1--145. Springer, Berlin, 2002.

\bibitem[RS87]{RosenbergSchochet-UCT}
J. Rosenberg and C. Schochet.
\newblock The {K}\"unneth theorem and the universal coefficient theorem for
  {K}asparov's generalized {$K$}-functor.
\newblock {\em Duke Math. J.}, 55(2):431--474, 1987.

\bibitem[TWW15]{TikuisisWhiteWinter-QDnuc}
A. Tikuisis, S. White, and W. Winter.
\newblock Quasidiagonality of nuclear {$C^\ast$}-algebras.
\newblock {\em  Ann. of Math. (2)}, 185:229--284, 2017.

\bibitem[TW07]{TomsWinter-ssa}
A.~S. Toms and W. Winter.
\newblock Strongly self-absorbing {$C^*$}-algebras.
\newblock {\em Trans. Amer. Math. Soc.}, 359(8):3999--4029, 2007.

\bibitem[WZ09]{WinterZacharias-orderzero}
W. Winter and J. Zacharias.
\newblock Completely positive maps of order zero.
\newblock {\em M\"unster J. Math.}, 2:311--324, 2009.

\end{thebibliography}
\end{document}